\documentclass[reqno,10pt,centertags,draft]{amsart}
\usepackage{amsmath,amsthm,amscd,amssymb,latexsym,upref}
\date{\today}
\usepackage{showkeys}

\newcommand{\bbN}{{\mathbb{N}}}
\newcommand{\bbR}{{\mathbb{R}}}

\newcommand{\bbZ}{{\mathbb{Z}}}
\newcommand{\bbC}{{\mathbb{C}}}

\newcommand{\cB}{{\mathcal B}}

\newcommand{\cD}{{\mathcal D}}
\newcommand{\cE}{{\mathcal E}}

\newcommand{\cH}{{\mathcal H}}

\newcommand{\cK}{{\mathcal K}}

\newcommand{\cN}{{\mathcal N}}

\newcommand{\cR}{{\mathcal R}}
\newcommand{\cS}{{\mathcal S}}

\newcommand{\cX}{{\mathcal X}}

\newcommand{\no}{\notag}
\newcommand{\lb}{\label}
\newcommand{\f}{\frac}

\newcommand{\ol}{\overline}
\newcommand{\ti}{\tilde}
\newcommand{\wti}{\widetilde}

\newcommand{\Oh}{O}

\newcommand{\loc}{\text{\rm{loc}}}

\newcommand{\ran}{\text{\rm{ran}}}

\newcommand{\dom}{\text{\rm{dom}}}

\newcommand{\dist}{\text{\rm{dist}}}

\newcommand{\bi}{\bibitem}
\newcommand{\hatt}{\widehat}
\newcommand{\beq}{\begin{equation}}
\newcommand{\eeq}{\end{equation}}
\newcommand{\ba}{\begin{align}}
\newcommand{\ea}{\end{align}}

\newcommand{\tr}{\text{\rm{tr}}}

\newcommand{\abs}[1]{\lvert#1\rvert}

\renewcommand{\Im}{\text{\rm Im}}
\renewcommand{\ln}{\text{\rm ln}}

\DeclareMathOperator{\sgn}{sgn}

\DeclareMathOperator*{\slim}{s-lim}

\newcommand{\norm}[1]{\left\Vert#1\right\Vert}

\newcommand{\ha}[1]{\widehat{#1}}
\newcommand{\Om}{\Omega}
\newcommand{\dOm}{{\partial\Omega}}
\newcommand{\si}{\sigma}

\newcommand{\la}{\lambda}

\newcommand{\ga}{\gamma}

\newcommand{\eps}{\varepsilon}

\newcommand{\LOm}{L^2(\Om;d^nx)}
\newcommand{\LdOm}{L^2(\dOm;d^{n-1}\si)}

\allowdisplaybreaks \numberwithin{equation}{section}

\newtheorem{theorem}{Theorem}[section]

\newtheorem{lemma}[theorem]{Lemma}
\newtheorem{corollary}[theorem]{Corollary}
\newtheorem{hypothesis}[theorem]{Hypothesis}
\theoremstyle{definition}

\newtheorem{remark}[theorem]{Remark}

\begin{document}

\title[Variations on a Theme of Jost and Pais]{Variations on a Theme of Jost and Pais}
\author[F.\ Gesztesy, M.\ Mitrea and M.\ Zinchenko]{Fritz Gesztesy, Marius Mitrea, and Maxim Zinchenko}
\address{Department of Mathematics,
University of Missouri, Columbia, MO 65211, USA}
\email{fritz@math.missouri.edu}
\urladdr{http://www.math.missouri.edu/personnel/faculty/gesztesyf.html}
\address{Department of Mathematics, University of
Missouri, Columbia, MO 65211, USA}
\email{marius@math.missouri.edu}
\urladdr{http://www.math.missouri.edu/personnel/faculty/mitream.html}
\address{Department of Mathematics,
California Institute of Technology, Pasadena, CA 91125, USA}
\email{maxim@caltech.edu}
\urladdr{http://math.caltech.edu/$\sim$maxim}
\thanks{Based upon work partially supported by the US National Science
Foundation under Grant Nos.\ DMS-0405526 and DMS-0400639,
FRG-0456306.}
\thanks{To appear in {\it J. Funct. Anal.}}
\date{\today}

\begin{abstract}
We explore the extent to which a variant of a celebrated formula due to 
Jost and Pais, which reduces the Fredholm perturbation
determinant associated with the Schr\"odinger operator on a half-line
to a simple Wronski determinant of appropriate distributional solutions
of the underlying Schr\"odinger equation, generalizes to higher dimensions.
In this multi-dimensional extension the
half-line is replaced by an open set $\Om\subset\bbR^n$,
$n\in\bbN$, $n\geq 2$, where $\Omega$ has a compact, nonempty
boundary $\partial\Om$ satisfying certain regularity conditions. Our
variant involves ratios of perturbation
determinants corresponding to Dirichlet and Neumann boundary
conditions on $\partial\Om$ and invokes the corresponding
Dirichlet-to-Neumann map. As a result, we succeed in reducing a certain ratio
of modified Fredholm perturbation determinants associated with operators in
$L^2(\Om; d^n x)$, $n\in\bbN$, to modified Fredholm determinants associated with operators in $L^2(\partial\Om; d^{n-1}\sigma)$, $n\geq 2$. 

Applications involving the Birman--Schwinger principle and eigenvalue counting functions are discussed.
\end{abstract}

\maketitle

\section{Introduction}\label{s1}

To illustrate the reason behind the title of this paper, we briefly recall a
celebrated result of Jost and Pais \cite{JP51}, who proved  in 1951 a spectacular
reduction of the Fredholm determinant associated with the Birman--Schwinger kernel of a one-dimensional Schr\"odinger operator on a half-line, to a  simple Wronski determinant of distributional solutions of the underlying
Schr\"odinger equation. This Wronski determinant also equals the so-called Jost function of the corresponding half-line Schr\"odinger operator. In this paper we prove a certain multi-dimensional variant of this result.

To describe the result due to Jost and Pais \cite{JP51}, we need a few preparations.
Denoting by $H_{0,+}^D$ and $H_{0,+}^N$ the one-dimensional Dirichlet and Neumann Laplacians  in $L^2((0,\infty);dx)$, and assuming
\begin{equation}
V\in L^1((0,\infty);dx),   \lb{1.1}
\end{equation}
we introduce the perturbed Schr\"odinger operators
$H_{+}^D$ and $H_{+}^N$ in $L^2((0,\infty);dx)$ by
\begin{align}
&H_{+}^Df=-f''+Vf,  \no \\
&f\in \dom\big(H_{+}^D\big)=\{g\in L^2((0,\infty); dx) \,|\, g,g'
\in AC([0,R])
\text{ for all $R>0$}, \\
& \hspace*{4.95cm} g(0)=0, \, (-g''+Vg)\in L^2((0,\infty); dx)\}, \no \\
&H_{+}^Nf=-f''+Vf,  \no \\
&f\in \dom\big(H_{+}^N\big)=\{g\in L^2((0,\infty); dx) \,|\, g,g'
\in AC([0,R])
\text{ for all $R>0$}, \\
& \hspace*{4.85cm} g'(0)=0, \, (-g''+Vg)\in L^2((0,\infty); dx)\}. \no
\end{align}
Thus, $H_{+}^D$ and $H_{+}^N$ are self-adjoint if and only if $V$ is
real-valued, but since the latter restriction plays no special role in our results, we
will not assume real-valuedness of $V$ throughout this paper.

A fundamental system of solutions $\phi_+^D(z,\cdot)$,
$\theta_+^D(z,\cdot)$, and the Jost solution $f_+(z,\cdot)$ of
\begin{equation}
-\psi''(z,x)+V\psi(z,x)=z\psi(z,x), \quad z\in\bbC\backslash\{0\}, \;
x\geq 0,   \lb{1.4}
\end{equation}
are then introduced via the standard Volterra integral equations
\begin{align}
\phi_+^D(z,x)&=z^{-1/2}\sin(z^{1/2}x)+\int_0^x dx' \, z^{-1/2}\sin(z^{1/2}(x-x'))
V(x')\phi_+^D(z,x'), \\
\theta_+^D(z,x)&=\cos(z^{1/2}x)+\int_0^x dx' \, z^{-1/2}\sin(z^{1/2}(x-x'))
V(x')\theta_+^D (z,x'), \\
f_+(z,x)&=e^{iz^{1/2}x}-\int_x^\infty dx' \,
z^{-1/2}\sin(z^{1/2}(x-x')) V(x')f_+(z,x'),  \lb{1.7} \\
&\hspace*{3.85cm}  z\in\bbC\backslash\{0\}, \; \Im(z^{1/2})\geq 0, \;
x\geq 0.  \no
\end{align}

In addition, we introduce
\begin{equation}
u=\exp(i\arg(V))\abs{V}^{1/2}, \quad v=\abs{V}^{1/2}, \, \text{ so
that } \, V=u\, v,
\end{equation}
and denote by $I_+$ the identity operator in $L^2((0,\infty); dx)$. Moreover,
we denote by
\begin{equation}
W(f,g)(x)=f(x)g'(x)-f'(x)g(x), \quad x \geq 0,
\end{equation}
the Wronskian of $f$ and $g$, where $f,g \in C^1([0,\infty))$. We
also use the standard convention to abbreviate (with a slight abuse of notation)
the operator of multiplication in $L^2((0,\infty);dx)$ by an element $f\in
L^1_{\loc}((0,\infty);dx)$ (and similarly in the higher-dimensional
context later) by the same symbol $f$ (rather than $M_f$, etc.). For additional
notational conventions we refer to the paragraph at the end of this introduction.

Then, the following results hold:

\begin{theorem} \lb{t1.1}
Assume $V\in L^1((0,\infty);dx)$ and let $z\in\bbC\backslash [0,\infty)$
with $\Im(z^{1/2})>0$. Then,
\begin{equation}
\ol{u\big(H_{0,+}^D-z I_+\big)^{-1}v}, \,
\ol{u\big(H_{0,+}^N-z I_+\big)^{-1}v} \in \cB_1(L^2((0,\infty);dx))
\end{equation}
and
\begin{align}
\det\Big(I_+ +\ol{u\big(H_{0, +}^D-z I_+\big)^{-1}v}\,\Big) &=
1+z^{-1/2}\int_0^\infty dx\, \sin(z^{1/2}x)V(x)f_+(z,x)   \no \\
&= W(f_+(z,\cdot),\phi_+^D(z,\cdot)) = f_+(z,0),    \lb{1.11}  \\
\det\Big(I_+ +\ol{u\big(H_{0, +}^N-z I_+\big)^{-1}v}\,\Big)
&= 1+ i z^{-1/2} \int_0^\infty
dx\, \cos(z^{1/2}x)V(x)f_+(z,x) \no  \\
&= - \frac{W(f_+(z,\cdot),\theta_+^D (z,\cdot))}{i z^{1/2}} =
\frac{f_+'(z,0)}{i z^{1/2}}.    \lb{1.12}
\end{align}
\end{theorem}

Equation \eqref{1.11} is the modern formulation of the classical result due to Jost and
Pais \cite{JP51} (cf.\ the detailed discussion in \cite{GM03}). Performing calculations similar to Section 4 in \cite{GM03} for the pair of operators $H_{0,+}^N$ and $H_+^N$, one obtains the analogous result \eqref{1.12}. For similar considerations in the context of finite interval problems, we refer to Dreyfus and Dym \cite{DD78} and Levit and Smilansky \cite{LS77}.

We emphasize that \eqref{1.11} and \eqref{1.12} exhibit the
remarkable fact that the Fredholm determinant associated with trace
class operators in the infinite-dimensional space $L^2((0,\infty);
dx)$ is reduced to a simple Wronski determinant of $\bbC$-valued
distributional solutions of \eqref{1.4}. This fact goes back to Jost
and Pais \cite{JP51} (see also \cite{GM03}, \cite{Ne72}, \cite{Ne80},
\cite[Sect.\ 12.1.2]{Ne02}, \cite{Si00}, \cite[Proposition 5.7]{Si05},
and the extensive literature cited in these references).
The principal aim of this paper is to explore the extent to which
this fact may generalize to higher dimensions $n\in\bbN$, $n\geq 2$.
While a straightforward generalization of \eqref{1.11}, \eqref{1.12}
appears to be difficult, we will next derive a formula for the ratio
of such determinants which indeed permits a direct extension to
higher dimensions.

For this purpose we introduce the boundary trace operators
$\ga_D$ (Dirichlet trace) and $\ga_N$ (Neumann trace) which, in the
current one-dimensional half-line situation, are just the functionals,
\begin{equation}
\ga_D \colon \begin{cases} C([0,\infty)) \to \bbC, \\
\hspace*{1.3cm} g \mapsto g(0), \end{cases}   \quad
\ga_N \colon \begin{cases}C^1([0,\infty)) \to \bbC,  \\
\hspace*{1.43cm} h \mapsto   - h'(0). \end{cases}
\end{equation}
In addition, we denote by $m_{0,+}^D$, $m_+^D$, $m_{0,+}^N$, and $m_+^N$
the Weyl--Titchmarsh $m$-functions corresponding to $H_{0,+}^D$,
$H_{+}^D$, $H_{0,+}^N$, and $H_{+}^N$, respectively, that is,
\begin{align}
m_{0,+}^D(z) &= i z^{1/2}, \qquad m_{0,+}^N (z)= -\frac{1}{m_{0,+}^D(z)}
= i z^{-1/2},  \lb{1.14} \\
m_{+}^D(z) &= \frac{f_+'(z,0)}{f_+(z,0)}, \quad m_{+}^N (z)=
-\frac{1}{m_{+}^D(z)} = -\frac{f_+(z,0)}{f_+'(z,0)}.  \lb{1.15}
\end{align}
We briefly recall the spectral theoretic significance of $m_{+}^D$ in the special case where $V$ is real-valued: It is a Herglotz function (i.e., it maps the open complex upper half-plane $\bbC_+$ analytically into itself) and the measure $d\rho^D_+$ in its Herglotz representation is then the spectral measure of the operator $H_{+}^D$ and hence encodes all spectral information of $H_{+}^D$. Similarly, $m_{+}^D$ also encodes all spectral information of $H_{+}^N$ since $-1/m_{+}^D = m_{+}^N$ is also a Herglotz function and the measure
$d\rho^N_+$ in its Herglotz representation represents the spectral measure of the operator $H_{+}^N$. In particular, $d\rho^D_+$ (respectively, $d\rho^N_+$) uniquely determine $V$ a.e.\ on $(0,\infty)$ by the inverse spectral approach of Gelfand and Levitan \cite{GL55} or Simon \cite{Si99}, \cite{GS00} (see also Remling \cite{Re03} and Section 6 in the survey \cite{Ge07}).

Then we obtain the following result for the ratio of the perturbation determinants in
\eqref{1.11} and \eqref{1.12}:

\begin{theorem} \lb{t1.2}
Assume $V\in L^1((0,\infty);dx)$ and let $z\in\bbC\backslash\si(H_+^D)$
with $\Im(z^{1/2})>0$. Then,
\begin{align}
& \frac{\det\Big(I_+ +\ol{u\big(H_{0, +}^N-z I_+\big)^{-1} v}\,\Big)}
{\det\Big(I_+ +\ol{u\big(H_{0, +}^D-z I_+\big)^{-1}v}\,\Big)}
\no \\
&\quad = 1 - \Big(\,\ol{\ga_N(H_+^D-z I_+)^{-1}V
\big[\ga_D(H_{0,+}^N-\ol{z}I_+)^{-1}\big]^*}\,\Big)  \lb{1.16}  \\
& \quad = \f{W(f_+(z),\phi_+^N(z))}{i z^{1/2}W(f_+(z),\phi_+^D(z))} =
\f{f'_+(z,0)}{i z^{1/2}f_+(z,0)} = \f{m_+^D(z)}{m_{0,+}^D(z)} =
\f{m_{0,+}^N(z)}{m_+^N(z)}.   \lb{1.17}
\end{align}
\end{theorem}

At first sight it may seem unusual to even attempt to derive
\eqref{1.16} in the one-dimensional context since \eqref{1.17}
already yields the reduction of a Fredholm determinant to a simple
Wronski determinant. However, we will see in Section \ref{s4} (cf.\
Theorem \ref{t4.1}) that it is precisely \eqref{1.16} that permits a
natural extension to dimensions $n\in\bbN$, $n\geq 2$.
Moreover, the latter is also instrumental in proving the analog of
\eqref{1.17} in terms of Dirichlet-to-Neumann maps (cf.\ Theorem
\ref{t4.2}).

The proper multi-dimensional generalizations to Schr\"odinger operators in $L^2(\Om;d^n x)$, corresponding to an open set $\Om \subset \bbR^n$ with compact, nonempty boundary $\partial\Om$, more precisely, the proper operator-valued generalization of the Weyl--Titchmarsh function $m_+^D(z)$ is then given by the Dirichlet-to-Neumann map, denoted by $M_{\Om}^D(z)$. This operator-valued map  indeed plays a fundamental role in our extension of \eqref{1.17} to the higher-dimensional case. In particular, under Hypothesis \ref{h2.6} on $\Omega$ and $V$ (which regulates smoothness properties of $\partial\Om$ and $L^p$-properties of $V$), we will prove the following multi-dimensional extension of \eqref{1.16} and \eqref{1.17} in Section
\ref{s4}:

\begin{theorem} \lb{t1.3}
Assume Hypothesis \ref{h2.6} and let $k\in\bbN$, $k\geq p$ and
$z\in\bbC\big\backslash\big(\si\big(H_{\Om}^D\big)\cup
\si\big(H_{0,\Om}^D\big) \cup \si\big(H_{0,\Om}^N\big)\big)$. Then,
\begin{align}
&
\frac{\det{}_k\Big(I_{\Om}+\ol{u\big(H_{0,\Om}^N-zI_{\Om}\big)^{-1}v}\,\Big)}
{\det{}_k\Big(I_{\Om}+\ol{u\big(H_{0,\Om}^D-zI_{\Om}\big)^{-1}v}\,\Big)} \no \\
& \quad = \det{}_k\Big(I_{\dOm} -
\ol{\ga_N\big(H_{\Om}^D-zI_{\Om}\big)^{-1} V
\big[\ga_D(H_{0,\Om}^N-\ol{z}I_{\Om})^{-1}\big]^*}\,\Big)
e^{\tr(T_k(z))}   \lb{1.18}  \\
& \quad = \det{}_k\big(M_{\Om}^{D}(z)M_{0,\Om}^{D}(z)^{-1}\big)
e^{\tr(T_k(z))}.   \lb{1.19}
\end{align}
\end{theorem}
Here, ${\det}_k(\cdot)$ denotes the modified Fredholm determinant in
connection with $\cB_k$ perturbations of the identity and $T_k(z)$ is
some trace class operator. In particular, $T_2(z)$ is given by
\begin{equation}
T_2(z)=\ol{\ga_N\big(H_{0,\Om}^D-zI_{\Om}\big)^{-1} V
\big(H_{\Om}^D-zI_{\Om}\big)^{-1} V
\big[\ga_D \big(H_{0,\Om}^N-\ol{z}I_{\Om}\big)^{-1}\big]^*},
\end{equation}
where $I_{\Om}$ and $I_{\partial\Om}$ represent the identity operators
in $L^2(\Om; d^n x)$ and  $L^2(\partial\Om; d^{n-1} \sigma)$,
respectively (with $d^{n-1}\sigma$ denoting the surface measure on
$\partial\Om$). The sudden appearance of the term
$\exp(\tr(T_k(z)))$ in \eqref{1.18} and \eqref{1.19}, when compared
to the one-dimensional case, is due to the necessary use of the
modified determinant ${\det}_k(\cdot)$ in Theorem \ref{t1.3}.

We note that the multi-dimensional extension \eqref{1.18} of \eqref{1.16}, under the stronger hypothesis $V\in L^2(\Om; d^n x)$, $n=2,3$, first appeared in \cite{GLMZ05}. However, the present results in Theorem \ref{t1.3} go decidedly beyond those in 
\cite{GLMZ05} in the following sense: $(i)$ the class of domains $\Omega$ permitted by Hypothesis \ref{h2.6} (actually, Hypothesis \ref{h2.1}) is greatly enlarged as compared to \cite{GLMZ05};  $(ii)$ the multi-dimensional extension \eqref{1.19} of \eqref{1.17} invoking Dirichlet-to-Neumann maps is a new (and the most significant) result in this paper; $(iii)$ while \cite{GLMZ05} focused on dimensions $n=2,3$, we now treat the general case $n\in\bbN$, $n\geq 2$; $(iv)$ we provide an application involving eigenvalue counting functions at the end of Section \ref{s4}; $(v)$ we study a representation of the product formula for modified Fredholm determinants, which should be of independent interest, at the beginning of Section \ref{s4}. 

The principal reduction in Theorem \ref{t1.3} reduces (a ratio of) modified Fredholm determinants associated with operators in $L^2(\Om; d^n x)$ on the left-hand side of \eqref{1.18} to modified Fredholm determinants associated with operators in
$L^2(\partial\Om; d^{n-1} \sigma)$ on the right-hand side of \eqref{1.18} and especially, in \eqref{1.19}. This is the analog of the reduction described in the one-dimensional context of Theorem \ref{t1.2}, where $\Om$ corresponds to the half-line $(0,\infty)$ and its boundary $\partial\Om$ corresponds to the one-point set $\{0\}$. As a result, the ratio of determinants on the left-hand side of \eqref{1.16} associated with operators in
$L^2((0,\infty); dx)$ is reduced to ratios of Wronskians and Weyl--Titchmarsh functions on the right-hand side of \eqref{1.16} and in \eqref{1.17}.

Finally, we briefly list most of the notational conventions used
throughout this paper. Let $\cH$ be a separable complex Hilbert
space, $(\cdot,\cdot)_{\cH}$ the scalar product in $\cH$ (linear in
the second factor), and $I_{\cH}$ the identity operator in $\cH$.
Next, let $T$ be a linear operator mapping (a subspace of) a
Banach space into another, with $\dom(T)$ and $\ran(T)$ denoting the
domain and range of $T$. The closure of a closable operator $S$ is
denoted by $\ol S$. The kernel (null space) of $T$ is denoted by
$\ker(T)$. The spectrum and resolvent set of a closed linear operator
in $\cH$ will be denoted by $\sigma(\cdot)$ and $\rho(\cdot)$. The
Banach spaces of bounded and compact linear operators in $\cH$ are
denoted by $\cB(\cH)$ and $\cB_\infty(\cH)$, respectively. Similarly,
the Schatten--von Neumann (trace) ideals will subsequently be denoted
by $\cB_k(\cH)$, $k\in\bbN$. Analogous notation $\cB(\cH_1,\cH_2)$,
$\cB_\infty (\cH_1,\cH_2)$, etc., will be used for bounded, compact,
etc., operators between two Hilbert spaces $\cH_1$ and $\cH_2$. In
addition, $\tr(T)$ denotes the trace of a trace class operator
$T\in\cB_1(\cH)$ and $\det_{p}(I_{\cH}+S)$ represents the (modified)
Fredholm determinant associated with an operator $S\in\cB_k(\cH)$,
$k\in\bbN$ (for $k=1$ we omit the subscript $1$). Moreover, $\cX_1
\hookrightarrow \cX_2$ denotes the continuous embedding of the Banach
space $\cX_1$ into the Banach space $\cX_2$.

For general references on the theory of (modified) Fredholm determinants we
refer, for instance, to \cite[Sect.\ XI.9]{DS88}, \cite[Ch.\ Chs.\ IX--XI]{GGK00}, 
\cite[Ch.\ Sect.\ 4.2]{GK69}, \cite[Sect.\ XIII.17]{RS78}, \cite{Si77}, and \cite[Ch.\ 9]{Si05}.

\section{Schr\"odinger Operators with Dirichlet and Neumann boundary conditions}
\label{s2}

In this section we primarily focus on various properties of Dirichlet,
$H^D_{0,\Om}$, and Neumann, $H^N_{0,\Om}$, Laplacians in
$L^2(\Om;d^n x)$ associated with open sets $\Omega\subset \bbR^n$, $n\in\bbN$, $n\geq 2$, introduced in Hypothesis \ref{h2.1} below. In particular, we study mapping properties of $\big(H^{D,N}_{0,\Om}-zI_{\Om}\big)^{-q}$, $q\in [0,1]$ (with $I_{\Om}$ the identity operator in $L^2(\Om; d^n x)$) and trace ideal properties of the maps 
$f \big(H^{D,N}_{0,\Om}-zI_{\Om}\big)^{-q}$, $f\in L^p(\Om; d^n x)$, for appropriate $p\geq 2$, and
$\gamma_N \big(H^{D}_{0,\Om}-zI_{\Om}\big)^{-r}$, and
$\gamma_D \big(H^{N}_{0,\Om}-zI_{\Om}\big)^{-s}$, for appropriate
$r>3/4$, $s>1/4$, with $\gamma_N$ and $\gamma_D$ being the Neumann and Dirichlet boundary trace operators defined in \eqref{2.2} and \eqref{2.3}.

At the end of this section we then introduce the Dirichlet and Neumann
Schr\"odinger operators $H^D_{\Om}$ and $H^N_{\Om}$ in
$L^2(\Om;d^n x)$, that is, perturbations of the Dirichlet and Neumann Laplacians $H^D_{0,\Om}$ and $H^N_{0,\Om}$ by a potential $V$ satisfying
Hypothesis \ref{h2.6}.

We start with introducing our assumptions on the set $\Omega$:

\begin{hypothesis} \lb{h2.1}
Let $n\in\bbN$, $n\geq 2$, and assume that $\Omega\subset{\bbR}^n$ is
an open set with a compact, nonempty boundary
$\partial\Omega$. In addition,
we assume that one of the following three conditions holds: \\
$(i)$ \, $\Omega$ is of class $C^{1,r}$ for some $1/2 < r <1$; \\
$(ii)$  \hspace*{.0001pt} $\Omega$ is convex; \\
$(iii)$ $\Omega$ is a Lipschitz domain satisfying a {\it uniform exterior ball condition} $($UEBC\,$)$.  
\end{hypothesis}

We note that while $\dOm$ is assumed to be compact, $\Om$ may be unbounded in connection with conditions $(i)$ or $(iii)$.
For more details in this context we refer to Appendix \ref{sA}.

First, we introduce the boundary trace operator $\ga_D^0$
(Dirichlet trace) by
\begin{equation}
\ga_D^0\colon C(\ol{\Om})\to C(\dOm), \quad \ga_D^0 u = u|_\dOm .
\end{equation}
Then there exists a bounded, linear operator $\gamma_D$ (cf.
\cite[Theorem 3.38]{Mc00}),
\begin{equation}
\ga_D\colon H^{s}(\Om)\to H^{s-(1/2)}(\dOm) \hookrightarrow \LdOm,
\quad 1/2<s<3/2, \lb{2.2}
\end{equation}
whose action is compatible with that of $\ga_D^0$. That is, the two
Dirichlet trace  operators coincide on the intersection of their
domains. We recall that $d^{n-1}\sigma$ denotes the surface measure on
$\dOm$ and we refer to Appendix \ref{sA} for our notation in
connection with Sobolev spaces.

Next, we introduce the operator $\ga_N$ (Neumann trace) by
\begin{align}
\ga_N = \nu\cdot\ga_D\nabla \colon H^{s+1}(\Om)\to \LdOm, \quad
1/2<s<3/2, \lb{2.3}
\end{align}
where $\nu$ denotes outward pointing normal unit vector to
$\partial\Om$. It follows from \eqref{2.2} that $\ga_N$ is also a
bounded operator.

Given Hypothesis \ref{h2.1}, we introduce the self-adjoint and nonnegative Dirichlet and
Neumann Laplacians $H_{0,\Om}^D$ and $H_{0,\Om}^N$ associated
with the domain $\Om$ as follows,
\begin{align}
H^D_{0,\Om} = -\Delta, \quad \dom\big(H^D_{0,\Om}\big) = \{u\in H^{2}(\Om)
\,|\, \ga_D u = 0\}, \lb{2.4}
\\
H^N_{0,\Om} = -\Delta, \quad \dom\big(H^N_{0,\Om}\big) = \{u\in H^{2}(\Om)
\,|\, \ga_N u = 0\}. \lb{2.5}
\end{align}
A detailed discussion of $H^D_{0,\Om}$ and $H^N_{0,\Om}$ is provided in Appendix
\ref{sA}.

\begin{lemma} \lb{l2.2}
Assume Hypothesis \ref{h2.1}. Then the operators $H_{0,\Om}^D$
and $H_{0,\Om}^N$ introduced in \eqref{2.4} and \eqref{2.5} are
nonnegative and self-adjoint in $\LOm$ and the following boundedness 
properties hold for all $q\in [0,1]$ and $z\in\bbC\backslash[0,\infty)$, 
\begin{align}
\big(H_{0,\Om}^D-zI_{\Om}\big)^{-q},\, \big(H_{0,\Om}^N-zI_{\Om}\big)^{-q}
\in\cB\big(\LOm,H^{2q}(\Om)\big). \lb{2.6}
\end{align}
\end{lemma}

The fractional powers in \eqref{2.6} (and in subsequent analogous
cases) are defined via the functional calculus implied by the
spectral theorem for self-adjoint operators.

As explained in Appendix \ref{sA} (cf.\ particularly Lemma \ref{lA.2}), the
key ingredients in proving Lemma \ref{l2.2} are the inclusions
\begin{equation}
\dom\big(H_{0,\Om}^D\big) \subset H^2(\Om), \quad
\dom\big(H^N_{0,\Om}\big) \subset H^2(\Om)
\end{equation}
and methods based on real interpolation spaces.

For the remainder of this paper we agree to the simplified notation that the operator of multiplication by the measurable function $f$  in $\LOm$ is again denoted by the symbol $f$.

The next result is an extension of \cite[Lemma 6.8]{GLMZ05} and aims at an explicit discussion of the $z$-dependence of the constant $c$ appearing in estimate (6.48) of \cite{GLMZ05}.

\begin{lemma} \lb{l2.3}
Assume Hypothesis \ref{h2.1} and let $2\leq p$, $(n/2p)<q\leq1$,
$f\in L^p(\Om;d^nx)$, and $z\in\bbC\backslash[0,\infty)$. Then,
\begin{align} \lb{2.7}
f\big(H_{0,\Om}^D-zI_{\Om}\big)^{-q}, \, f\big(H_{0,\Om}^N-zI_{\Om}\big)^{-q}
\in\cB_p\big(\LOm\big),
\end{align}
and for some $c>0$ $($independent of $z$ and $f$\,$)$
\begin{align}
\begin{split}
&\big\| f\big(H_{0,\Om}^D-zI_{\Om}\big)^{-q}\big\|_{\cB_p(\LOm)}^2
\\
& \qquad \leq
c\bigg(1+\frac{\abs{z}^{2q}+1}{\dist\big(z,\si\big(H_{0,\Om}^D\big)\big)^{2q}}\bigg)
\|(\abs{\cdot}^2-z)^{-q}\|_{L^p(\bbR^n;d^nx)}^2
\|f\|_{L^p(\Om;d^nx)}^2,
\\
&\big\| f\big(H_{0,\Om}^N-zI_{\Om}\big)^{-q}\big\|_{\cB_p(\LOm)}^2
\\
& \qquad \leq
c\bigg(1+\frac{\abs{z}^{2q}+1}{\dist\big(z,\si\big(H_{0,\Om}^N\big)\big)^{2q}}\bigg)
\|(\abs{\cdot}^2-z)^{-q}\|_{L^p(\bbR^n;d^nx)}^2
\|f\|_{L^p(\Om;d^nx)}^2.
\end{split} \lb{2.8}
\end{align}
\end{lemma}
\begin{proof}
We start by noting that under the assumption that $\Om$ is a
Lipschitz domain, there is a bounded extension operator $\cE$,
\begin{equation}
\cE\in\cB\big(H^{s}(\Om),H^{s}(\bbR^n)\big) \, \text{ such that }
\, (\cE u){|_\Om} = u, \quad u\in H^{s}(\Om), \lb{2.9}
\end{equation}
for all $s\in\bbR$ (see, e.g., \cite{Ry99}). Next, for notational
convenience, we denote by $H_{0,\Om}$ either one of the operators
$H_{0,\Om}^D$ or $H_{0,\Om}^N$ and by $\cR_{\Om}$ the restriction
operator
\begin{equation}
\cR_{\Om}\colon \begin{cases} L^2(\bbR^n;d^nx) \to \LOm, \\
\hspace*{1.65cm}u \mapsto u|_{\Om}. \end{cases}
\end{equation}
Moreover, we introduce the following extension $\ti f$ of $f$,
\begin{equation}
\ti f(x) =
\begin{cases}f(x), & x\in\Om, \\ 0, &
x\in\bbR^n\backslash\Om, \end{cases} \quad  \ti f\in
L^p(\bbR^n;d^nx).
\end{equation}
Then,
\begin{equation} \lb{2.12}
f (H_{0,\Om}-zI_{\Om})^{-q}= \cR_\Om {\ti f}
(H_{0}-zI)^{-q}(H_{0}-zI)^{q}\cE (H_{0,\Om}-zI_{\Om})^{-q},
\end{equation}
where (for simplicity) $I$ denotes the identity operator in
$L^2(\bbR^n;d^nx)$ and $H_0$ denotes the nonnegative self-adjoint
operator
\begin{equation}
H_0 = -\Delta, \quad \dom(H_0)=H^{2}(\bbR^n)
\end{equation}
in $L^2(\bbR^n;d^nx)$.

Let $g\in \LOm$ and define $h=(H_{0,\Om}-zI_{\Om})^{-q}g$, then
by Lemma \ref{lA.2}, $h\in H^{2q}(\Om)\subset \LOm$. Using the
spectral theorem for the nonnegative self-adjoint operator
$H_{0,\Om}$ in $\LOm$, one computes,
\begin{align} \lb{2.14}
\norm{h}_{\LOm}^2 &= \norm{(H_{0,\Om}-zI_{\Om})^{-q}g}_{\LOm}^2
\no
\\
&= \int_{\si(H_{0,\Om})}
\abs{\la-z}^{-2q}\big(dE_{H_{0,\Om}}(\la)g,g\big)_{\LOm}
\\
&\leq \dist(z,\si(H_{0,\Om}))^{-2q}\norm{g}_{\LOm}^2 \no
\end{align}
and since
$(H_{0,\Om}+I_{\Om})^{-q}\in\cB(L^{2}(\Om;d^nx),H^{2q}(\Om))$,
\begin{align} \lb{2.15}
\norm{h}_{H^{2q}(\Om)}^2 &=
\norm{(H_{0,\Om}+I_{\Om})^{-q}(H_{0,\Om}+I_{\Om})^{q}h}_{H^{2q}(\Om)}^2
\leq c \norm{(H_{0,\Om}+I_{\Om})^{q}h}_{L^{2}(\Om;d^nx)}^2 \no
\\
&= c\int_{\si(H_{0,\Om})}
\abs{\la+1}^{2q}\big(dE_{H_{0,\Om}}(\la)h,h\big)_{\LOm} \no
\\
&\leq 2c\int_{\si(H_{0,\Om})}
\big(\abs{\la-z}^{2q}+\abs{z+1}^{2q}\big)
\big(dE_{H_{0,\Om}}(\la)h,h\big)_{\LOm}
\\
&= 2c\big(\norm{(H_{0,\Om}-zI_{\Om})^{q}h}_{H^{2q}(\Om)}^2 +
\abs{z+1}^{2q}\norm{h}_{\LOm}^2 \big)\no
\\
&\leq 2c\big(1+\abs{z+1}^{2q}\dist(z,\si(H_{0,\Om}))^{-2q}\big)
\norm{g}_{\LOm}^2, \no
\end{align}
where $E_{H_{0,\Om}}(\cdot)$ denotes the family of spectral projections of
$H_{0,\Om}$. Moreover, utilizing the representation of
$(H_{0}-zI)^{q}$ as the operator of multiplication by
$\big(\abs{\xi}^2-z\big)^{q}$ in the Fourier space
$L^2(\bbR^n;d^n\xi)$, and the fact that by \eqref{2.9}
\begin{equation}
\cE\in\cB\big(H^{2q}(\Om),H^{2q}(\bbR^n)\big) \cap
\cB\big(\LOm,L^2(\bbR^n;d^nx)\big),
\end{equation}
one computes
\begin{align} \lb{2.17}
\begin{split}
\norm{(H_{0}-zI)^{q}\cE h}_{L^2(\bbR^n;d^nx)}^2 &= \int_{\bbR^n}
d^n\xi\, \big|\abs{\xi}^2-z\big|^{2q} \, \abs{(\ha{\cE h})(\xi)}^2
\\
&\leq 2\int_{\bbR^n} d^n \xi
\big(\abs{\xi}^{4q}+\abs{z}^{2q}\big)\abs{(\ha{\cE h})(\xi)}^2
\\
&\leq 2\big(\norm{\cE h}_{H^{2q}(\bbR^n)}^2 +
\abs{z}^{2q}\norm{\cE h}_{L^2(\bbR^n;d^nx)}^2\big)
\\
&\leq 2c\big(\norm{h}_{H^{2q}(\Om)}^2 +
\abs{z}^{2q}\norm{h}_{\LOm}^2\big).
\end{split}
\end{align}
Combining the estimates \eqref{2.14}, \eqref{2.15}, and
\eqref{2.17}, one obtains
\begin{align}
(H_{0}-zI)^{q}\cE (H_{0,\Om}-zI_{\Om})^{-q} \in
\cB\big(\LOm,L^2(\bbR^n;d^nx)\big) \lb{2.18}
\end{align}
and the following norm estimate with some constant $c>0$,
\begin{align}
\norm{(H_{0}-zI)^{q}\cE (H_{0,\Om}-zI_{\Om})^{-q}}_{
\cB(\LOm,L^2(\bbR^n;d^nx)) }^2 \leq
c+\frac{c(\abs{z}^{2q}+1)}{\dist(z,\si(H_{0,\Om}))^{2q}}.
\lb{2.19}
\end{align}

Next, by  \cite[Theorem 4.1]{Si05} (or \cite[Theorem XI.20]{RS79})
one obtains
\begin{align}
\ti f(H_{0}-zI)^{-q}\in\cB_p\big(L^2(\bbR^n;d^nx)\big) \lb{2.20}
\end{align}
and
\begin{align}\begin{split}
\big\|\ti f(H_{0}-zI)^{-q}\big\|_{\cB_p(L^2(\bbR^n;d^nx))} &\leq
c\, \|(\abs{\cdot}^2-z)^{-q}\|_{L^p(\bbR^n;d^nx)} \|\ti
f\|_{L^p(\bbR^n;d^nx)} \\&= c\,
\|(\abs{\cdot}^2-z)^{-q}\|_{L^p(\bbR^n;d^nx)}
\|f\|_{L^p(\Om;d^nx)}. \lb{2.21}
\end{split}\end{align}
Thus, \eqref{2.7} follows from \eqref{2.12}, \eqref{2.18},
\eqref{2.20}, and \eqref{2.8} follows from \eqref{2.12},
\eqref{2.19}, and \eqref{2.21}.
\end{proof}

Next we recall certain mapping properties of powers of the resolvents of Dirichlet and Neumann Laplacians multiplied by the Neumann and Dirichlet boundary trace operators, respectively:

\begin{lemma} \lb{l2.4}
Assume Hypothesis \ref{h2.1} and let $\eps > 0$,
$z\in\bbC\backslash[0,\infty)$. Then,
\begin{align}
\ga_N\big(H_{0,\Om}^D-zI_{\Om}\big)^{-(3+\eps)/4},
\ga_D\big(H_{0,\Om}^N-zI_{\Om}\big)^{-(1+\eps)/4} \in
\cB\big(\LOm,\LdOm\big).  \lb{2.22}
\end{align}
\end{lemma}

As in \cite[Lemma 6.9]{GLMZ05}, Lemma \ref{l2.4} follows from Lemma \ref{l2.2}
and from \eqref{2.2} and \eqref{2.3}.

\begin{corollary} \lb{c2.5}
Assume Hypothesis \ref{h2.1} and let $f_1\in L^{p_1}(\Om;d^nx)$,
$p_1\geq 2$, $p_1>2n/3$, $f_2\in L^{p_2}(\Om;d^nx)$, $p_2 > 2n$,
and $z\in\bbC\backslash[0,\infty)$. Then, denoting by $f_1$ and
$f_2$ the operators of multiplication by functions $f_1$ and
$f_2$ in $\LOm$, respectively, one has
\begin{align}
\ol{\ga_D\big(H_{0,\Om}^N-zI_{\Om}\big)^{-1}f_1} &\in
\cB_{p_1}\big(L^2(\Om;d^nx),L^2(\dOm;d^{n-1}\si)\big), \lb{2.25}
\\
\ol{\ga_N\big(H_{0,\Om}^D-zI_{\Om}\big)^{-1}f_2} &\in
\cB_{p_2}\big(L^2(\Om;d^nx),L^2(\dOm;d^{n-1}\si)\big) \lb{2.26}
\end{align}
and for some $c_j(z)>0$ $($independent of $f_j$$)$, $j=1,2$,
\begin{align}
\Big\| \, \ol{\ga_D\big(H_{0,\Om}^N-zI_{\Om}\big)^{-1}f_1}\, \Big\|_
{\cB_{p_1}(L^2(\Om;d^nx),L^2(\dOm;d^{n-1}\si))} &\leq c_1(z)
\norm{f_1}_{L^{p_1}(\Om;d^nx)}, \lb{2.27}
\\
\Big\| \, \ol{\ga_N\big(H_{0,\Om}^D-zI_{\Om}\big)^{-1}f_2}\, \Big\|_
{\cB_{p_2}(L^2(\Om;d^nx),L^2(\dOm;d^{n-1}\si))} &\leq c_2(z)
\norm{f_2}_{L^{p_2}(\Om;d^nx)}. \lb{2.28}
\end{align}
\end{corollary}

As in \cite[Corollary 6.10]{GLMZ05}, Corollary \ref{c2.5} follows from Lemmas \ref{l2.3} and \ref{l2.4}.

Finally, we turn to our assumptions on the potential $V$ and the corresponding definition of Dirichlet and Neumann Schr\"odinger operators $H^D_{\Om}$ and $H^N_{\Om}$ in $L^2(\Om; d^n x)$:

\begin{hypothesis} \lb{h2.6}
Suppose that $\Om$ satisfies Hypothesis \ref{h2.1}  and assume that
$V\in L^p(\Om;d^nx)$ for some $p$ satisfying $p>4/3$ in the case
$n=2$, and $p>n/2$ in the case $n\geq3$.
\end{hypothesis}

Assuming Hypothesis \ref{h2.6}, we next introduce the perturbed operators
$H_{\Om}^D$ and $H_{\Om}^N$ in $\LOm$ by alluding to abstract perturbation results summarized in Appendix \ref{sB} as follows: Let $V$, $u$, and $v$ denote
the operators of multiplication by functions $V$,
$u=\exp(i\arg(V))\abs{V}^{1/2}$, and $v=\abs{V}^{1/2}$ in
$L^{2}(\Om;d^nx)$, respectively. Since $u,v\in
L^{2p}(\Om;d^nx)$, Lemma \ref{l2.3} yields
\begin{align}
u\big(H_{0,\Om}^D-zI_{\Om}\big)^{-1/2}, \,
\ol{\big(H_{0,\Om}^D-zI_{\Om}\big)^{-1/2}v} &\in\cB_{2p}\big(\LOm\big),
\quad z\in\bbC\backslash [0,\infty), \lb{2.31}
\\
u\big(H_{0,\Om}^N-zI_{\Om}\big)^{-1/2}, \,
\ol{\big(H_{0,\Om}^N-zI_{\Om}\big)^{-1/2}v} &\in\cB_{2p}\big(\LOm\big),
\quad z\in\bbC\backslash [0,\infty), \lb{2.32}
\end{align}
and hence, in particular,
\begin{align}
& \dom(u)=\dom(v) \supseteq  H^{1}(\Om) \supset H^{2}(\Om)
\supset \dom\big(H^N_{0,\Om}\big), \\
& \dom(u)=\dom(v) \supseteq  H^{1}(\Om) \supseteq H^{1}_0(\Om)
\supset \dom\big(H^D_{0,\Om}\big).
\end{align}
Thus, operators $H_{0,\Om}^D$, $H_{0,\Om}^N$, $u$, and $v$
satisfy Hypothesis \ref{hB.1}\,$(i)$. Moreover, \eqref{2.31} and
\eqref{2.32} imply
\begin{equation}
\ol{u\big(H_{0,\Om}^D-zI_{\Om}\big)^{-1}v}, \,
\ol{u\big(H_{0,\Om}^N-zI_{\Om}\big)^{-1}v} \in\cB_p\big(\LOm\big),
\quad z\in\bbC\backslash [0,\infty),  \lb{2.35}
\end{equation}
which verifies Hypothesis \ref{hB.1}\,$(ii)$ for $H_{0,\Om}^D$ and
$H_{0,\Om}^N$. Utilizing \eqref{2.8} in Lemma \ref{l2.3} with
$-z>0$ sufficiently large, such that the $\cB_{2p}$-norms of the
operators in \eqref{2.31} and \eqref{2.32} are less than 1, and hence
the $\cB_p$-norms of the operators in \eqref{2.35} are less than 1,
one also verifies Hypothesis \ref{hB.1}\,$(iii)$. Thus, applying
Theorem \ref{tB.2} one obtains the densely defined, closed operators
$H_{\Om}^D$ and $H_{\Om}^N$ (which are extensions of $H_{0,\Om}^D+V$
on $\dom\big(H_{0,\Om}^D\big)\cap\dom(V)$ and $H_{0,\Om}^N+V$ on
$\dom\big(H_{0,\Om}^N\big)\cap\dom(V)$, respectively). In particular,
the resolvent of $H_{\Om}^D$ (respectively, $H_{\Om}^N$) is explicitly given
by the analog of \eqref{B.5} in terms of the resolvent of $H_{0,
\Om}^D$ (respectively, $H_{0, \Om}^N$) and the factorization $V=uv$.

We note in passing that \eqref{2.6}--\eqref{2.8}, \eqref{2.22},
\eqref{2.25}--\eqref{2.28}, \eqref{2.31}, \eqref{2.32},
\eqref{2.35}, etc., extend of course to all $z$ in the resolvent
set of the corresponding operators $H_{0,\Om}^D$ and
$H_{0,\Om}^N$.

\section{Dirichlet and Neumann boundary value problems \\
and Dirichlet-to-Neumann maps} \label{s3}

This section is devoted to Dirichlet and Neumann boundary value problems associated with the Helmholtz differential expression $-\Delta - z$ as well as the corresponding differential expression $-\Delta + V - z$ in the presence of a potential $V$, both in connection with the open set $\Omega$. In addition, we provide a detailed discussion of Dirichlet-to-Neumann, $M^D_{0,\Om}$, $M^D_{\Om}$, and Neumann-to-Dirichlet maps,
$M^N_{0,\Om}$, $M^N_{\Om}$, in $L^2(\partial\Om; d^{n-1}\sigma)$.

Denote by
 \begin{equation}
 \wti\ga_N : \big\{u\in H^1(\Om) \,\big|\, \Delta u \in
\big(H^1(\Om)\big)^*\big\} \to H^{-1/2}(\dOm)  \lb{3.0}
\end{equation}
a weak Neumann trace operator defined by
\begin{align} \lb{3.1a}
\langle\wti\ga_N u, \phi\rangle = \int_\Om d^n x \, \nabla
u(x)\cdot\nabla \Phi(x)  + \langle\Delta u, \Phi\rangle
\end{align}
for all $\phi\in H^{1/2}(\dOm)$ and $\Phi\in H^1(\Om)$ such that
$\ga_D\Phi = \phi$. We note that this definition is
independent of the particular extension $\Phi$ of $\phi$, and that
$\wti\ga_N$ is a bounded extension of the Neumann trace operator
$\ga_N$ defined in \eqref{2.3}. For more details we refer to
equations \eqref{A.11}--\eqref{A.16}.

We start with the Helmholtz Dirichlet and Neumann boundary value problems:

\begin{theorem} \lb{t3.1}
Suppose $\Om$ is an open Lipschitz domain with a
compact nonempty boundary $\dOm$. Then for every $f \in
H^1(\dOm)$ and $z\in\bbC\big\backslash\si\big(H_{0,\Om}^D\big)$ the
following Dirichlet boundary value problem,
\begin{align} \lb{3.1}
\begin{cases}
(-\Delta - z)u_0^D = 0 \text{ on }\, \Om, \quad u_0^D \in H^{3/2}(\Om), \\
\ga_D u_0^D = f \text{ on }\, \dOm,
\end{cases}
\end{align}
has a unique solution $u_0^D$ satisfying $\wti\ga_N u_0^D \in
\LdOm$. Moreover, there exist constants $C^D=C^D(\Omega,z)>0$ such that
\begin{equation}
\|u_0^D\|_{H^{3/2}(\Omega)} \leq C^D \|f\|_{H^1(\partial\Omega)}.  \lb{3.3a}
\end{equation}
Similarly, for every $g\in\LdOm$ and
$z\in\bbC\backslash\si\big(H_{0,\Om}^N\big)$ the following Neumann
boundary value problem,
\begin{align} \lb{3.2}
\begin{cases}
(-\Delta - z)u_0^N = 0 \text{ on }\,\Om,\quad u_0^N \in H^{3/2}(\Om), \\
\wti\ga_N u_0^N = g\text{ on }\,\dOm,
\end{cases}
\end{align}
has a unique solution  $u_0^N$. Moreover, there exist constants $C^N=C^N(\Omega,z)>0$ such that
\begin{equation}
\|u_0^N\|_{H^{3/2}(\Omega)} \leq C^N \|g\|_{\LdOm}.  \lb{3.4a}
\end{equation}
In addition,  \eqref{3.1}--\eqref{3.4a} imply that the following maps are bounded  
\begin{align}
\big[\ga_N\big(\big(H^D_{0,\Om}-zI_\Om\big)^{-1}\big)^*\big]^* &: H^1(\dOm) \to
H^{3/2}(\Om), \quad z\in\bbC\big\backslash\si\big(H^D_{0,\Om}\big),   \lb{3.4b}
\\
\big[\ga_D \big(\big(H^N_{0,\Om}-zI_\Om\big)^{-1}\big)^*\big]^* &: \LdOm \to
H^{3/2}(\Om), \quad z\in\bbC\big\backslash\si\big(H^N_{0,\Om}\big).   \lb{3.4c}
\end{align}
Finally, the solutions $u_0^D$ and $u_0^N$ are given by the formulas
\begin{align}
u_0^D (z) &= -\big(\ga_N \big(H_{0,\Om}^D-\ol{z}I_\Om\big)^{-1}\big)^*f,
\lb{3.3}
\\
u_0^N (z) &= \big(\ga_D \big(H_{0,\Om}^N-\ol{z}I_\Om\big)^{-1}\big)^*g. 
\lb{3.4}
\end{align}
\end{theorem}
\begin{proof}
It follows from Theorem 9.3 in \cite{Mi96} that the boundary
value problems,
\begin{align}
&\begin{cases} \lb{3.5}
(\Delta + z)u_0^D = 0 \text{ on } \Om, \quad \cN(\nabla u_0^D)\in\LdOm, \\
\ga_D u_0^D = f\in H^1(\dOm) \text{ on } \dOm
\end{cases}
\intertext{and}
&\begin{cases} \lb{3.6}
(\Delta + z)u_0^N = 0 \text{ on } \Om, \quad \cN(\nabla u_0^N)\in\LdOm, \\
\wti\ga_N u_0^N = g\in\LdOm \text{ on } \dOm,
\end{cases}
\end{align}
have unique solutions for all $z\in\bbC\backslash\si\big(H_{0,\Om}^D\big)$
and $z\in\bbC\backslash\si\big(H_{0,\Om}^N\big)$, respectively, satisfying natural estimates. Here $\cN(\cdot)$ denotes the non-tangential maximal function (cf.\ \cite{JK95}, \cite{Mi96})
\begin{equation}
(\cN w)(x) = \sup_{y\in\Gamma(x)}|w(y)|, \quad x\in\partial\Om, \end{equation}
where $w$ is a locally bounded function and $\Gamma(x)$ is a nontangential approach region with vertex at $x$, that is, for some fixed constant $C>1$  one has
\begin{equation}
\Gamma(x)=\{ y\in\Om \;|\; |x-y| < C \, {\rm dist}(y,\partial\Om) \}.
\end{equation}

In the case of a bounded domain $\Om$, it follows from Corollary
5.7 in \cite{JK95} that for any harmonic function $v$ in $\Om$,
\begin{align}
\cN(\nabla v)\in\LdOm \, \text{ if and only if } \, v\in H^{3/2}(\Om),   \lb{3.7}
\end{align}
accompanied with natural estimates. For any solution $u$ of the Helmholtz equation
$(\Delta+z)u=0$ on a bounded domain $\Om$, one can introduce the 
harmonic function 
\begin{equation}
v(x)=u(x)+z\int_\Om d^n y\, E_n(x-y) u(y), \quad  x\in\Om, 
\end{equation}
such that $\cN(\nabla u)\in\LdOm$ if and only if
$\cN(\nabla v)\in\LdOm$, and $u\in H^{3/2}(\Om)$ if and only if $v\in H^{3/2}(\Om)$. (Again, 
natural estimates are valid in each case.)  
Here $E_n$ denotes the fundamental solution of the Laplace equation in $\bbR^n$, 
$n\in\bbN$, $n\geq 2$, 
\begin{equation}
E_n(x) =\begin{cases} \f{1}{2\pi} \ln(|x|), & n=2, \\ 
\f{1}{n(2-n)\omega_{n-1}}|x|^{2-n}, & n\geq 3, \end{cases}, 
\quad x\in\bbR^n\backslash\{0\}, 
\end{equation}
with $\omega_{n-1}$ denoting the area of the unit sphere in $\bbR^n$. 
The equivalence in \eqref{3.7} extends from
harmonic functions to all functions $u$ satisfying the Helmholtz
equation, $(\Delta+z)u=0$ on a bounded domain $\Om$,
\begin{align}
\cN(\nabla u)\in\LdOm \, \text{ if and only if } \, u\in H^{3/2}(\Om). \lb{3.8}
\end{align}
Thus, in the case of a bounded domain $\Om$, \eqref{3.1} and
\eqref{3.2} follow from \eqref{3.5}, \eqref{3.6}, and
\eqref{3.8}. Moreover, one has the chain of estimates
\begin{equation}
\|u^D_0\|_{H^{3/2}(\Omega)} \leq C_1 \big[\big\|\cN\big(\nabla u^D_0\big)\big\|_{\LdOm} 
+ \|u^D_0\|_{L^2(\Om; d^n x)}\big] \leq C_2 \|f\|_{H^1(\LdOm)}
\end{equation}
for some constants $C_k>0$, $k=1,2$. In the case of an unbounded domain 
$\Om$, one first
obtains \eqref{3.8} for $\Om\cap B$, where $B$ is a sufficiently
large ball containing $\dOm$. Then, since
$z\in\bbC\big\backslash\si\big(H_{0,\Om}^D\big) =
\bbC\big\backslash\si\big(H_{0,\Om}^N\big) = \bbC\backslash[0,\infty)$ (since now
$\Om$ contains the exterior of a ball in $\bbR^n$), one
exploits the exponential decay of solutions of the Helmholtz
equation to extend \eqref{3.8} from $\Om\cap B$ to $\Om$. This,
together with \eqref{3.5} and \eqref{3.6}, yields \eqref{3.1} and
\eqref{3.2}.

Next, we turn to the proof of \eqref{3.3} and \eqref{3.4}. We note that by 
Lemma \ref{l2.4},
\begin{align}
\ga_N \big(H_{0,\Om}^D-\ol{z}I_\Om\big)^{-1}, \ga_D
\big(H_{0,\Om}^N-\ol{z}I_\Om\big)^{-1} \in \cB\big(\LOm,\LdOm\big),
\end{align}
and hence
\begin{align}
\big(\ga_N \big(H_{0,\Om}^D-\ol{z}I_\Om\big)^{-1}\big)^*, \big(\ga_D
\big(H_{0,\Om}^N-\ol{z}I_\Om\big)^{-1}\big)^* \in \cB\big(\LdOm,\LOm\big).  \lb{3.21a}
\end{align}
Then, denoting by $u_0^D$ and $u_0^N$ the unique solutions of
\eqref{3.1} and \eqref{3.2}, respectively, and using Green's
formula, one computes
\begin{align}
\big(u_0^D,v\big)_{\LOm} &=
\big(u_0^D,(-\Delta-\ol{z})\big(H_{0,\Om}^D-\ol{z}I_\Om\big)^{-1}v\big)_{\LOm}
\no
\\ &=
\big((-\Delta-z)u_0^D,\big(H_{0,\Om}^D-\ol{z}I_\Om\big)^{-1}v\big)_{\LOm}
\no
\\ &\quad +
\big(\wti\ga_N u_0^D, \ga_D
\big(H_{0,\Om}^D-\ol{z}I_\Om\big)^{-1}v\big)_{\LdOm}  \no
\\ &\quad -
\big(\ga_D u_0^D, \ga_N
\big(H_{0,\Om}^D-\ol{z}I_\Om\big)^{-1}v\big)_{\LdOm} \no
\\ &=
-\big(f, \ga_N \big(H_{0,\Om}^D-\ol{z}I_\Om\big)^{-1}v\big)_{\LdOm} \no
\\ &=
-\big(\big(\ga_N \big(H_{0,\Om}^D-\ol{z}I_\Om\big)^{-1}\big)^*f,v\big)_{\LOm}
\end{align}
and
\begin{align}
\big(u_0^N,v\big)_{\LOm} &=
\big(u_0^N,(-\Delta-\ol{z})\big(H_{0,\Om}^N-\ol{z}I_\Om\big)^{-1}v\big)_{\LOm}
\no
\\ &=
\big((-\Delta-z)u_0^N,\big(H_{0,\Om}^N-\ol{z}I_\Om\big)^{-1}v\big)_{\LOm}
\no
\\ &\quad +
\big(\wti\ga_N u_0^N, \ga_D
\big(H_{0,\Om}^N-\ol{z}I_\Om\big)^{-1}v\big)_{\LdOm}  \no
\\ &\quad -
\big(\ga_D u_0^N, \ga_N
\big(H_{0,\Om}^N-\ol{z}I_\Om\big)^{-1}v\big)_{\LdOm} \no
\\ &=
\big(g, \ga_D \big(H_{0,\Om}^N-\ol{z}I_\Om\big)^{-1}v\big)_{\LdOm} \no
\\ &=
\big(\big(\ga_D \big(H_{0,\Om}^N-\ol{z}I_\Om\big)^{-1}\big)^*g,v\big)_{\LOm}
\end{align}
for any $v\in\LOm$. This proves \eqref{3.3} and \eqref{3.4} with the operators involved understood in the sense of \eqref{3.21a}. Granted \eqref{3.3a} and \eqref{3.4a}, one finally obtains \eqref{3.4b} and \eqref{3.4c}.
\end{proof}

We temporarily strengthen our hypothesis on $V$ and introduce the following assumption:

\begin{hypothesis} \lb{h3.2}
Suppose the set $\Om$ satisfies Hypothesis \ref{h2.1} and assume that
$V\in L^p(\Om;d^nx)$ for some $p>2$ if $n=2,3$ and $p\geq 2n/3$ if $n\geq4$.
\end{hypothesis}

By employing a perturbative approach, we now extend Theorem \ref{t3.1} in connection with the Helmholtz differential expression $-\Delta - z$ on
$\Om$ to the case of a Schr\"odinger  differential expression  $-\Delta + V - z$ on $\Om$.

\begin{theorem} \lb{t3.3}
Assume Hypothesis \ref{h3.2}. Then for every $f \in
H^1(\dOm)$ and
$z\in\bbC\big\backslash\si\big(H_{\Om}^D\big)$
the following Dirichlet boundary value problem,
\begin{align} \lb{3.9}
\begin{cases}
(-\Delta + V - z)u^D = 0 \text{ on }\, \Om, \quad u^D \in H^{3/2}(\Om), \\
\ga_D u^D = f \text{ on }\, \dOm,
\end{cases}
\end{align}
has a unique solution $u^D$ satisfying $\wti\ga_N u^D \in
\LdOm$. Moreover, there exist constants $C^D=C^D(\Omega,z)>0$ such that
\begin{equation}
\|u^D\|_{H^{3/2}(\Omega)} \leq C^D \|f\|_{H^1(\partial\Omega)}.  \lb{3.9a}
\end{equation}
Similarly, for every $g\in\LdOm$ and
$z\in\bbC\big\backslash\si\big(H_{\Om}^N\big)$
the following Neumann boundary value problem,
\begin{align} \lb{3.10}
\begin{cases}
(-\Delta + V - z)u^N = 0 \text{ on }\, \Om,\quad u^N \in H^{3/2}(\Om), \\
\wti\ga_N u^N = g\text{ on }\, \dOm,
\end{cases}
\end{align}
has a unique solution  $u^N$. Moreover, there exist constants $C^N=C^N(\Omega,z)>0$ such that
\begin{equation}
\|u^N\|_{H^{3/2}(\Omega)} \leq C^N \|g\|_{\LdOm}.  \lb{3.10a}
\end{equation}
In addition,  \eqref{3.9}--\eqref{3.10a} imply that the following maps are bounded  
\begin{align}
\big[\ga_N\big(\big(H^D_{\Om}-zI_\Om\big)^{-1}\big)^*\big]^* &: H^1(\dOm) \to
H^{3/2}(\Om), \quad z\in\bbC\big\backslash\si\big(H^D_{0,\Om}\big),   \lb{3.10b}
\\
\big[\ga_D \big(\big(H^N_{\Om}-zI_\Om\big)^{-1}\big)^*\big]^* &: \LdOm \to
H^{3/2}(\Om), \quad z\in\bbC\big\backslash\si\big(H^N_{0,\Om}\big).   \lb{3.10c}
\end{align}
Finally, the solutions $u^D$ and $u^N$ are given by the formulas
\begin{align}
u^D (z) &= -\big[\ga_N \big(\big(H_{\Om}^D-zI_\Om\big)^{-1}\big)^*\big]^*f,
\lb{3.11}
\\
u^N (z) &= \big[\ga_D \big(\big(H_{\Om}^N-zI_\Om\big)^{-1}\big)^*\big]^*g.
\lb{3.12}
\end{align}
\end{theorem}
\begin{proof}
We temporarily assume that
$z\in\bbC\big\backslash\big(\si\big(H_{0,\Om}^D\big)\cup\si\big(H_{\Om}^D\big)\big)$ in the case of the Dirichlet problem and
$z\in\bbC\big\backslash\big(\si\big(H_{0,\Om}^N\big)\cup\si\big(H_{\Om}^N\big)\big)$ in the context of the Neumann problem. Uniqueness of solutions follows from the fact that $z\notin\si(H_\Om^D)$ and $z\notin\si(H_\Om^N)$, respectively.

Next, we will show that the functions
\begin{align}
u^D (z) &= u_0^D (z) - \big(H_\Om^D-zI_\Om\big)^{-1} V u_0^D (z), \lb{3.13}
\\
u^N (z)&= u_0^N (z) - \big(H_\Om^N-zI_\Om\big)^{-1} V u_0^N (z),  \lb{3.14}
\end{align}
with $u_0^D, u_0^N$ given by Theorem \ref{t3.1}, satisfy \eqref{3.11}
and \eqref{3.12}, respectively. Indeed, it follows from Theorem
\ref{t3.1} that $u_0^D,u_0^N\in H^{3/2}(\Om)$ and $\wti\ga_N u_0^D
\in \LdOm$. Using the Sobolev embedding theorem
\begin{align*}
H^{3/2}(\Om) \hookrightarrow L^q(\Om;d^nx) \text{ for all } q\geq2
\text{ if } n=2,3 \text{ and } 2\leq q \leq 2n/(n-3) \text{ if }
n\geq4,
\end{align*}
and the fact that $V\in L^p(\Om;d^nx)$, $p>2$ if $n=2,3$ and
$p\geq 2n/3$ if $n\geq4$, one concludes that $Vu_0^D,
Vu_0^N\in\LOm$, and hence \eqref{3.13} and \eqref{3.14} are 
well-defined. Moreover, it follows from Lemma \ref{l2.3} that
$V\big(H_{0,\Om}^D-zI_\Om\big)^{-1}$,
$V\big(H_{0,\Om}^N-zI_\Om\big)^{-1}\in\cB_p\big(\LOm\big)$, and hence
\begin{align}
\big[I_{\Om}+V\big(H_{0,\Om}^D-zI_\Om\big)^{-1}\big]^{-1} \in \cB\big(\LOm\big), \quad
z\in\bbC\big\backslash\big(\si\big(H_{0,\Om}^D\big)\cup\si\big(H_{\Om}^D\big)\big),
\lb{3.15}
\\
\big[I_{\Om}+V\big(H_{0,\Om}^N-zI_\Om\big)^{-1}\big]^{-1} \in \cB\big(\LOm\big), \quad
z\in\bbC\big\backslash\big(\si\big(H_{0,\Om}^N\big)\cup\si\big(H_{\Om}^N\big)\big),
\lb{3.16}
\end{align}
by applying Theorem \ref{tB.3}. Thus, by \eqref{2.4} and \eqref{2.5},
\begin{align}
\big(H_\Om^D-zI_\Om\big)^{-1} V u_0^D &=
\big(H_{0,\Om}^D-zI_\Om\big)^{-1}\big[I_{\Om}+V\big(H_{0,\Om}^D-zI_\Om\big)^{-1}\big]^{-1}Vu_0^D
\in H^2(\Om),
\\
\big(H_\Om^N-zI_\Om\big)^{-1} V u_0^N &=
\big(H_{0,\Om}^N-zI_\Om\big)^{-1}\big[I_{\Om}+V\big(H_{0,\Om}^N-zI_\Om\big)^{-1}\big]^{-1}Vu_0^N
\in H^2(\Om),
\end{align}
and hence $u^D,u^N\in H^{3/2}(\Om)$ and $\wti\ga_N u^D \in \LdOm$.
Moreover,
\begin{align}
(-\Delta+V-z)u^D &= (-\Delta-z)u_0^D + Vu_0^D -
(-\Delta+V-z)\big(H_{\Om}^D-zI_\Om\big)^{-1}Vu_0^D \no
\\ &=
Vu_0^D - I_\Om Vu_0^D = 0,
\\
(-\Delta+V-z)u^N &= (-\Delta-z)u_0^N + Vu_0^N -
(-\Delta+V-z)\big(H_{\Om}^N-zI_\Om\big)^{-1}Vu_0^N \no
\\ &=
Vu_0^N - I_\Om Vu_0^N = 0,
\end{align}
and by \eqref{2.4}, \eqref{2.5} and \eqref{3.15}, \eqref{3.16} one also obtains,
\begin{align}
\ga_D u^D &= \ga_D u_0^D - \ga_D\big(H_{\Om}^D-zI_\Om\big)^{-1}Vu_0^D \no
\\&=
f - \ga_D \big(H_{0,\Om}^D-zI_\Om\big)^{-1}
\big[I_{\Om}+V\big(H_{0,\Om}^D-zI_\Om\big)^{-1}\big]^{-1} Vu_0^D=f,
\\
\wti\ga_N u^N &= \wti\ga_N u_0^N -
\wti\ga_N\big(H_{\Om}^N-zI_\Om\big)^{-1}Vu_0^N \no
\\&=
g - \ga_N \big(H_{0,\Om}^N-zI_\Om\big)^{-1}
\big[I_{\Om}+V\big(H_{0,\Om}^N-zI_\Om\big)^{-1}\big]^{-1} Vu_0^N=g.
\end{align}

Finally, \eqref{3.11} and \eqref{3.12} follow from \eqref{3.3},
\eqref{3.4}, \eqref{3.13}, \eqref{3.14}, and the resolvent
identity,
\begin{align}
u^D (z) &= \big[I_\Om - \big(H_\Om^D-zI_\Om\big)^{-1} V\big] \big[-\ga_N
\big(\big(H_{0,\Om}^D-zI_\Om\big)^{-1}\big)^*\big]^*f \no
\\ &=
-\big[\ga_N \big(\big(H_{0,\Om}^D- zI_\Om\big)^{-1}\big)^*
\big[I_\Om - \big(H_\Om^D-zI_\Om\big)^{-1} V\big]^*\big]^*f \no
\\ &=
-\big[\ga_N \big(\big(H_\Om^D-zI_\Om\big)^{-1}\big)^*\big]^*f,
\\[1mm]
u^N (z) &= \big[I_\Om - \big(H_\Om^N-zI_\Om\big)^{-1} V\big] \big[\ga_D
\big(\big(H_{0,\Om}^N-zI_\Om\big)^{-1}\big)^*\big]^*g \no
\\ &=
\big[\ga_D \big(\big(H_{0,\Om}^N-zI_\Om\big)^{-1}\big)^*
\big[I_\Om - \big(H_\Om^N-zI_\Om\big)^{-1} V\big]^*\big]^*g \no
\\ &=
\big[\ga_D \big(\big(H_\Om^N-zI_\Om\big)^{-1}\big)^*\big]^*g.
\end{align}
Analytic continuation with respect to $z$ then permits one to remove the additional condition $z \notin \si\big(H_{0,\Om}^D\big)$ in the case of the Dirichlet problem, and
the additional condition $z \notin \si\big(H_{0,\Om}^N\big)$ in the context of the Neumann problem.
\end{proof}

Assuming Hypothesis \ref{h2.1}, we now introduce the
Dirichlet-to-Neumann map $M_{0,\Om}^{D}(z)$ 
associated with $(-\Delta-z)$ on $\Om$, as follows,
\begin{align}
M_{0,\Om}^{D}(z) \colon
\begin{cases}
H^1(\dOm) \to \LdOm,
\\
\hspace*{10mm} f \mapsto -\wti\ga_N u_0^D,
\end{cases}  \quad z\in\bbC\big\backslash\si\big(H_{0,\Om}^D\big), \lb{3.20}
\end{align}
where $u_0^D$ is the unique solution of
\begin{align}
(-\Delta-z)u_0^D = 0 \,\text{ on }\Om, \quad u_0^D\in
H^{3/2}(\Om), \quad \ga_D u_0^D = f \,\text{ on }\dOm,
\end{align}
Similarly, assuming Hypothesis \ref{h3.2}, we introduce the
Dirichlet-to-Neumann map $M_\Om^{D}(z)$, 
associated with $(-\Delta+V-z)$ on $\Om$, by 
\begin{align}
M_\Om^{D}(z) \colon
\begin{cases}
H^1(\dOm) \to \LdOm,
\\
\hspace*{10mm} f \mapsto -\wti\ga_N u^D,
\end{cases}  \quad z\in\bbC\big\backslash\si\big(H_{\Om}^D\big), \lb{3.22}
\end{align}
where $u^D$ is the unique solution of
\begin{align}
(-\Delta+V-z)u^D = 0 \,\text{ on }\Om,  \quad u^D \in
H^{3/2}(\Om), \quad \ga_D u^D= f \,\text{ on }\dOm.
\end{align}
By Theorems \ref{t3.1} and \ref{t3.3} one obtains
\begin{equation}
M_{0,\Om}^{D}(z), M_\Om^{D}(z) \in \cB\big(H^1(\partial\Om), \LdOm \big).
\end{equation}

In addition, assuming Hypothesis \ref{h2.1}, we introduce the Neumann-to-Dirichlet map  $M_{0,\Om}^{N}(z)$ associated with $(-\Delta-z)$ on $\Om$, as follows,
\begin{align}
M_{0,\Om}^{N}(z) \colon \begin{cases} \LdOm \to H^1(\dOm),
\\
\hspace*{20.5mm} g \mapsto \ga_D u_0^N, \end{cases}  \quad
z\in\bbC\big\backslash\si\big(H_{0,\Om}^N\big), \lb{3.24}
\end{align}
where $u_0^N$ is the unique solution of
\begin{align}
(-\Delta-z)u_0^N = 0 \,\text{ on }\Om, \quad u_0^N\in
H^{3/2}(\Om), \quad \wti\ga_N u_0^N = g \,\text{ on }\dOm,
\end{align}
Similarly, assuming Hypothesis \ref{h3.2}, we introduce the Neumann-to-Dirichlet map 
$M_\Om^{N}(z)$ associated with $(-\Delta+V-z)$ on $\Om$ by 
\begin{align}
M_\Om^{N}(z) \colon \begin{cases}
\LdOm \to H^1(\dOm),
\\
\hspace*{20.5mm} g \mapsto \ga_D u^N,
\end{cases}  \quad
z\in\bbC\big\backslash\si\big(H_{\Om}^N\big), \lb{3.26}
\end{align}
where $u^N$ is the unique solution of
\begin{align}
(-\Delta+V-z)u^N = 0 \,\text{ on }\Om,  \quad u^N \in
H^{3/2}(\Om), \quad \wti\ga_N u^N= g \,\text{ on }\dOm.
\end{align}
Again, by Theorems \ref{t3.1} and \ref{t3.3} one obtains
\begin{equation}
M_{0,\Om}^{N}(z), M_\Om^{N}(z) \in \cB\big(\LdOm, H^1(\partial\Om) \big).
\end{equation}

Moreover, under the assumption of Hypothesis \ref{h2.1} for $M_{0,\Om}^D(z)$ and 
$M_{0,\Om}^N(z)$, and under the assumption of Hypothesis \ref{h3.2} for 
$M_{\Om}^D(z)$ and $M_{\Om}^N(z)$, one infers the following equalities: 
\begin{align}
M_{0,\Om}^{N}(z) &= - M_{0,\Om}^{D}(z)^{-1}, \quad
z\in\bbC\big\backslash\big(\si\big(H_{0,\Om}^D\big)\cup\si\big(H_{0,\Om}^N\big)\big),
\lb{3.28}
\\
M_{\Om}^{N}(z) &= - M_{\Om}^{D}(z)^{-1}, \quad
z\in\bbC\big\backslash\big(\si\big(H_{\Om}^D\big)\cup\si\big(H_{\Om}^N\big)\big),
\lb{3.29}
\intertext{and} 
M^{D}_{0,\Om}(z) &= \wti\gamma_N\big[\gamma_N
\big(\big(H^D_{0,\Om} - zI_\Om\big)^{-1}\big)^*\big]^*, \quad
z\in\bbC\big\backslash\si\big(H_{0,\Om}^D\big), \lb{3.30}
\\
M^{D}_{\Om}(z) &= \wti\gamma_N\big[\gamma_N \big(\big(H^D_{\Om} -
zI_\Om\big)^{-1}\big)^*\big]^*, \quad
z\in\bbC\big\backslash\si\big(H_{\Om}^D\big),
\lb{3.31}
\\
M^{N}_{0,\Om}(z) &= \gamma_D\big[\gamma_D
\big(\big(H^N_{0,\Om} - zI_\Om\big)^{-1}
\big)^*\big]^*, \quad
z\in\bbC\big\backslash\si\big(H_{0,\Om}^N\big), \lb{3.32}
\\
M^{N}_{\Om}(z) &= \gamma_D\big[\gamma_D
\big(\big(H^N_{\Om} - zI_\Om\big)^{-1}\big)^*\big]^*, \quad
z\in\bbC\big\backslash\si\big(H_{\Om}^N\big).
\lb{3.33}
\end{align}

The representations \eqref{3.30}--\eqref{3.33} provide a convenient point of departure for proving the operator-valued Herglotz property of
$M^{D}_{\Om}$ and $M^{N}_{\Om}$. We will return to this topic in a future paper.

Next, we note that the above formulas \eqref{3.30}--\eqref{3.33}
may be used as alternative definitions of the Dirichlet-to-Neumann
and Neumann-to-Dirichlet maps. In particular, we will next use
\eqref{3.31} and \eqref{3.33} to extend the above definition of the
operators $M^{D}_{\Om}(z)$ and $M^{N}_{\Om}(z)$ to a more general
setting. This is done in the following two lemmas.

\begin{lemma} \lb{l3.4}
Assume Hypothesis \ref{h2.6}. Then the following boundedness 
properties hold: 
\begin{align}
& \ga_N \big(H^D_{\Om}-zI_\Om\big)^{-1} \in\cB\big(\LOm, \LdOm\big), \quad
z\in\bbC\big\backslash\si\big(H_{\Om}^D\big), \lb{3.38a}
\\
& \ga_D \big(H^N_{\Om}-zI_\Om\big)^{-1} \in \cB\big(\LOm, H^1(\dOm)\big), \quad
z\in\bbC\big\backslash\si\big(H_{\Om}^N\big), \lb{3.39a}
\\
&\big[\ga_N \big(\big(H^D_{\Om}-zI_\Om\big)^{-1}\big)^*\big]^* \in 
\cB\big(H^1(\dOm), H^{3/2}(\Om)\big), \quad
z\in\bbC\big\backslash\si\big(H_{\Om}^D\big), \lb{3.40a}
\\
& \big[\ga_D \big(\big(H^N_{\Om}-zI_\Om\big)^{-1}\big)^*\big]^* \in 
\cB\big(\LdOm, H^{3/2}(\Om)\big), \quad
z\in\bbC\big\backslash\si\big(H_{\Om}^N\big).   \lb{3.41a}
\end{align}
Moreover, the operators $M^{D}_{\Om}(z)$ in \eqref{3.31} and $M^{N}_{\Om}(z)$ 
in \eqref{3.33} remain well-defined and satisfy 
\begin{align}
& M^{D}_{\Om}(z) \in \cB\big(H^{1}(\dOm), \LdOm\big), \quad
z\in\bbC\big\backslash\si\big(H_{\Om}^D\big), \lb{3.42a}
\\
& M^{N}_{\Om}(z) \in \cB\big(\LdOm, H^{1}(\dOm)\big), \quad
z\in\bbC\big\backslash\si\big(H_{\Om}^N\big).  \lb{3.43a}
\end{align}
In particular, $M^{N}_{\Om}(z)$, $z\in\bbC\big\backslash\si\big(H_{\Om}^N\big)$, are compact operators in $L^2(\partial\Om; d^{n-1}\sigma)$. 
\end{lemma}
\begin{proof}
We temporarily assume that
$z\in\bbC\big\backslash\big(\si\big(H_{0,\Om}^D\big)\cup\si\big(H_{\Om}^D\big)\big)$
in the case of Dirichlet Laplacian and that
$z\in\bbC\big\backslash\big(\si\big(H_{0,\Om}^N\big)\cup\si\big(H_{\Om}^N\big)\big)$
in the context of Neumann Laplacian.

Next, let $u,v$ and $\wti u, \wti v$ denote the following
factorizations of the perturbation $V$,
\begin{align}
V(x) = u(x)v(x),\quad u(x) &= \exp(i\arg(V(x)))\abs{V(x)}^{1/2},\quad
v(x)=\abs{V(x)}^{1/2}, \lb{3.44a}
\\
V(x) = \wti u(x) \wti v(x),\quad \wti u(x) &=
\exp(i\arg(V(x)))\abs{V(x)}^{p/p_1},\quad \wti
v(x)=\abs{V(x)}^{p/p_2}, \lb{3.45a}
\end{align}
where
\begin{align} \lb{3.46a}
p_1=\begin{cases} 3p/2, & n=2, \\ 4p/3, & n\geq3,\end{cases} \qquad
p_2=\begin{cases}3p, &n=2, \\ 4p, & n\geq3.\end{cases}
\end{align}
We note that Hypothesis \ref{h2.6} and \eqref{3.44a}, \eqref{3.45a}
imply
\begin{align}
\wti u \in L^{p_1}(\Om;d^nx), \; \wti v\in L^{p_2}(\Om;d^nx),
\;\text{ and }\; u,v\in L^{2p}(\Om;d^nx). \lb{3.46b}
\end{align}

It follows from the definition of the operators $H_{\Om}^D$ and
$H_{\Om}^N$ and, in particular, from \eqref{B.5} that
\begin{align}
\big(H^D_{\Om}-zI_\Om\big)^{-1} &= \big(H^D_{0,\Om}-zI_\Om\big)^{-1}
- \big(H^D_{0,\Om}-zI_\Om\big)^{-1}v
\Big[I_\Om+\ol{u\big(H^D_{0,\Om}-zI_\Om\big)^{-1}v}\,\Big]^{-1}
u\big(H^D_{0,\Om}-zI_\Om\big)^{-1} \no
\\ &=
\big(H^D_{0,\Om}-zI_\Om\big)^{-1} - \big(H^D_{0,\Om}-zI_\Om\big)^{-1}
\wti v \Big[I_\Om+ \ol{\wti u \big(H^D_{0,\Om}-zI_\Om\big)^{-1} \wti
v}\,\Big]^{-1} \wti u \big(H^D_{0,\Om}-zI_\Om\big)^{-1}, \lb{3.47a}
\\
\big(H^N_{\Om}-zI_\Om\big)^{-1} &= \big(H^N_{0,\Om}-zI_\Om\big)^{-1}
- \big(H^N_{0,\Om}-zI_\Om\big)^{-1}v
\Big[I_\Om+\ol{u\big(H^N_{0,\Om}-zI_\Om\big)^{-1}v}\,\Big]^{-1}
u\big(H^N_{0,\Om}-zI_\Om\big)^{-1} \no
\\&=
\big(H^N_{0,\Om}-zI_\Om\big)^{-1} - \big(H^N_{0,\Om}-zI_\Om\big)^{-1}
\wti v \Big[I_\Om+ \ol{\wti u \big(H^N_{0,\Om}-zI_\Om\big)^{-1} \wti
v}\,\Big]^{-1} \wti u \big(H^N_{0,\Om}-zI_\Om\big)^{-1}. \lb{3.48a}
\end{align}

Next, we establish a number of boundedness properties that will imply
\eqref{3.38a}--\eqref{3.43a}. First, note that it follows from
Hypothesis \ref{2.6} and \eqref{3.46a} that
$p_1=\frac32p>2>2n/3$, $p_2=3p>4$ for $n=2$ and
$p_1=\frac43p>2n/3$, $p_2=4p>2n$ for $n\geq3$. Then,
utilizing Lemma \ref{l2.3}, one obtains
\begin{align}
& \wti u \big(H^D_{0,\Om}-zI_\Om\big)^{-1} \in \cB\big(\LOm,\LOm\big), \quad
z\in\bbC\big\backslash\si\big(H^D_{0,\Om}\big), \lb{3.51a}
\\
& \wti u \big(H^N_{0,\Om}-zI_\Om\big)^{-1} \in \cB\big(\LOm,\LOm\big), \quad
z\in\bbC\big\backslash\si\big(H^N_{0,\Om}\big), \lb{3.52a}
\\
& \big(H^D_{0,\Om}-zI_\Om\big)^{-\frac{1-\eps}{4}} \wti v \in 
\cB\big(\LOm,\LOm\big), \quad z\in\bbC\big\backslash\si\big(H^D_{0,\Om}\big),
\lb{3.53a}
\\
& \big(H^N_{0,\Om}-zI_\Om\big)^{-\frac{1-\eps}{4}} \wti v \in 
\cB\big(\LOm,\LOm\big), \quad z\in\bbC\big\backslash\si\big(H^N_{0,\Om}\big),
\lb{3.54a}
\end{align}
and, utilizing Lemma \ref{l2.2} and the inclusion \eqref{incl-xxx}, one
obtains for $\eps\in(0,1-2n/p_2)$, 
\begin{align}
& \big(H^D_{0,\Om}-zI_\Om\big)^{-\frac{3+\eps}{4}} \in 
\cB\big(\LOm,H^{\frac{3+\eps}{2}}(\Om)\big), \quad
z\in\bbC\big\backslash\si\big(H^D_{0,\Om}\big), \lb{3.54b}
\\
& \big(H^N_{0,\Om}-zI_\Om\big)^{-\frac{3+\eps}{4}} \in 
\cB\big(\LOm,H^{\frac{3+\eps}{2}}(\Om)\big), \quad
z\in\bbC\big\backslash\si\big(H^N_{0,\Om}\big). \lb{3.54c}
\end{align}
In addition, 
\begin{align}
& \big(H^D_{0,\Om}-zI_\Om\big)^{-\frac{3+\eps}{4}} : \LOm \to
H^{\frac{3+\eps}{2}}(\Om) \hookrightarrow H^{3/2}(\Om), \quad
z\in\bbC\big\backslash\si\big(H^D_{0,\Om}\big), \lb{3.55a}
\\
& \big(H^N_{0,\Om}-zI_\Om\big)^{-\frac{3+\eps}{4}} : \LOm \to
H^{\frac{3+\eps}{2}}(\Om) \hookrightarrow H^{3/2}(\Om), \quad
z\in\bbC\big\backslash\si\big(H^N_{0,\Om}\big). \lb{3.56a}
\end{align}
In particular, one concludes from
\eqref{3.53a}--\eqref{3.56a} that
\begin{align}
& \big(H^D_{0,\Om}-zI_\Om\big)^{-1} \wti v \in 
\cB\big( \LOm,H^{3/2}(\Om)\big),
\quad z\in\bbC\big\backslash\si\big(H^D_{0,\Om}\big), \lb{3.57a}
\\
& \big(H^N_{0,\Om}-zI_\Om\big)^{-1} \wti v \in 
\cB\big(\LOm,H^{3/2}(\Om)\big),
\quad z\in\bbC\big\backslash\si\big(H^N_{0,\Om}\big). \lb{3.58a}
\end{align}
In addition, it follows from \eqref{3.53a}--\eqref{3.56a}, the definition
of $\ga_N$ \eqref{2.3}, inclusion \eqref{incl-xxx}, and Lemma
\ref{lA.6} that
\begin{align}
& \ga_N \big(H^D_{0,\Om}-zI_\Om\big)^{-1} \wti v \in 
\cB\big(\LOm,\LdOm\big),
\quad z\in\bbC\big\backslash\si\big(H^D_{0,\Om}\big), \lb{3.59a}
\\
& \ga_D \big(H^N_{0,\Om}-zI_\Om\big)^{-1}\wti v \in 
\cB\big(\LOm,H^1(\dOm)\big),
\quad z\in\bbC\backslash\si(H^N_{0,\Om}\big). \lb{3.60a}
\end{align}

Next, it follows from Theorem \ref{t3.1} that
\begin{align}
& \big[\ga_N\big(H^D_{0,\Om}-\ol{z}I_\Om\big)^{-1}\big]^* \in 
\cB\big(H^1(\dOm),H^{3/2}(\Om)\big), \quad
z\in\bbC\big\backslash\si\big(H^D_{0,\Om}\big), \lb{3.61a}
\\
& \big[\ga_D \big(H^N_{0,\Om}-\ol{z}I_\Om\big)^{-1}\big]^* \in 
\cB\big(\LdOm,H^{3/2}(\Om)\big), \quad 
z\in\bbC\big\backslash\si\big(H^N_{0,\Om}\big).
\lb{3.62a}
\end{align}
Then, employing the Sobolev embedding theorem
\begin{align}
H^{3/2}(\Om)\hookrightarrow L^q(\Om;d^nx)
\end{align}
with $q$ satisfying $1/q=(1/2)-(1/p_1)>(1/2)-3/(2n)$, $n\geq2$,
and the fact that $\wti u \in L^{p_1}(\Om;d^nx)$, one obtains the
following boundedness properties from \eqref{3.61a} and \eqref{3.62a},
\begin{align}
& \wti u\big[\ga_N \big(H^D_{0,\Om}-\ol{z}I_\Om\big)^{-1}\big]^* \in
\cB\big(H^1(\dOm),\LOm\big), \quad
z\in\bbC\big\backslash\si\big(H^D_{0,\Om}\big), \lb{3.65a}
\\
& \wti u\big[\ga_D \big(H^N_{0,\Om}-\ol{z}I_\Om\big)^{-1}\big]^* \in 
\cB\big(\LdOm,\LOm\big), \quad z\in\bbC\big\backslash\si\big(H^N_{0,\Om}\big).
\lb{3.66a}
\end{align}

Moreover, it follows from Theorem \ref{tB.3} that the operators
$\big[I_\Om+ \wti u \big(H^D_{0,\Om}-zI_\Om\big)^{-1} \wti v\big]$
and $\big[I_\Om+ \wti u \big(H^N_{0,\Om}-zI_\Om\big)^{-1} \wti
v\big]$ are boundedly invertible on $\LOm$ for
$z\in\bbC\big\backslash\big(\si\big(H_{0,\Om}^D\big)\cup\si\big(H_{\Om}^D\big)\big)$
and
$z\in\bbC\big\backslash\big(\si\big(H_{0,\Om}^N\big)\cup\si\big(H_{\Om}^N\big)\big)$,
respectively, that is, the following operators are bounded, 
\begin{align}
& \big[I_\Om+ \wti u \big(H^D_{0,\Om}-zI_\Om\big)^{-1} \wti v\big]^{-1}
\in \cB\big(\LOm,\LOm\big), \quad
z\in\bbC\big\backslash\big(\si\big(H_{0,\Om}^D\big)\cup\si\big(H_{\Om}^D\big)\big),
\lb{3.67a}
\\
& \big[I_\Om+ \wti u \big(H^N_{0,\Om}-zI_\Om\big)^{-1} \wti v\big]^{-1}
\in \cB\big(\LOm,\LOm\big), \quad
z\in\bbC\big\backslash\big(\si\big(H_{0,\Om}^N\big)\cup\si\big(H_{\Om}^N\big)\big).
\lb{3.68a}
\end{align}

Finally, combining \eqref{3.47a}--\eqref{3.68a}, one obtains the
assertions of Lemma \ref{l3.4} as follows: \eqref{3.38a} follows from \eqref{3.47a},
\eqref{3.51a}, \eqref{3.59a}, \eqref{3.67a}; \eqref{3.39a} follows
from \eqref{3.48a}, \eqref{3.52a}, \eqref{3.60a}, \eqref{3.68a};
\eqref{3.40a} follows from \eqref{3.47a}, \eqref{3.57a},
\eqref{3.65a}, \eqref{3.67a}; \eqref{3.41a} follows from
\eqref{3.48a}, \eqref{3.58a}, \eqref{3.66a}, \eqref{3.68a};

Thus, by \eqref{3.20}, \eqref{3.59a}, \eqref{3.65a}, and
\eqref{3.67a}, we may introduce the operator 
\begin{equation}
M^D_{\Om}(z) = M^D_{0,\Om}(z) - \ga_N \big(H^D_{0,\Om}-zI_\Om\big)^{-1}
\wti v \Big[I_\Om+ \ol{\wti u \big(H^D_{0,\Om}-zI_\Om\big)^{-1} \wti v}\,\Big]^{-1} \wti
u \big(\ga_N\big(H^D_{0,\Om}-\ol{z}I_\Om\big)^{-1}\big)^*, \lb{3.49a}
\end{equation}
and observe that it satisfies \eqref{3.42a}. In addition, \eqref{3.47a} shows that \eqref{3.31} remains in effect under Hypothesis \ref{h2.6}. 

Similarly, by \eqref{3.24}, \eqref{3.60a}, \eqref{3.66a}, and
\eqref{3.68a}, we may introduce the operator 
\begin{equation}
M^N_{\Om}(z) = M^N_{0,\Om}(z) - \ga_D \big(H^N_{0,\Om}-zI_\Om\big)^{-1}
\wti v \Big[I_\Om+ \ol{\wti u \big(H^N_{0,\Om}-zI_\Om\big)^{-1} \wti v}\,\Big]^{-1} \wti
u \big(\ga_D \big(H^N_{0,\Om}-\ol{z}I_\Om\big)^{-1}\big)^*,    \lb{3.50a}
\end{equation}
and observe that it satisfies \eqref{3.43a}. In addition, \eqref{3.48a} shows that 
\eqref{3.33} remains in effect under Hypothesis \ref{h2.6}. Moreover, since 
$H^1(\partial\Om)$ embeds compactly into $L^2(\partial\Om; d^{n-1}\sigma)$ 
(cf.\ \eqref{EQ1} and \cite[Proposition~2.4]{MM07}), 
$M^{N}_{\Om}(z)$, $z\in\bbC\big\backslash\si\big(H_{\Om}^N\big)$, are compact 
operators in $L^2(\partial\Om; d^{n-1}\sigma)$. 

Finally, formulas \eqref{3.31} and \eqref{3.33}
together with analytic continuation with respect to $z$ then permit one to
remove the additional restrictions $z\notin\si\big(H_{0,\Om}^D\big)$ and
$z\notin\si\big(H_{0,\Om}^N\big)$, respectively. 
\end{proof}

Actually, one can go a step further and allow an additional
perturbation $V_1\in L^\infty(\Om;d^nx)$ of $H^D_{\Om}$ and $H^N_{\Om}$,
\begin{align}
H^D_{1,\Om} &= H^D_{\Om} + V_1, \quad \dom(H^D_{1,\Om})=\dom(H^D_{\Om}),
\lb{3.70a}  \\
H^N_{1,\Om} &=H^N_{\Om} + V_1, \quad \dom(H^N_{1,\Om})=\dom(H^N_{\Om}).
\lb{3.70b}
\end{align}
Defining the Dirichlet-to-Neumann and Neumann-to-Dirichlet operators
$M^{D}_{1,\Om}$ and $M^{N}_{1,\Om}$ in an analogous fashion as in
\eqref{3.31} and \eqref{3.33},
\begin{align}
&M^{D}_{1,\Om}(z) = \wti\gamma_N\big[\gamma_N
\big(\big(H^D_{1,\Om} - zI_\Om\big)^{-1}\big)^*\big]^*, \quad
z\in\bbC\big\backslash\si\big(H_{1,\Om}^D\big), \lb{3.71a}
\\
& M^{N}_{1,\Om}(z) = \gamma_D\big[\gamma_D \big(\big(H^N_{1,\Om} -
zI_\Om\big)^{-1}\big)^*\big]^*, \quad
z\in\bbC\big\backslash\si\big(H_{1,\Om}^N\big), \lb{3.72a}
\end{align}
one can then prove the following result:

\begin{lemma}
Assume Hypothesis \ref{h2.6} and let $V_1\in L^\infty(\Om;d^nx)$. Then
the operators $M^{D}_{1,\Om}(z)$ and $M^{N}_{1,\Om}(z)$ defined by
\eqref{3.71a} and \eqref{3.72a} satisfy the following boundedness properties,
\begin{align}
M^{D}_{1,\Om}(z) \in \cB\big(H^{1}(\dOm), \LdOm\big), \quad
z\in\bbC\big\backslash\si\big(H_{1,\Om}^D\big), \lb{3.73a}
\\
M^{N}_{1,\Om}(z) \in \cB\big(\LdOm, H^{1}(\dOm)\big), \quad
z\in\bbC\big\backslash\si\big(H_{1,\Om}^N\big). \lb{3.74a}
\end{align}
\end{lemma}
\begin{proof}
We temporarily assume that
$z\in\bbC\big\backslash\big(\si\big(H_{\Om}^D\big)\cup\si\big(H_{1,\Om}^D\big)\big)$
in the case of $M^{D}_{1,\Om}$ and that
$z\in\bbC\big\backslash\big(\si\big(H_{\Om}^N\big)\cup\si\big(H_{1,\Om}^N\big)\big)$
in the context of $M^{N}_{1,\Om}$.

Next, using resolvent identities and \eqref{3.70a}, \eqref{3.70b}, one computes
\begin{align}
\big(H^D_{1,\Om}-zI_\Om\big)^{-1} &= \big(H^D_{\Om}-zI_\Om\big)^{-1}
- \big(H^D_{\Om}-zI_\Om\big)^{-1} \Big[I_\Om+ V_1
\big(H^D_{\Om}-zI_\Om\big)^{-1}\,\Big]^{-1} V_1
\big(H^D_{\Om}-zI_\Om\big)^{-1}, \lb{3.75a}
\\
\big(H^N_{1,\Om}-zI_\Om\big)^{-1} &= \big(H^N_{\Om}-zI_\Om\big)^{-1}
- \big(H^N_{\Om}-zI_\Om\big)^{-1} \Big[I_\Om+ V_1
\big(H^N_{\Om}-zI_\Om\big)^{-1}\,\Big]^{-1} V_1
\big(H^N_{\Om}-zI_\Om\big)^{-1}, \lb{3.76a}
\end{align}
and hence,
\begin{align}
M^D_{1,\Om} &= M^D_{\Om} - \ga_N \big(H^D_{\Om}-zI_\Om\big)^{-1}
\Big[I_\Om+ V_1 \big(H^D_{\Om}-zI_\Om\big)^{-1}\,\Big]^{-1} V_1
\big[\ga_N\big(\big(H^D_{\Om}-zI_\Om\big)^{-1}\big)^*\big]^*,
\lb{3.77a}
\\
M^N_{1,\Om} &= M^N_{\Om} - \ga_D \big(H^N_{\Om}-zI_\Om\big)^{-1}
\Big[I_\Om+ V_1 \big(H^N_{\Om}-zI_\Om\big)^{-1}\,\Big]^{-1} V_1
\big[\ga_D\big(\big(H^N_{\Om}-zI_\Om\big)^{-1}\big)^*\big]^*.
\lb{3.78a}
\end{align}
The assertions \eqref{3.73a} and \eqref{3.74a} now follow from
\eqref{3.38a}--\eqref{3.43a} and the fact that by Theorem \ref{tB.3}, the
operators $\big[I_\Om+ V_1 \big(H^D_{\Om}-zI_\Om\big)^{-1}\big]$ and
$\big[I_\Om+V_1 \big(H^N_{\Om}-zI_\Om\big)^{-1}\big]$ are boundedly invertible on
$\LOm$ for all
$z\in\bbC\big\backslash\big(\si\big(H_{\Om}^D\big)\cup\si\big(H_{1,\Om}^D\big)\big)$
and
$z\in\bbC\big\backslash\big(\si\big(H_{\Om}^N\big)\cup\si\big(H_{1,\Om}^N\big)\big)$,
respectively. Formulas \eqref{3.71a} and \eqref{3.72a} together with
analytic continuation with respect to $z$ then permit one to remove the additional
restrictions $z\notin\si\big(H_{\Om}^D\big)$ and $z\notin\si\big(H_{\Om}^N\big)$, respectively.
\end{proof}

Weyl--Titchmarsh operators, in a spirit close to ours, have
recently been discussed by Amrein and Pearson \cite{AP04} in
connection with the interior and exterior of a ball in $\bbR^3$
and potentials $V\in L^\infty(\bbR^3;d^3x)$. For additional
literature on Weyl--Titchmarsh operators, relevant in the context
of boundary value spaces (boundary triples, etc.), we refer, for
instance, to \cite{ABMN05}, \cite{BL06}, \cite{BMN06},
\cite{BMN00}, \cite{BMN02}, \cite{BM04}, \cite{DM91}, \cite{DM95},
\cite{GKMT01}, \cite[Ch.\ 3]{GG91}, \cite{MM06}, \cite{Ma04}, \cite{MPP07}
\cite{Pa87}, \cite{Pa02}. For applications of the
Dirichlet-to-Neumann map to Borg--Levinson-type inverse spectral
problems we refer to \cite{Ch90}, \cite{NSU88}, \cite{PS02},
\cite{Sa05}, \cite{SU86}, \cite{SU87} (see also \cite{KLW05} for
an alternative approach based on the boundary control method). The
inverse problem of detecting the number of connected components
(i.e., the number of holes) in $\partial\Omega$ using the
high-energy spectral asymptotics of the Dirichlet-to-Neumann map
is studied in \cite{HL01},

Next, we prove the following auxiliary result, which will play a
crucial role in Theorem \ref{t4.2}, the principal result of this
paper.

\begin{lemma} \lb{l3.5}
Assume Hypothesis \ref{h2.6}. Then the following
identities hold,
\begin{align}
M_{0,\Om}^D(z) - M_\Om^D(z) &= \ol{\wti\gamma_N
\big(H^D_{\Om}-zI_\Om\big)^{-1} V \big[\gamma_N
\big(\big(H^D_{0,\Om}-zI_\Om\big)^{-1}\big)^*\big]^*}, \no
\\
&\hspace*{3.1cm}
z\in\bbC\big\backslash\big(\si\big(H_{0,\Om}^D\big)\cup\si\big(H_{\Om}^D\big)\big),
\lb{3.35}
\\
M_\Om^D(z) M_{0,\Om}^D(z)^{-1} &= I_\dOm - \ol{\wti\gamma_N
\big(H^D_{\Om}-zI_\Om\big)^{-1} V \big[\gamma_D
\big(\big(H^N_{0,\Om}-zI_\Om\big)^{-1}\big)^*\big]^*}, \no
\\
&\hspace*{2.45cm}
z\in\bbC\big\backslash\big(\si\big(H_{0,\Om}^D\big)\cup\si\big(H_{\Om}^D\big)
\cup\si\big(H_{0,\Om}^N\big)\big). \lb{3.36}
\end{align}
\end{lemma}
\begin{proof}
Let $z\in\bbC\big\backslash\big(\si\big(H_{0,\Om}^D\big)\cup\si\big(H_{\Om}^D\big)\big)$.
Then \eqref{3.35} follows from \eqref{3.30}, \eqref{3.31}, and
the resolvent identity
\begin{align}
M_{0,\Om}^D(z) - M_\Om^D(z) &=
\wti\gamma_N\big[\gamma_N\big(\big(H^D_{0,\Om}-zI_\Om\big)^{-1} -
\big(H^D_{\Om}-zI_\Om\big)^{-1}\big)^*\big]^* \no
\\ &=
\ol{\wti\gamma_N\big[\gamma_N \big(
\big(H^D_{\Om}-zI_\Om\big)^{-1}V\big(H^D_{0,\Om}-zI_\Om\big)^{-1} \big)^*\big]^*}
\\ &=
\ol{\wti\gamma_N \big(H^D_{\Om}-zI_\Om\big)^{-1} V \big[\gamma_N
\big(\big(H^D_{0,\Om}-zI_\Om\big)^{-1}\big)^*\big]^*}. \no
\end{align}
Next, let
$z\in\bbC\big\backslash\big(\si\big(H_{0,\Om}^D\big)\cup\si\big(H_{\Om}^D\big)
\cup\si\big(H_{0,\Om}^N\big)\big)$, then it follows from \eqref{3.28},
\eqref{3.32}, and \eqref{3.35} that
\begin{align} \lb{3.40}
M_\Om^D(z) M_{0,\Om}^D(z)^{-1} &= I_\dOm + \big(M_\Om^D(z) -
M_{0,\Om}^D(z)\big)M_{0,\Om}^D(z)^{-1} \no
\\ &=
I_\dOm + \big(M_{0,\Om}^D(z) - M_{\Om}^D(z)\big)M_{0,\Om}^N(z)
\no
\\ &=
I_\dOm + \ol{\wti\gamma_N \big(H^D_{\Om}-zI_\Om\big)^{-1} V \big[\gamma_N
\big(\big(H^D_{0,\Om}-zI_\Om\big)^{-1}\big)^*\big]^*}
\\ &\quad \times
\gamma_D\big[\gamma_D
\big(\big(H^N_{0,\Om}-zI_\Om\big)^{-1}\big)^*\big]^*. \no
\end{align}
Let $g\in\LdOm$. Then by Theorem \ref{t3.1},
\begin{align}
u=\big[\gamma_D\big(\big(H^N_{0,\Om}-zI_\Om\big)^{-1}\big)^*\big]^*g
\lb{3.41}
\end{align}
is the unique solution of
\begin{align}
(-\Delta-z)u = 0 \,\text{ on }\Om, \quad u\in H^{3/2}(\Om),
\quad \wti\ga_N u = g \,\text{ on }\dOm.
\end{align}
Setting $f=\ga_D u \in H^1(\dOm)$ and utilizing Theorem
\ref{t3.1} once again, one obtains
\begin{align}
u &= -\big[\ga_N \big(H_{0,\Om}^D-\ol{z}I_\Om\big)^{-1}\big]^*f \no
\\ &=
-\big[\gamma_N \big(\big(H^D_{0,\Om}-zI_\Om\big)^{-1}\big)^*\big]^*
\gamma_D\big[\gamma_D
\big(\big(H^N_{0,\Om}-zI_\Om\big)^{-1}\big)^*\big]^*g. \lb{3.43}
\end{align}
Thus, it follows from \eqref{3.41} and \eqref{3.43} that
\begin{align}
\big[\gamma_N \big(\big(H^D_{0,\Om}-zI_\Om\big)^{-1}\big)^*\big]^*
\gamma_D\big[\gamma_D
\big(\big(H^N_{0,\Om}-zI_\Om\big)^{-1}\big)^*\big]^* =
-\big[\gamma_D\big(\big(H^N_{0,\Om}-zI_\Om\big)^{-1}\big)^*\big]^*.
\lb{3.44}
\end{align}
Finally, insertion of \eqref{3.44} into \eqref{3.40} yields
\eqref{3.36}.
\end{proof}

It follows from \eqref{4.24}--\eqref{4.29a}, $\widetilde \gamma_N$ can be replaced by
$\gamma_N$ on the right-hand side of \eqref{3.35} and \eqref{3.36}.

We note that the right-hand side (and hence the left-hand side) of \eqref{3.36}
permits an analytic continuation to $z\in \si\big(H_{0,\Om}^D\big)$ as long as
$z\notin \big(\si\big(H_{\Om}^D\big)\cup\si\big(H_{0,\Om}^N\big)\big)$.

\section{A Multi-Dimensional Variant of a Formula due to Jost and Pais}
\label{s4}

In this section we prove our multi-dimensional variants of the Jost and Pais formula as discussed in the introduction.

We start with an elementary comment on determinants which, however, lies at the heart of the matter of our multi-dimensional variant of the one-dimensional Jost and Pais result. Suppose $A \in \cB(\cH_1, \cH_2)$, $B \in \cB(\cH_2, \cH_1)$ with $A B \in \cB_1(\cH_2)$ and $B A \in \cB_1(\cH_1)$. Then,
\begin{equation}
\det (I_{\cH_2}-AB) = \det (I_{\cH_1}-BA).   \lb{4.0}
\end{equation}
Equation \eqref{4.0} follows from the fact that all nonzero eigenvalues of $AB$ and $BA$ coincide including their algebraic multiplicities. The latter fact, in turn, can be derived from the formula
\begin{equation}
A(BA - z I_{\cH_1})^{-1} B= I_{\cH_2} + z(AB - z I_{\cH_2})^{-1}, \quad
z \in \bbC \backslash (\si(AB)\cup\si(BA))
\end{equation}
(and its companion with $A$ and $B$ interchanged), as discussed in
detail by Deift \cite{De78}.

In particular, $\cH_1$ and $\cH_2$ may have different dimensions.
Especially, one of them may be infinite and the other finite, in
which case one of the two determinants in \eqref{4.0} reduces to a
finite determinant. This case indeed occurs in the original
one-dimensional case studied by Jost and Pais \cite{JP51} as
described in detail in \cite{GM03} and the references therein.
In the proof of Theorem \ref{t4.1} below, the role of $\cH_1$ and
$\cH_2$ will be played by $L^2(\Om; d^n x)$ and
$L^2(\partial\Om;d^{n-1} \sigma)$, respectively. In the context
of KdV flows and reflectionless (i.e., generalizations of
soliton-type) potentials represented as Fredholm determinants, a
reduction of such determinants (in some cases to finite determinants)
has also been studied by Kotani \cite{Ko04}, relying on certain
connections to stochastic analysis.

We start with an auxiliary lemma which is of independent interest in the area
of modified Fredholm determinants.

\begin{lemma} \lb{l4.1}
Let $\cH$ be a separable, complex Hilbert space, and
assume $A,B\in\cB_k(\cH)$ for some fixed $k\in\bbN$. Then there exists a
polynomial $T_k(\cdot,\cdot)$ in $A$ and $B$ with $T_k(A,B)\in\cB_1(\cH)$, such that the following formula holds
\begin{align} \lb{4.3a}
\det{}_k((I_{\cH}-A)(I_{\cH}-B)) =
\det{}_k(I_{\cH}-A)\det{}_k(I_{\cH}-B)e^{\tr(T_k(A,B))}.
\end{align}
Moreover, $T_k(\cdot,\cdot)$ is unique up to cyclic permutations of its
terms, and an explicit formula for $T_k$ may be derived from the representation 
\begin{align} \lb{4.4a}
T_k(A,B) = \sum_{m=k}^{2k-2} P_m(A,B),
\end{align}
where $P_m(\cdot,\cdot)$, $m=1,\dots,2k-2$, denote homogeneous polynomials in
$A$ and $B$ of degree $m$ $($i.e., each term of $P_m(A,B)$ contains precisely the total number $m$ of $A$'s and $B$'s$)$ that one obtains after
rearranging the following expression in powers of $t$,
\begin{align} \lb{4.5a}
\sum_{j=1}^{k-1}\frac{1}{j}\big((tA+tB-t^2AB)^j-(tA)^j-(tB)^j\big) =
\sum_{m=1}^{2k-2} t^m P_m(A,B), \quad t\in\bbR.
\end{align}
In particular, computing $T_k(A,B)$ from \eqref{4.4a} and \eqref{4.5a}, and subsequently using cyclic permutations to simplify the resulting expressions, then yields for the terms $T_k(A,B)$ in \eqref{4.3a}
\begin{align} \lb{4.6a}
T_1(A,B) =& \;0, \no \\
T_2(A,B) =& - AB, \no \\
T_3(A,B) =& - A^2B - AB^2 + \frac12 ABAB, \no \\
T_4(A,B) =& - A^3B - AB^3 -\frac12 ABAB - A^2B^2 + A^2BAB + AB^2AB -
\frac13 ABABAB, \\
T_5(A,B) =& - A^4B - AB^4 - A^3B^2 - A^2B^3 - A^2BAB - AB^2AB +
A^3BAB + AB^3AB \no \\ &+ A^2B^2AB + A^2BAB^2 + \frac23ABABAB +
\frac12 A^2BA^2B + \frac12 AB^2AB^2 \no \\
& - A^2BABAB -AB^2ABAB + \frac14 ABABABAB, \, \text{ etc.} \no
\end{align}
\end{lemma}
\begin{proof}
Suppose temporarily that $A,B\in\cB_1(\cH)$. Then it follows from
\cite[Theorem 9.2]{Si05} that
\begin{align}
\det{}_k((I_{\cH}-A)(I_{\cH}-B)) =
\det{}_k(I_{\cH}-A)\det{}_k(I_{\cH}-B)e^{\tr(\wti T_k(A,B))},
\end{align}
where
\begin{align}
\wti T_k(A,B) =
\sum_{j=1}^{k-1}\frac{1}{j}\big((A+B-AB)^j-(A)^j-(B)^j\big),
\end{align}
and hence, by \eqref{4.5a}
\begin{align} \lb{4.9a}
\wti T_k(A,B) = \sum_{m=1}^{2k-2} P_m(A,B).
\end{align}
Since $\tr(\cdot)$ is linear and invariant under cyclic permutation
of its argument, it remains to show that $T_k(A,B)$ in \eqref{4.4a}
and $\wti T_k(A,B)$ in \eqref{4.9a} are equal up to cyclic
permutations of their terms, that is, to show that $P_m(A,B)$ vanish for
$m=1,\dots,k-1$ after a finite number of cyclic permutations of their terms.

Let $\wti P_m(\cdot,\cdot)$, $m\geq1$, denote a sequence of polynomials in
$A$ and $B$, obtained after rearranging the following expression in powers of 
$t\in\bbC$,
\begin{align}
\begin{split}
&\ln((I_{\cH}-tA)(I_{\cH}-tB))-\ln(I_{\cH}-tA)-\ln(I_{\cH}-tB)
\\
&\quad =
\sum_{j=1}^{\infty}\frac{1}{j}\big((tA+tB-t^2AB)^j-(tA)^j-(tB)^j\big)
= \sum_{m=1}^{\infty} t^m \wti P_m(A,B) \, \text{ for $|t|$ sufficiently small.}
\end{split}\lb{4.10a}
\end{align}
Then it follows from \eqref{4.5a} and \eqref{4.10a} that
$P_m(A,B)=\wti P_m(A,B)$ for $m=1,\dots,k-1$, and hence, it suffices
to show that $\wti P_m(A,B)$ vanish for $m=1,\dots,k-1$ after a finite number of cyclic
permutations of their terms. The latter fact now follows from the
Baker--Campbell--Hausdorff (BCH) formula as follows: First, assume
$D, E \in \cB(\cH)$, $\cH$. Then,
\begin{equation} \lb{4.12a}
e^{t D}e^{t E} = e^{t D+t E+F(t)} \, \text{ for $|t|$ sufficiently small,}
\end{equation}
where $F(t)$ is given by a norm convergent infinite sum of certain repeated
commutators involving $D$ and $E$, as discussed, for instance, in \cite{Su77}
(cf.\ also \cite{BC04}). Explicitly, $F$ is of the form
\begin{equation}
F(t)=\sum_{\ell=2}^{\infty} t^\ell F_\ell,  \quad
F_p=\f{1}{p!}\bigg[\f{d^p}{dt^p}\ln\bigg(\sum_{j=0}^\infty \sum_{k=0}^\infty
\f{t^{j+k}}{j! k!} D^j E^k \bigg)\bigg]\bigg|_{t=0}, \quad p\in\bbN, \; p\geq 2,
\end{equation}
where
\begin{equation}
F_2=\f{1}{2}[D,E], \; F_3=\f{1}{6}[F_2,E-D], \;
F_4=\f{1}{12}[[F_2,D],E], \, \text{ etc.}
\end{equation}
That each $F_\ell$, $\ell\geq2$, is indeed at most a finite sum of commutators
follows from a formula derived by Dynkin (cf., e.g., \cite[eqs.\ (1)--(4)]{Bo89},
\cite[eqs.\ (2.5), (2.6), (3.7), (3.8)]{Ot91}).

If in addition, $D,E \in\cB_1(\cH)$, the expression for $F(t)$ is actually convergent in the $\cB_1(\cH)$-norm for $|t|$ sufficiently small. Thus, $F(t)$ vanishes after a finite number of cyclic permutations of each of its coefficients $F_n$.

Next, setting $D=\ln(I_{\cH}-tA)$, $E=\ln(I_{\cH}-tB)$ and taking the natural logarithm in \eqref{4.12a} then implies
\begin{equation}
\ln((I_{\cH}-tA)(I_{\cH}-tB))-\ln(I_{\cH}-tA)-\ln(I_{\cH}-tB) = F(t)
\end{equation}
and hence
\begin{equation}
\ln((I_{\cH}-tA)(I_{\cH}-tB))-\ln(I_{\cH}-tA)-\ln(I_{\cH}-tB) = 0
\end{equation}
after a finite number of cyclic permutations in each of the coefficients $F_\ell$ in
$F(t)=\sum_{\ell=2}^\infty t^\ell F_\ell $. Thus, by \eqref{4.10a}, each $\wti P_m(A,B)$,
$m\geq1$, vanishes after a finite number of cyclic permutations of its terms. Consequently, $P_m(A,B)$ vanish for $m=1,\dots,k-1$ after a finite number of cyclic permutations of their terms.

Finally, to remove the assumption $A,B\in\cB_1(\cH)$, one uses a standard
approximation argument of operators in $\cB_k(\cH)$ by operators in $\cB_1(\cH)$,  together with the fact that both sides of \eqref{4.3a} are well-defined for
$A,B\in\cB_k(\cH)$.
\end{proof}

Next, we prove an extension of a result in \cite{GLMZ05} to arbitrary space dimensions:

\begin{theorem} \lb{t4.1}
Assume Hypothesis \ref{h2.6}, let $k\in\bbN$, $k\geq p$, and
$z\in\bbC\big\backslash\big(\si\big(H_{\Om}^D\big)\cup
\si\big(H_{0,\Om}^D\big) \cup \si\big(H_{0,\Om}^N\big)\big)$. Then,
\begin{align}
\ol{\ga_N\big(H_{\Om}^D-zI_{\Om}\big)^{-1}V \big[\ga_D
\big(H_{0,\Om}^N-\ol{z}I_{\Om}\big)^{-1}\big]^*}
\in\cB_p\big(\LdOm\big)\subset\cB_k\big(\LdOm\big) \lb{4.2}
\end{align}
and
\begin{align}
&
\frac{\det{}_k\Big(I_{\Om}+\ol{u\big(H_{0,\Om}^N-zI_{\Om}\big)^{-1}v}\,\Big)}
{\det{}_k\Big(I_{\Om}+\ol{u\big(H_{0,\Om}^D-zI_{\Om}\big)^{-1}v}\,\Big)}  \no \\
&\quad = \det{}_k\Big(I_{\dOm} -
\ol{\ga_N\big(H_{\Om}^D-zI_{\Om}\big)^{-1}V \big[\ga_D
\big(H_{0,\Om}^N-\ol{z}I_{\Om}\big)^{-1}\big]^*} \, \Big)
\exp\big(\tr(T_k(z))\big).   \lb{4.3}
\end{align}
Here $T_k(z)\in \cB_1\big(L^2(\partial\Om; d^{n-1}\sigma)\big)$ denotes one of the cyclic permutations of the
polynomial $T_k(\cdot,\cdot)$ defined in Lemma \ref{l4.1} with the following
choice of $A=A_0(z)$ and $B=B_0(z)$, with $A_0(z)$ and $B_0(z)$ given by
\begin{align}
\begin{split} \lb{4.4}
A_0(z) &= \Big[\,\ol{\ga_D
\big(H_{0,\Om}^N-\ol{z}I_{\Om}\big)^{-1}\ol{\wti u}}\,\Big]^* \,
\ol{\ga_N \big(H_{\Om}^D -zI_{\Om}\big)^{-1}\wti
v}\in\cB_p\big(\LOm\big)\subset\cB_k\big(\LOm\big),
\\
B_0(z) &= -\ol{\wti u\big(H_{0,\Om}^D-zI_{\Om}\big)^{-1}\wti v}
\in\cB_p\big(\LOm\big)\subset\cB_k\big(\LOm\big),
\end{split}
\end{align}
and the functions $u$, $v$, $\wti u$, and $\wti v$ are given by
\begin{align}
u(x) &= \exp(i\arg(V(x)))\abs{V(x)}^{1/2},\quad
v(x)=\abs{V(x)}^{1/2},
\\
\wti u(x) &= \exp(i\arg(V(x)))\abs{V(x)}^{p/p_1},\quad \wti
v(x)=\abs{V(x)}^{p/p_2},
\end{align}
with
\begin{align} \lb{4.6}
p_1=\begin{cases} 3p/2,&n=2, \\ 4p/3, &n\geq3,\end{cases} \qquad
p_2=\begin{cases} 3p,&n=2, \\ 4p, & n\geq3, \end{cases}
\end{align}
and $V=uv=\tilde u \tilde v$. In particular,
\begin{align}
T_2(z) &= \ol{\ga_N\big(H_{0,\Om}^D-zI_{\Om}\big)^{-1}V
\big(H_{\Om}^D-zI_{\Om}\big)^{-1}V \Big[\ga_D
\big(H_{0,\Om}^N-\ol{z}I_{\Om}\big)^{-1}\Big]^*}
\in\cB_1\big(\LdOm\big). \lb{4.5}
\end{align}
\end{theorem}
\begin{proof}
From the outset we note that the left-hand side of \eqref{4.3} is
well-defined by \eqref{2.35}. Let
$z\in\bbC\big\backslash\big(\si\big(H_{\Om}^D\big) \cup
\si\big(H_{0,\Om}^D\big) \cup \si\big(H_{0,\Om}^N\big)\big)$ and
note that $\f1{p_1}+\f{1}{p_2}=\f1p$ for all $n\geq2$, and hence
$V=uv=\wti u \wti v$.

Next, we introduce
\begin{equation} \lb{4.7}
K_D(z)=-\ol{u\big(H_{0,\Om}^D-zI_{\Om}\big)^{-1}v}, \quad
K_N(z)=-\ol{u\big(H_{0,\Om}^N-zI_{\Om}\big)^{-1}v}
\end{equation}
(cf.\ \eqref{B.4}) and note that by Theorem \ref{tB.3}
\begin{align}
[I_{\Om}-K_D(z)]^{-1} \in\cB\big(\LOm\big), \quad
z\in\bbC\big\backslash\big(\si\big(H_{\Om}^D\big)\cup\si\big(H_{0,\Om}^D\big)\big).
\end{align}
Then Lemma \ref{l4.1} with $A=\wti A_0(z)$ and $B=\wti B_0(z)$
defined by
\begin{align}
\wti A_0(z) &= I_\Om - (I_\Om-K_N(z))[I_{\Om}-K_D(z)]^{-1} =
(K_N(z)-K_D(z))[I_{\Om}-K_D(z)]^{-1}, \lb{4.10}
\\
\wti B_0(z) &= K_D(z) = -\ol{u\big(H_{0,\Om}^D-zI_{\Om}\big)^{-1}v},
\lb{4.10A}
\end{align}
yields
\begin{align}
&\frac{\det{}_k\Big(I_{\Om}+\ol{u\big(H_{0,\Om}^N-zI_{\Om}\big)^{-1}v}\,\Big)}
{\det{}_k\Big(I_{\Om}+\ol{u\big(H_{0,\Om}^D-zI_{\Om}\big)^{-1}v}\,\Big)}
=
\frac{\det{}_k\big(I_{\Om}-K_N(z)\big)}{\det{}_k\big(I_{\Om}-K_D(z)\big)}
\no \\
&\quad =
\det{}_k\big(I_{\Om}-(K_N(z)-K_D(z))[I_{\Om}-K_D(z)]^{-1}\big)
\exp\big(\tr(T_k(\wti A_0(z),\wti B_0(z)))\big), \lb{4.12}
\end{align}
where $T_k(\cdot,\cdot)$ is the polynomial defined in \eqref{4.4a}.
Explicit formulas for the first few $T_k$ are computed in
\eqref{4.6a}.

Next, temporarily suppose that $V\in L^p(\Om;d^nx)\cap
L^\infty(\Om;d^nx)$. Using Lemma \ref{lA.3} (an extension of a
result of Nakamura \cite[Lemma 6]{Na01}) and Remark \ref{rA.5} (cf.\
\eqref{Na1-bis}), one finds
\begin{align}
\begin{split}
K_N(z)-K_D(z) &= \ol{u\big[\big(H_{0,\Om}^D-zI_{\Om}\big)^{-1}-
\big(H_{0,\Om}^N -zI_{\Om}\big)^{-1}\big]v}
\\ &=
\ol{u\big[\ga_D \big(H_{0,\Om}^N-\ol{z}I_{\Om}\big)^{-1}\big]^*}\,
\ol{\ga_N \big(H_{0,\Om}^D -zI_{\Om}\big)^{-1}v}
\\ &=
\Big[\,\ol{\ga_D
\big(H_{0,\Om}^N-\ol{z}I_{\Om}\big)^{-1}\ol{u}}\,\Big]^* \,
\ol{\ga_N \big(H_{0,\Om}^D -zI_{\Om}\big)^{-1}v}.
\end{split}\lb{4.13}
\end{align}
Inserting \eqref{4.13} into \eqref{4.10} and utilizing \eqref{4.7}
and the following resolvent identity which follows from \eqref{B.5},
\begin{align} \lb{4.13a}
\ol{\big(H_{\Om}^D-zI_{\Om}\big)^{-1} v} =
\ol{\big(H_{0,\Om}^D-zI_{\Om}\big)^{-1} v} \Big[I_{\Om}+\ol{
u\big(H_{0,\Om}^D-zI_{\Om}\big)^{-1} v}\,\Big]^{-1},
\end{align}
one obtains the following equality for $\wti A_0(z)$,
\begin{align}
\begin{split} \lb{4.4A}
\wti A_0(z) &= \Big[\,\ol{\ga_D
\big(H_{0,\Om}^N-\ol{z}I_{\Om}\big)^{-1}\ol{u}}\,\Big]^* \,
\ol{\ga_N \big(H_{\Om}^D -zI_{\Om}\big)^{-1}v}.
\end{split}
\end{align}
Moreover, insertion of \eqref{4.13} into \eqref{4.12} yields
\begin{align} \lb{4.14}
&\frac{\det{}_k\Big(I_{\Om}+\ol{u\big(H_{0,\Om}^N-zI_{\Om}\big)^{-1}v}\,\Big)}
{\det{}_k\Big(I_{\Om}+\ol{u\big(H_{0,\Om}^D-zI_{\Om}\big)^{-1}v}\,\Big)}
\no
\\
&\quad = \det{}_k\Big(I_{\Om} - \Big[\,\ol{\ga_D
\big(H_{0,\Om}^N-\ol{z}I_{\Om}\big)^{-1}\ol{u}}\,\Big]^* \ol{\ga_N
\big(H_{0,\Om}^D-zI_{\Om}\big)^{-1}v}
\Big[I_{\Om}+\ol{u\big(H_{0,\Om}^D-zI_{\Om}\big)^{-1}v}\,\Big]^{-1}\Big)
\\
&\qquad \times \exp\big(\tr(T_k(\wti A_0(z),\wti B_0(z)))\big). \no
\end{align}
Utilizing Corollary \ref{c2.5} with $p_1$ and $p_2$ as in
\eqref{4.6}, one finds 
\begin{align}
\ol{\ga_D \big(H_{0,\Om}^N-\ol{z}I_{\Om}\big)^{-1}\ol{u}}
&\in\cB_{p_1}\big(\LOm,\LdOm\big),
\\
\ol{\ga_N\big(H_{0,\Om}^D-zI_{\Om}\big)^{-1}v}
&\in\cB_{p_2}\big(\LOm,\LdOm\big),
\end{align}
and hence,
\begin{align}
\Big[\,\ol{\ga_D
\big(H_{0,\Om}^N-\ol{z}I_{\Om}\big)^{-1}\ol{u}}\,\Big]^* \ol{\ga_N
\big(H_{0,\Om}^D-zI_{\Om}\big)^{-1}v} &\in \cB_p\big(\LOm\big)
\subset\cB_k\big(\LOm\big),
\\
\ol{\ga_N \big(H_{0,\Om}^D-zI_{\Om}\big)^{-1}v} \Big[\,\ol{\ga_D
\big(H_{0,\Om}^N-\ol{z}I_{\Om}\big)^{-1}\ol{u}}\,\Big]^* &\in
\cB_p\big(\LdOm\big) \subset\cB_k\big(\LdOm\big).
\end{align}
Then, using the fact that
\begin{align}
\Big[I_{\Om}+\ol{u\big(H_{0,\Om}^D-zI_{\Om}\big)^{-1}v}\,\Big]^{-1}
\in \cB\big(\LOm\big), \quad z\in\bbC\big\backslash
\big(\si\big(H_{\Om}^D\big)\cup\si\big(H_{0,\Om}^D\big)\big),
\end{align}
one applies the idea expressed in formula \eqref{4.0} and rearranges
the terms in \eqref{4.14} as follows:
\begin{align} \lb{4.20}
&\frac{\det{}_k\Big(I_{\Om}+\ol{u\big(H_{0,\Om}^N-zI_{\Om}\big)^{-1}v}\,\Big)}
{\det{}_k\Big(I_{\Om}+\ol{u\big(H_{0,\Om}^D-zI_{\Om}\big)^{-1}v}\,\Big)}
\\
&\quad = \det{}_k\Big(I_{\dOm} - \ol{\ga_N
\big(H_{0,\Om}^D-zI_{\Om}\big)^{-1}v}
\Big[I_{\Om}+\ol{u\big(H_{0,\Om}^D-zI_{\Om}\big)^{-1}v}\,\Big]^{-1}
\Big[\, \ol{ \ga_D
\big(H_{0,\Om}^N-\ol{z}I_{\Om}\big)^{-1}\ol{u}}\,\Big]^*\Big) \no
\\
&\qquad \times \exp\big(\tr(T_k(\wti A_0, \wti B_0))\big). \no
\end{align}
Similarly, using the cyclicity property of $\tr(\cdot)$, one
rearranges $T_k \big(\wti A_0(z),\wti B_0(z)\big)$ to get an operator on
$\LdOm$ which in the following we denote by $T_k(z)$. This is always
possible since each term of $T_k \big(\wti A_0(z),\wti B_0(z)\big)$ has at
least one factor of $\wti A_0(z)$. Then using equalities
\eqref{4.4}, \eqref{4.10A}, \eqref{4.4A}, and $uv=\ti u \ti v$, one
concludes that $T_k(z)$ is a cyclic permutation of $T_k(A_0,B_0)$
with $A_0(z)$ and $B_0(z)$ given by \eqref{4.4}. In particular,
rearranging $T_2 \big(\wti A_0(z),\wti B_0(z)\big)=-\wti A_0(z) \wti B_0(z)$
or equivalently $T_2(A_0(z),B_0(z))=-A_0(z)B_0(z)$, one obtains
$T_2(z)=-\wti B_0(z) \wti A_0(z) = -B_0(z) A_0(z)$, and hence
equality \eqref{4.5}. Thus, \eqref{4.3}, subject to the extra
assumption $V\in L^p(\Om;d^n x)\cap L^\infty(\Om;d^n x)$, follows
from \eqref{4.13a} and \eqref{4.20}.

Finally, assuming only $V\in L^p(\Om;d^n x)$ and utilizing Theorem
\ref{tB.3}, Lemma \ref{l2.3}, and Corollary \ref{c2.5} once again,
one obtains
\begin{align}
\Big[I_{\Om}+\ol{\wti u\big(H_{0,\Om}^D-zI_{\Om}\big)^{-1} \wti
v}\,\Big]^{-1} &\in \cB\big(\LOm\big), \lb{4.24}
\\
\wti u\big(H_{0,\Om}^D-zI_{\Om}\big)^{-p/p_1} &\in
\cB_{p_1}\big(\LOm\big), \lb{4.25}
\\
\wti v\big(H_{0,\Om}^D-zI_{\Om}\big)^{-p/p_2} &\in
\cB_{p_2}\big(\LOm\big), \lb{4.26}
\\
\ol{\ga_D \big(H_{0,\Om}^N-zI_{\Om}\big)^{-1}\wti u} &\in
\cB_{p_1}\big(\LOm,\LdOm\big), \lb{4.27}
\\
\ol{\ga_N\big(H_{0,\Om}^D-zI_{\Om}\big)^{-1}\wti v} &\in
\cB_{p_2}\big(\LOm,\LdOm\big), \lb{4.28}
\end{align}
and thus,
\begin{equation}
\ol{\wti u\big(H_{0,\Om}^D-zI_{\Om}\big)^{-1}\wti v} \in
\cB_p\big(\LOm\big)\subset\cB_k\big(\LOm\big).  \lb{4.29}
\end{equation}
Relations \eqref{4.24}--\eqref{4.29} together with the following
resolvent identity that follows from \eqref{B.5},
\begin{align}
\ol{\big(H_{\Om}^D-zI_{\Om}\big)^{-1}\wti v} =
\ol{\big(H_{0,\Om}^D-zI_{\Om}\big)^{-1}\wti v} \Big[I_{\Om}+\ol{\wti
u\big(H_{0,\Om}^D-zI_{\Om}\big)^{-1}\wti v}\,\Big]^{-1},   \lb{4.29a}
\end{align}
prove the $\cB_k$-property \eqref{4.2}, \eqref{4.4}, and \eqref{4.5},
and hence, the left- and the right-hand sides of \eqref{4.3} are
well-defined for $V\in L^p(\Om;d^nx)$. Thus, using \eqref{2.8},
\eqref{2.27}, \eqref{2.28}, the continuity of $\det{}_k(\cdot)$ with
respect to the $\cB_k$-norm $\|\cdot\|_{\cB_k\big(\LOm\big)}$, the
continuity of $\tr(\cdot)$ with respect to the trace norm
$\|\cdot\|_{\cB_1\big(\LOm\big)}$, and an approximation of $V\in
L^p(\Om;d^nx)$ by a sequence of potentials $V_j \in
L^p(\Om;d^nx)\cap L^\infty(\Om;d^nx)$, $j\in\bbN$, in the norm of
$L^p(\Om;d^nx)$ as $j\uparrow\infty$, then extends the result from
$V\in L^p(\Om;d^nx)\cap L^\infty(\Om;d^nx)$ to $V\in L^p(\Om;d^nx)$.
\end{proof}

Given these preparations, we are ready for the principal result of this paper, the multi-dimensional analog of Theorem \ref{t1.2}:

\begin{theorem} \lb{t4.2}
Assume Hypothesis \ref{h2.6}, let $k\in\bbN$, $k\geq p$, and
$z\in\bbC\big\backslash\big(\si\big(H_{\Om}^D\big)\cup
\si\big(H_{0,\Om}^D\big) \cup \si\big(H_{0,\Om}^N\big)\big)$. Then,
\begin{equation}
M_{\Om}^{D}(z)M_{0,\Om}^{D}(z)^{-1} - I_{\partial\Om} = -
\ol{\ga_N\big(H_{\Om}^D-zI_{\Om}\big)^{-1} V
\Big[\ga_D \big(H_{0,\Om}^N-\ol{z}I_{\Om}\big)^{-1}\Big]^*} \in
\cB_k\big(L^2(\partial\Om; d^{n-1}\sigma)\big)
\end{equation}
and
\begin{align}
&
\frac{\det{}_k\Big(I_{\Om}+\ol{u\big(H_{0,\Om}^N-zI_{\Om}\big)^{-1}v}\,\Big)}
{\det{}_k\Big(I_{\Om}+\ol{u\big(H_{0,\Om}^D-zI_{\Om}\big)^{-1}v}\,\Big)} \no \\
& \quad = \det{}_k\Big(I_{\dOm} -
\ol{\ga_N\big(H_{\Om}^D-zI_{\Om}\big)^{-1} V
\big[\ga_D\big(H_{0,\Om}^N-\ol{z}I_{\Om}\big)^{-1}\big]^*}\,\Big)
\exp\big(\tr(T_k(z))\big)   \lb{4.30}  \\
& \quad = \det{}_k\big(M_{\Om}^{D}(z)M_{0,\Om}^{D}(z)^{-1}\big)
\exp\big(\tr(T_k(z))\big)   \lb{4.31}
\end{align}
with $T_k(z)$ defined in Theorem \ref{t4.1}.
\end{theorem}
\begin{proof}
The result follows from combining Lemma \ref{l3.5} and Theorem \ref{t4.1}.
\end{proof}

\begin{remark}  \lb{r4.4}
Assume Hypothesis \ref{h2.6}, let $k\in\bbN$, $k\geq p$, and
$z\in\bbC\big\backslash\big(\si\big(H_{\Om}^N\big)\cup
\si\big(H_{0,\Om}^D\big) \cup \si\big(H_{0,\Om}^N\big)\big)$. Then,
\begin{equation}
M_{0,\Om}^{N}(z)^{-1}M_{\Om}^{N}(z) - I_{\partial\Om} =
\ol{\ga_N \big(H_{0,\Om}^D-zI_{\Om}\big)^{-1}V \Big[\ga_D \big(\big(H_{\Om}^N-z
I_{\Om}\big)^{-1}\big)^*\Big]^*} \in
\cB_k\big(L^2(\partial\Om; d^{n-1}\sigma)\big)   \lb{4.32}
\end{equation}
and one can also prove the following analog of \eqref{4.30}:
\begin{align}
&\frac{\det{}_k\Big(I_{\Om}+\ol{u(H_{0,\Om}^D-zI_{\Om})^{-1}v}\,\Big)}
{\det{}_k\Big(I_{\Om}+\ol{u(H_{0,\Om}^N-zI_{\Om})^{-1}v}\,\Big)} \no \\
&\quad = \det{}_k\Big(I_{\dOm} +
\ol{\ga_N(H_{0,\Om}^D-zI_{\Om})^{-1}V \big[\ga_D((H_{\Om}^N-z
I_{\Om})^{-1})^*\big]^*}\,\Big) \exp\big(\tr(T_k(z))\big), \lb{4.33} \\
& \quad = \det{}_k\big(M_{0,\Om}^{N}(z)^{-1} M_{\Om}^{N}(z)\big)
\exp\big(\tr(T_k(z))\big)   \lb{4.34}
\end{align}
where $T_k(z)$ denotes one of the cyclic permutations of the
polynomial $T_k(A,B)$ defined in Lemma \ref{l4.1} with the following
choice of $A=A_1(z)$ and $B=B_1(z)$,
\begin{align}
\begin{split} \no
A_1(z) &= -\Big[\,\ol{\ga_D
\big(H_{\Om}^N-\ol{z}I_{\Om}\big)^{-1}\ol{\wti u}}\,\Big]^* \,
\ol{\ga_N \big(H_{0,\Om}^D -zI_{\Om}\big)^{-1}\wti
v}\in\cB_p\big(\LOm\big)\subset\cB_k\big(\LOm\big),
\\
B_1(z) &= -\ol{\wti u\big(H_{0,\Om}^N-zI_{\Om}\big)^{-1}\wti v}
\in\cB_p\big(\LOm\big)\subset\cB_k\big(\LOm\big),
\end{split}
\end{align}
and the functions $u$, $v$, $\wti u$, and $\wti v$ are given by
\begin{align}
u(x) &= \exp(i\arg(V(x)))\abs{V(x)}^{1/2},\quad
v(x)=\abs{V(x)}^{1/2},
\\
\wti u(x) &= \exp(i\arg(V(x)))\abs{V(x)}^{p/p_1},\quad \wti
v(x)=\abs{V(x)}^{p/p_2},
\end{align}
with
\begin{align}
p_1=\begin{cases} 3p/2, &n=2, \\ 4p/3, &n\geq3,\end{cases} \qquad
p_2=\begin{cases}3p, &n=2, \\ 4p, & n\geq3, \end{cases}
\end{align}
and $V=uv=\tilde u \tilde v$. In particular,
\begin{align}
T_2(z) &= -\ol{\ga_N\big(H_{0,\Om}^D-zI_{\Om}\big)^{-1}V
\big(H_{\Om}^N-zI_{\Om}\big)^{-1}V
\Big[\ga_D\big(H_{0,\Om}^N-\ol{z}I_{\Om}\big)^{-1}\Big]^*}.
\end{align}
\end{remark}

\begin{remark} \lb{r4.5}
It seems tempting at this point to turn to an abstract version of Theorem \ref{t4.2} using  the notion of boundary value spaces (see, e.g., \cite{BL06}, \cite{DM91},
\cite{DM95}, \cite[Ch.\ 3]{GG91} and the references therein). However, the analogs of the necessary mapping and trace ideal properties as recorded in 
Sections \ref{s2} and \ref{s3} do not seem to be available at the present time for general self-adjoint extensions of a densely defined, closed symmetric operator (respectively, maximal accretive extensions of closed accretive operators) in a separable complex Hilbert space. For this reason we decided to start with the special, but important case of multi-dimensional Schr\"odinger operators.
\end{remark}

\vspace*{.5mm}

A few comments are in order at this point:

The sudden appearance of the exponential term $\exp(\tr(T_k(z)))$ in
\eqref{4.30}, \eqref{4.31}, and \eqref{4.32}, when compared to the one-dimensional
case, is due to the necessary use of the modified determinant
${\det}_k(\cdot)$, $k\geq 2$, in Theorems \ref{t4.1} and \ref{t4.2}.

As mentioned in the introduction, the multi-dimensional extension
\eqref{4.30} of \eqref{1.16}, under the stronger hypothesis $V\in
L^2(\Om; d^n x)$, $n=2,3$, first appeared in \cite{GLMZ05}. However,
the present results in Theorem \ref{t4.2} go decidedly beyond those
in \cite{GLMZ05} in the following sense: \\
$(i)$ the class of domains $\Omega$ permitted by Hypothesis \ref{h2.1}
is substantiallly expanded as compared to \cite{GLMZ05}. \\
$(ii)$ For $n=2,3$, the conditions on $V$ satisfying Hypothesis \ref{h2.6} are
now nearly optimal by comparison with the Sobolev inequality
(cf.\ Cheney \cite{Ch84}, Reed and Simon \cite[Sect.\ IX.4]{RS75},
Simon \cite[Sect.\ I.1]{Si71}).  \\
$(iii)$ The multi-dimensional extension \eqref{4.31} of \eqref{1.17}
invoking Dirichlet-to-Neumann maps is a new (and the most significant)
result in this paper.  \\
$(iv)$ While the results in \cite{GLMZ05} were confined to dimensions
$n=2,3$, all results in this paper are now derived in the general
case $n\in\bbN$, $n\geq 2$.

The principal reduction in Theorem \ref{t4.2} reduces (a ratio of)
modified Fredholm determinants associated with operators in $L^2(\Om;
d^n x)$ on the left-hand side of \eqref{4.30} to modified Fredholm
determinants associated with operators in $L^2(\partial\Om; d^{n-1}
\sigma)$ on the right-hand side of \eqref{4.30} and especially, in
\eqref{4.31}. This is the analog of the reduction described in the
one-dimensional context of Theorem \ref{t1.2}, where $\Om$
corresponds to the half-line $(0,\infty)$ and its boundary
$\partial\Om$ thus corresponds to the one-point set $\{0\}$.

In the context of elliptic operators on smooth $k$-dimensional
manifolds, the idea of reducing a ratio of zeta-function regularized
determinants to a calculation over the $k-1$-dimensional boundary has
been studied by Forman \cite{Fo87}. He also pointed out that
if the manifold consists of an interval, the special  case of a pair
of boundary points then permits one to reduce the zeta-function
regularized determinant to the determinant of a finite-dimensional
matrix. The latter case is of course an analog of the one-dimensional
Jost and Pais formula mentioned in the introduction (cf.\ Theorems
\ref{t1.1} and  \ref{t1.2}). Since then, this topic has been further
developed in various directions and we refer, for instance, to
Burghelea, Friedlander, and Kappeler \cite {BFK91}, \cite {BFK92},
\cite {BFK93}, \cite{BFK95}, Carron \cite{Ca02}, Friedlander \cite{Fr05}, 
Guillarmou and Guillop\'e \cite{GG07}, M\"uller \cite{Mu98}, Okikiolu 
\cite{Ok95}, \cite{Ok95a}, Park and Wojciechowski \cite{PW05}, \cite{PW05a}, 
and the references therein.

\medskip

Combining Theorems \ref{t4.2} and  \ref{tB.3} yields the following applications
of \eqref{4.30} and \eqref{4.33}:

\begin{theorem}  \lb{t4.6}
Assume Hypothesis \ref{h2.6} and $k\in\bbN$, $k\geq p$. \\
$(i)$One infers that
\begin{align}
\begin{split}
&\text{for all } \, z\in\bbC\big\backslash \big(\si\big(H_{\Om}^D\big)\cup
\si\big(H_{0,\Om}^D\big) \cup \si\big(H_{0,\Om}^N\big)\big), \text{ one has }
z\in\si\big(H_{\Om}^N\big)  \\
&\quad \text{if and only if } \, \det{}_k\Big( I_{\dOm} -
\ol{\ga_N \big(H_{\Om}^D-zI_{\Om}\big)^{-1}V
\big[\ga_D \big(H_{0,\Om}^N-\ol{z}I_{\Om}\big)^{-1}\big]^*}\, \Big)=0.  \lb{4.49}
\end{split}
\end{align}
$(ii)$ Similarly, one infers that
\begin{align}
\begin{split}
& \text{for all } \, z \in \bbC \big\backslash
\big(\si\big(H_{\Om}^N\big)\cup\si\big(H_{0,\Om}^N\big)\cup\si\big(H_{0,\Om}^D\big)\big),
\text{ one has } z\in\si\big(H_{\Om}^D\big)  \\
&\quad \text{if and only if }\, \det{}_k\Big(I_{\dOm} +
\ol{\ga_N \big(H_{0,\Om}^D-zI_{\Om}\big)^{-1}V \big[\ga_D \big(\big(H_{\Om}^N-z
I_{\Om}\big)^{-1}\big)^*\big]^*} \,\Big)=0.   \lb{4.50}
\end{split}
\end{align}
\end{theorem}
\begin{proof}
By the Birman--Schwinger principle, as discussed in Theorem \ref{tB.3}, for
any $k\in\bbN$ such that $k\geq p$ and
$z\in\bbC\big\backslash \big(\si\big(H_{\Om}^D\big)\cup \si\big(H_{0,\Om}^D\big) \cup
\si\big(H_{0,\Om}^N\big)\big)$, one has
\begin{equation}
z\in\si\big(H_{\Om}^N\big) \, \text{ if and only if } \,
\det{}_k\Big(I_{\Om}+ \ol{u\big(H_{0,\Om}^N -zI_{\Om}\big)^{-1}v}\,\Big)=0.
\end{equation}
Thus, \eqref{4.49} follows from \eqref{4.30}. In the same manner, \eqref{4.50}
follows from \eqref{4.33}.
\end{proof}

We conclude with another application to eigenvalue counting functions in the case where $H^D_{\Omega}$ and $H^N_{\Omega}$ are self-adjoint and have purely discrete spectra (i.e., empty essential spectra). To set the stage we introduce the following assumptions:

\begin{hypothesis}  \lb{h4.7} 
In addition to assuming Hypothesis \ref{h2.6} suppose that $V$ is real-valued and that $H^D_{\Omega}$ and $H^N_{\Omega}$ have purely discrete spectra.
\end{hypothesis}

\begin{remark}  \lb{r4.8} ${}$ \\
$(i)$ Real-valuedness of $V$ implies self-adjointness of $H^D_{\Omega}$ and 
$H^N_{\Omega}$ as noted in \eqref{B.11}. \\
$(ii)$ Since $\partial\Omega$ is assumed to be compact, purely discrete spectra of 
$H^D_{0,\Omega}$ and $H^N_{0,\Omega}$, that is, compactness of their resolvents 
(cf.\ \cite[Sect.\ XIII.14]{RS78}), is equivalent to $\Omega$ being bounded. Indeed, if 
$\Omega$ had an unbounded component, then one can construct Weyl sequences which would yield nonempty essential spectra of $H^D_{0,\Omega}$ and 
$H^N_{0,\Omega}$. On the other hand, $H^D_{0,\Omega}$ has empty essential spectrum for any bounded open set $\Omega \subset \bbR^n$ as discussed in the Corollary to \cite[Theorem XIII.73]{RS78}. Similarly, $H^N_{0,\Omega}$ has empty essential spectrum for any bounded open set $\Omega$ satisfying the segment property 
as discussed in Corollary 1 to \cite[Theorem XIII.75]{RS78}. Since any bounded Lipschitz domain satisfies the segment property (cf.\ \cite[Sect,\ 1.2.2]{Gr85}), any bounded domain $\Omega$ satisfying Hypothesis \ref{h2.1} yields a purely discrete spectrum of $H^N_{0,\Omega}$. \\
$(iii)$ We recall that $V$ is relatively form compact with respect to $H^D_{0,\Omega}$ and $H^N_{0,\Omega}$, that is, 
\begin{equation}
v\big(H^D_{0,\Om} - z I_{\Om}\big)^{-1/2}, \, v\big(H^N_{0,\Om} - z I_{\Om}\big)^{-1/2} 
\in \cB_{\infty}\big(L^2(\Om; d^n x)\big)
\end{equation}
for all $z$ in the resolvent sets of $H^D_{0,\Omega}$, respectively, $H^N_{0,\Omega}$ 
(in fact, much more is true as recorded in \eqref{2.31}  and \eqref{2.32} since 
$\cB_\infty$ can be replaced by $\cB_{2p}$). By \eqref{3.47a} and \eqref{3.48a} this yields that the difference of the resolvents of $H^D_{\Omega}$ and $H^N_{\Omega}$ is compact (in fact, it even lies in $\cB_{p}\big(L^2(\Om; d^n x)\big)$). By a variant of Weyl's theorem (cf., e.g., \cite[Theorem XIII.14]{RS78}), one concludes that $H^D_{\Omega}$ and $H^N_{\Omega}$ have empty essential spectrum if and only if $H^D_{0,\Omega}$ and $H^N_{0,\Omega}$ have (cf.\ \cite[Problem 39, p.\ 369]{RS78}). Thus, by part $(ii)$ of this remark, the assumption that $H^D_{\Omega}$ and $H^N_{\Omega}$ have purely discrete spectra in Hypothesis \ref{h4.7} can equivalently be replaced by the assumption that $\Omega$ is bounded (still assuming Hypothesis \ref{h2.6} and that $V$ is real-valued). 
\end{remark}

\medskip

Assuming Hypothesis \ref{h4.7}, $k\in\bbN$, $k\geq p$, we introduce (cf.\ also 
\cite{Ya07})
\begin{equation}
\xi_k(\lambda)= \begin{cases}  \pi^{-1} \Im\Big(\ln\Big(\det{}_k\Big(I_{\Om}
+\ol{u (H_{0,\Om} - \lambda I_{\Om})^{-1}v}\,\Big)\Big)\Big), 
& \lambda \in (e_0,\infty) \backslash (\sigma(H_{\Omega})\cup 
\sigma(H_{0,\Omega})), \\
0, & \lambda < e_0,
\end{cases}    \lb{4.57}  
\end{equation}
where 
\begin{equation}
e_0 = \inf(\sigma(H_{\Omega}), \sigma(H_{0,\Omega})), 
\end{equation}
and $H_{\Omega}$ and $H_{0,\Omega}$ temporarily abbreviate $H^D_{\Omega}$ and $H^D_{0, \Omega}$ in the case of Dirichlet boundary conditions on 
$\partial\Omega$ and $H^N_{\Omega}$ and $H^N_{0, \Omega}$ in the case of Neumann boundary conditions on $\partial\Omega$.  Moreover, we subsequently agree to write $\xi^D_k(\cdot)$ and $\xi^N_k(\cdot)$ for $\xi(\cdot)$ in the case of Dirichlet and  Neumann boundary conditions in $H_{\Omega}, H_{0, \Omega}$.

The branch of the logarithm in \eqref{4.57} has been fixed by putting $\xi_k(\lambda)=0$ for $\lambda$ in a neighborhood of $-\infty$. This is possible since
\begin{equation}
\lim_{\lambda\downarrow -\infty} \det{}_k\Big(I_{\Om}
+\ol{u (H_{0,\Om}- \lambda I_{\Om})^{-1}v}\,\Big) = 1.   \lb{4.58}
\end{equation}
Equation \eqref{4.58} in turn follows from Lemma \ref{l2.3} since 
\begin{equation}
\lim_{\lambda\downarrow -\infty} 
\Big\|\ol{u(H_{0,\Omega} - \lambda I_{\Omega})^{-1} 
v}\Big\|_{\cB_k(L^2(\Omega; d^n x))} = 0
\end{equation}
by applying the dominated convergence theorem to 
$\|(\abs{\cdot}^2-\lambda)^{-1/2}\|_{L^{2p}(\bbR^n;d^nx)}^2$ as $\lambda\downarrow-\infty$ in \eqref{2.8} (replacing $p$ by $2p$, $q$ by $1/2$, $f$ by $u$ and $v$, etc.). 
Since $H_{0,\Omega}$ is self-adjoint in $L^2(\Omega; d^n x)$ with purely discrete spectrum, for any $\lambda_0 \in\bbR$, we obtain the norm convergent expansion
\begin{equation}
(H_{0,\Omega} - z I_{\Omega})^{-1} \underset{z\to\lambda_0}{=} 
P_{0,\Omega, \lambda_0} (\lambda_0 -z)^{-1} 
+ \sum_{k=0}^{\infty} (-1)^k S_{0,\Omega, \lambda_0}^{k+1} (\lambda_0 -z)^{k}, \lb{4.63}
\end{equation}
where $P_{0,\Omega,\lambda_0}$ denotes the Riesz projection associated with 
$H_{0,\Omega}$ and the point $\lambda_0$, and $S_{0,\Omega,\lambda_0}$ is 
given by
\begin{equation}
S_{0,\Omega, \lambda_0} = \lim_{z\to\lambda_0} (H_{0,\Omega} - z I_{\Omega})^{-1} 
(I_{\Omega} - P_{0,\Omega, \lambda_0}),
\end{equation}
with the limit taken in the topology of $\cB(L^2(\Omega;d^n x))$. 
Hence, $S_{0,\Omega, \lambda_0} P_{0,\Omega, \lambda_0} = 
P_{0,\Omega,\lambda_0} S_{0,\Omega,\lambda_0} = 0$. If, in fact, $\lambda_0$ is a (necessarily discrete) eigenvalue of $H_{0,\Omega}$, then $P_{0,\Omega,\lambda_0}$ is the projection onto the corresponding eigenspace of $H_{0,\Omega}$ and the dimension of its range equals the multiplicity of the eigenvalue $\lambda_0$, denoted by
\begin{equation}
n_{0,\lambda_0} = \dim(\ran(P_{0,\Omega, \lambda_0})).
\end{equation}
We recall that all eigenvalues of $H_{0,\Omega}$ are semisimple, that is, their geometric and algebraic multiplicities coincide, since $H_{0,\Omega}$ is assumed to be self-adjoint. If $\lambda_0$ is not in the spectrum of $H_{0,\Omega}$ then, of course, $P_{0,\Omega,\lambda_0} =0$ and $n_{0,\lambda_0} =0$. In exactly, the same manner, and in obvious notation, one then also obtains
\begin{equation}
(H_{\Omega} - z I_{\Omega})^{-1} \underset{z\to\lambda_0}{=} 
P_{\Omega, \lambda_0} (\lambda_0 -z)^{-1} 
+ \sum_{k=0}^{\infty} (-1)^k S_{\Omega, \lambda_0}^{k+1} (\lambda_0 -z)^{k}  \lb{4.66}
\end{equation}
and 
\begin{equation}
n_{\lambda_0} = \dim(\ran(P_{\Omega, \lambda_0})).
\end{equation}

In the following we denote half-sided limits by
\begin{equation}
f(x_+)= \lim_{\varepsilon\downarrow 0} f(x+\varepsilon), \quad 
f(x_-) = \lim_{\varepsilon\uparrow 0} f(x-\varepsilon), \quad x\in\bbR.
\end{equation}
Moreover, we denote by $N_{H_{\Omega}}(\lambda)$ (respectively, 
$N_{H_{0,\Omega}}(\lambda)$), $\lambda\in\bbR$, the right-continuous function on $\bbR$ which counts the number of eigenvalues of $H_{\Omega}$ (respectively, $H_{0,\Omega}$) less than or equal to $\lambda$, counting multiplicities. 

\begin{lemma} \lb{l4.8}  
Assume Hypothesis \ref{h4.7} and let $k\in\bbN$, $k\geq p$. Then $\xi_k$ equals a fixed integer on any open interval in 
$\bbR\backslash(\sigma(H_{\Omega})\cup\sigma(H_{0,\Omega}))$. Moreover, for any 
$\lambda \in \bbR$, 
\begin{equation}
\xi_k(\lambda_+) - \xi_k(\lambda_-) = - (n_{\lambda} - n_{0,\lambda}),   \lb{4.69}
\end{equation}
and hence $\xi_k$ is piecewise integer-valued on $\bbR$ and normalized to vanish on 
$(-\infty, e_0)$ such that
\begin{equation}
\xi_k(\lambda) = -[N_{H_{\Omega}}(\lambda) - N_{H_{0,\Omega}}(\lambda)], \quad 
\lambda \in \bbR \backslash (\sigma(H_{\Omega})\cup\sigma(H_{0,\Omega})).   \lb{4.70}
\end{equation}
\end{lemma}
\begin{proof} 
Introducing the unitary operator $S$ in $L^2(\Omega; d^n x)$ of multiplication by the function $\sgn(V)$,
\begin{equation}
(Sf)(x) = \sgn(V(x)) f(x), \quad f \in L^2(\Omega; d^n x)   \lb{4.71}
\end{equation}
such that $Su = uS = v$, $Sv = vS = u$, $S^2 = I_{\Omega}$, one computes for 
$\lambda \in \bbR\backslash \sigma(H_{0,\Omega})$,
\begin{align}
\ol{\det{}_k\Big(I_{\Om}+\ol{u(H_{0,\Om} - \lambda I_{\Om})^{-1}v}\,\Big)} &= 
\det{}_k\Big(I_{\Om} +\ol{v(H_{0,\Om} - \lambda I_{\Om})^{-1}u}\,\Big)  \no \\
&= \det{}_k\Big(I_{\Om}+\ol{Su(H_{0,\Om} - \lambda I_{\Om})^{-1}vS}\,\Big)  \no \\
&= \det{}_k\Big(I_{\Om}+\ol{u(H_{0,\Om} - \lambda I_{\Om})^{-1}v}\,\Big),  
\lb{4.72}
\end{align} 
that is, $\det{}_k\Big(I_{\Om}+\ol{u(H_{0,\Om} - \lambda I_{\Om})^{-1}v}\,\Big)$ is real-valued for $\lambda \in \bbR\backslash \sigma(H_{0,\Omega})$. (Here the bars either denote complex conjugation, or the operator closure, depending on the context in which they are used.) Together with the 
Birman--Schwinger principle as expressed in Theorem \ref{tB.3}, this proves that $\xi_k$ equals a fixed integer on any open interval in 
$\bbR\backslash(\sigma(H_{\Omega})\cup\sigma(H_{0,\Omega}))$. 

Next, we note that  for 
$z\in \bbC\backslash(\sigma(H_{\Omega})\cup\sigma(H_{0,\Omega}))$, 
\begin{align}
\begin{split}
& - \f{d}{dz} \ln \Big(\det{}_k\Big(I_{\Om}
+\ol{u\big(H_{0,\Om}- z I_{\Om}\big)^{-1}v}\,\Big)\Big) = 
\tr\bigg((H_{\Omega} - z I_{\Omega})^{-1} - (H_{0,\Omega} - z I_{\Omega})^{-1}  \\
& \hspace*{1.9cm} 
- \sum_{\ell=1}^{k-1} (-1)^{\ell} \ol{(H_{0,\Omega} - z I_{\Omega})^{-1}v}  
\Big[\ol{u(H_{0,\Omega} - z I_{\Omega})^{-1}v} \Big]^{\ell-1}
u (H_{0,\Omega} - z I_{\Omega})^{-1}\bigg),   \lb{4.73}
\end{split}
\end{align}
which represents just a slight extension of the result recorded in \cite{Ya07}. Insertion of 
\eqref{4.63} and \eqref{4.66} into \eqref{4.73} then yields that for any 
$\lambda_0\in\bbR$, 
\begin{align}
- \f{d}{dz} \ln \Big(\det{}_k\Big(I_{\Om}
+\ol{u(H_{0,\Om}- z I_{\Om})^{-1}v}\,\Big)\Big) & \underset{z\to\lambda_0}{=}
\tr(P_{\Omega, \lambda_0} - P_{0,\Omega, \lambda_0}) (\lambda_0 -z)^{-1} +
\sum_{\ell=-k}^{\infty} c_{\ell} (\lambda_0 -z)^{\ell}   \no \\
& \underset{z\to\lambda_0}{=}
[n_{\lambda_0} - n_{0,\lambda_0}] (\lambda_0 -z)^{-1} +
\sum_{\ell=-k}^{\infty} c_{\ell} (\lambda_0 -z)^{\ell},    \lb{4.74} 
\end{align}
where 
\begin{equation}
c_{\ell} \in \bbR, \;\; \ell \in\bbZ, \, \ell \geq k, \, \text{ and } \,  c_{-1} =0.   \lb{4.75}
\end{equation}
That $c_{\ell} \in \bbR$ is clear from the real-valuedness of $V$ and the self-adjointness of $H_{\Omega}$ and $H_{0,\Omega}$ by expanding the $(\ell -1)$th power of  
$\ol{u(H_{0,\Omega} - z I_{\Omega})^{-1}v}$ in \eqref{4.73}. To demonstrate that 
$c_{-1}$ actually vanishes, that is, that the term proportional to $(\lambda_0 -z)^{-1}$ cancels in the sum $\sum_{\ell=-k}^{\infty} c_{\ell} (\lambda_0 -z)^{\ell}$ in \eqref{4.74}, we temporarily introduce $u_m = P_m u$, $v_m = v P_m$, where $\{P_m\}_{m\in\bbN}$ is a family of orthogonal projections in $L^2(\Omega; d^n x)$ satisfying
\begin{equation}
P_m^2=P_m=P_m^*, \quad \dim(\ran(P_m))=m, \quad \ran(P_m) \subset \dom(v), 
\; \; m\in\bbN, \quad \slim_{m\uparrow\infty} P_m = I_{\Omega},
\end{equation}
where $\slim$ denotes the limit in the strong operator topology. 
(E.g., it suffices to choose $P_m$ as appropriate spectral projections associated with 
$H_{0,\Omega}$.) In addition, we introduce $V_m = v_m u_m$ and the operator $H_{\Omega,m}$ in $L^2(\Omega; d^n x)$ by replacing $V$ by  $V_m$ in $H_{\Omega}$. Since 
\begin{equation}
V_m= (vP_m) P_m (u P_m)^*, 
\end{equation}
one obtains that $V_m$ is a trace class (in fact, finite rank) operator, that is, 
\begin{equation}
V_m \in \cB_1\big(L^2(\Om; d^n x)\big), \quad m\in\bbN.   \lb{4.77}
\end{equation}
Moreover, since by \eqref{2.31} and \eqref{2.32}, 
\begin{equation}
u(H_{0,\Omega} - z I_{\Omega})^{-1/2}, \ol{(H_{0,\Omega} - z I_{\Omega})^{-1/2} v} \in 
\cB_{2p}\big(L^2(\Omega; d^n x)\big), \quad  z\in \bbC\backslash\sigma(H_{0,\Omega}), 
\end{equation} 
one concludes that 
$\ol{P_m u (H_{0,\Omega} - z I_{\Omega})^{-1} v P_m} = 
P_m \ol{u (H_{0,\Omega} - z I_{\Omega})^{-1} v} P_m$, $m\in\bbN$, satisfies
\begin{align}
 \lim_{m\uparrow\infty} \big\|\ol{P_m u (H_{0,\Omega} - z I_{\Omega})^{-1} v P_m} - 
u (H_{0,\Omega} - z I_{\Omega})^{-1} v \big\|_{\cB_{p}(L^2(\Omega; d^n x))} &= 0, 
\quad z\in \bbC\backslash\sigma(H_{0,\Omega}),   \lb{4.78} \\
 \lim_{m\uparrow\infty} \big\|\ol{P_m u (H_{0,\Omega} - z I_{\Omega})^{-2} v P_m} - 
u (H_{0,\Omega} - z I_{\Omega})^{-2} v \big\|_{\cB_{p}(L^2(\Omega; d^n x))} &= 0, 
\quad  z\in \bbC\backslash\sigma(H_{0,\Omega}).   \lb{4.79}
\end{align} 
Applying the formula (cf.\ \cite[p.\ 44]{Ya92})
\begin{equation}
\f{d}{dz} \ln({\det}_k (I_{\cH} - A(z)) = - \tr\big((I_{\cH} - A(z))^{-1} A(z)^{k-1} A'(z) \big), \quad  z\in\cD, 
\end{equation}
where $A(\cdot)$ is analytic in some open domain $\cD\subseteq\bbC$ with respect to 
the $\cB_k(\cH)$-norm, $\cH$ a separable complex Hilbert space, one obtains for 
$z\in \bbC\backslash(\sigma(H_{\Omega})\cup\sigma(H_{0,\Omega}))$, 
\begin{align}
& - \f{d}{dz} \ln \Big({\det}_k \Big(I_{\Omega} 
+ \ol{u (H_{0,\Omega} - z I_{\Omega})^{-1} v} \, \Big)\Big)    \no \\
& \quad = (-1)^k \tr\Big(\Big[I_{\Omega} 
+ \ol{u (H_{0,\Omega} - z I_{\Omega})^{-1} 
v} \, \Big]^{-1} \Big[\,\ol{u (H_{0,\Omega} - z I_{\Omega})^{-1} v} \, \Big]^{k-1} 
\ol{u (H_{0,\Omega} - z I_{\Omega})^{-2} v} \, \Big),   \lb{4.81}  \\
& - \f{d}{dz} \ln\Big({\det}_k \Big(I_{\Omega} 
+ \ol{P_m u (H_{0,\Omega} - z I_{\Omega})^{-1} v P_m} \, \Big)\Big)  \no \\
& \quad = (-1)^k \tr \Big(\Big[I_{\Omega} + \ol{P_m u (H_{0,\Omega} 
- z I_{\Omega})^{-1} v P_m} \, \Big]^{-1} 
\Big[\ol{P_m u (H_{0,\Omega} - z I_{\Omega})^{-1} v P_m}\,\Big]^{k-1}  \lb{4.82} \\
& \hspace*{8cm} \times 
\ol{P_m u (H_{0,\Omega} - z I_{\Omega})^{-2} v P_m } \, \Big), 
\quad m\in\bbN.    \no 
\end{align} 
Combining equations \eqref{4.78}, \eqref{4.79} and \eqref{4.81}, \eqref{4.82} then yields 
\begin{align}
& \lim_{m\uparrow\infty} \f{d}{dz} \ln\Big({\det}_k \Big(I_{\Omega} + \ol{P_m u (H_{0,\Omega} - z I_{\Omega})^{-1} v P_m} \Big)\Big) = 
\f{d}{dz} \ln\Big({\det}_k \Big(I_{\Omega} + \ol{u (H_{0,\Omega} - z I_{\Omega})^{-1} 
v} \Big)\Big),   \no \\ 
& \hspace*{9.5cm}  z\in\bbC\backslash(\sigma(H_{\Omega})\cup\sigma(H_{0,\Omega})). 
 \lb{4.83} 
\end{align}
Because of \eqref{4.83}, to prove that $c_{-1} =0$ in \eqref{4.74} (as claimed in 
\eqref{4.75}), it suffices to replace $V$ in \eqref{4.74} by $V_m$ and prove that 
$c_{m,-1} =0$ for all $m\in\bbN$ in the following equation analogous to \eqref{4.74}, 
\begin{align}
\begin{split}
& - \f{d}{dz} \ln \Big(\det{}_k\Big(I_{\Om}
+ P_m \ol{u (H_{0,\Om}- z I_{\Om})^{-1}v} P_m \Big)\Big)   \\
& \quad  \underset{z\to\lambda_0}{=}
\tr(P_{\Omega, m, \lambda_0} - P_{0,\Omega, \lambda_0}) (\lambda_0 -z)^{-1} +
\sum_{\ell=-k}^{\infty} c_{m,\ell} (\lambda_0 -z)^{\ell},  \quad m\in\bbN,    \lb{4.84} 
\end{split}
\end{align}
where 
\begin{equation}
c_{m,\ell} \in \bbR, \;\; \ell \in\bbZ, \, \ell \geq k, \;\; m\in\bbN,    \lb{4.85}
\end{equation}
and $P_{\Omega, m, \lambda_0}$ denotes the corresponding Riesz projection associated with $H_{\Omega,m}$ (obtained by replacing $V$ by  $V_m$ in 
$H_{\Omega}$) and the point $\lambda_0$.

Applying the analog of formula \eqref{4.73} to $H_{\Omega,m}$ (cf.\ again \cite{Ya07}), and noting that $P_m$ has rank $m\in\bbN$, one concludes that for 
$z\in \bbC\backslash(\sigma(H_{\Omega})\cup\sigma(H_{0,\Omega}))$, 
\begin{align}
& - \f{d}{dz} \ln \Big(\det{}_k\Big(I_{\Om}
+ P_m\ol{u (H_{0,\Om}- z I_{\Om})^{-1}v} P_m\,\Big)\Big) 
= - \f{d}{dz} \ln \Big(\det{}_k\Big(I_{\Om}
+ \ol{P_m u (H_{0,\Om}- z I_{\Om})^{-1}v P_m}\,\Big)\Big)  \no \\
& \quad = 
\tr\bigg((H_{\Omega,m} - z I_{\Omega})^{-1} - (H_{0,\Omega} - z I_{\Omega})^{-1}  \no \\
& \hspace*{1.4cm} 
- \sum_{\ell=1}^{k-1} (-1)^{\ell} \ol{(H_{0,\Omega} - z I_{\Omega})^{-1}v P_m}  
\Big[\ol{P_m u(H_{0,\Omega} - z I_{\Omega})^{-1}v P_m} \Big]^{\ell-1}
P_m u (H_{0,\Omega} - z I_{\Omega})^{-1}\bigg)  \no \\
& \quad 
= \tr\big((H_{\Omega,m} - z I_{\Omega})^{-1} - (H_{0,\Omega} - z I_{\Omega})^{-1}\big) 
- \sum_{\ell=1}^{k-1} \f{(-1)^{\ell}}{\ell} \f{d}{dz} 
\tr\bigg(\Big[\ol{P_m u (H_{0,\Omega} - z I_{\Omega})^{-1}v P_m}\Big]^{\ell}\bigg)  
\lb{4.86}  \\ 
& \quad 
= \tr\big((H_{\Omega,m} - z I_{\Omega})^{-1} - (H_{0,\Omega} - z I_{\Omega})^{-1}\big) 
\no \\
& \qquad 
+ \sum_{\ell=1}^{k-1} (-1)^{\ell} \tr\bigg(  
\Big[\ol{P_m u(H_{0,\Omega} - z I_{\Omega})^{-1}v P_m} \Big]^{\ell-1} \, 
\ol{P_m u (H_{0,\Omega} - z I_{\Omega})^{-2} v P_m} \bigg), \quad m\in\bbN.   \no
\end{align}
Here we have used the fact that by \eqref{4.77},  
\begin{equation}
- \f{d}{dz} \ln \Big(\det{}\Big(I_{\Om}
+ \ol{P_m u (H_{0,\Om}- z I_{\Om})^{-1}v P_m}\,\Big)\Big)  = 
\tr\big((H_{\Omega,m} - z I_{\Omega})^{-1} - (H_{0,\Omega} - z I_{\Omega})^{-1}\big), 
\end{equation}
for $z\in \bbC\backslash(\sigma(H_{\Omega})\cup\sigma(H_{0,\Omega}))$, and that 
(cf.\ \cite[Theorem 9.2]{Si05})
\begin{align}
\begin{split}
\f{d}{dz} \ln ({\det}_k(I_{\cH} - B(z))) & = \f{d}{dz} \ln ({\det} (I_{\cH} - B(z))) + 
\sum_{\ell=1}^{k-1} \f{1}{\ell} \f{d}{dz} \tr \big(B(z)^\ell\big)  \\
& = \f{d}{dz} \ln ({\det} (I_{\cH} - B(z))) + \sum_{\ell=1}^{k-1} \tr \big(B(z)^{\ell-1} B'(z)\big),
\quad z\in\cD, 
\end{split}
\end{align} 
where $B(\cdot)$ is analytic in some open domain $\cD \subseteq \bbC$ with respect to the $\cB_1(\cH)$-norm (with $\cH$ a separable complex Hilbert space). 

The presence of the $d/dz$-term under the sum in \eqref{4.86} proves that the only 
$(\lambda_0 - z)^{-1}$-term in \eqref{4.84}, respectively, \eqref{4.86}, as 
$z\to\lambda_0$, must originate from the trace of the resolvent difference 
\begin{equation}
\tr\big((H_{\Omega,m} - z I_{\Omega})^{-1} - (H_{0,\Omega} - z I_{\Omega})^{-1}\big) 
\underset{z\to\lambda_0}{=} \tr(P_{\Omega, m, \lambda_0} 
- P_{0,\Omega, \lambda_0}) (\lambda_0 -z)^{-1} + \Oh(1), \quad m\in\bbN.
\end{equation}
Thus we have proved that 
\begin{equation}
c_{m,-1} = 0, \quad m\in\bbN,
\end{equation}
in \eqref{4.84}. By \eqref{4.83} this finally proves 
\begin{equation}
c_{-1} = 0  
\end{equation}
in \eqref{4.74}. Equations \eqref{4.74} and \eqref{4.75} then prove \eqref{4.69}. Together with the paragraph following \eqref{4.72}, this also proves \eqref{4.70}. 
\end{proof}

Given Lemma \ref{l4.8}, Theorem \ref{t4.2} yields the following application to differences of Dirichlet and Neumann eigenvalue counting functions:

\begin{theorem} \lb{t4.9}  
Assume Hypothesis \ref{h4.7} and let $k\in\bbN$, $k\geq p$. Then, for all 
$\lambda\in\bbR\backslash \big(\si\big(H^D_{\Omega}\big) \cup 
\si\big(H^D_{0,\Omega}\big) \cup \si\big(H^N_{0,\Omega}\big)\big)$, 
\begin{align}
& \xi_k^N(\lambda) - \xi_k^D(\lambda) = [N_{H^D_{\Om}} (\lambda) - 
N_{H^D_{0, \Om}} (\lambda)] 
- [N_{H^N_{\Om}}  (\lambda) - N_{H^N_{0, \Om}}  (\lambda)]    \no \\
& \quad = \pi^{-1} \Im\Big(\ln\Big({\det}_k \Big(I_{\partial\Omega} - 
\ol{\gamma_N \big(H^D_{\Omega} - \lambda I_{\Omega}\big)^{-1} V 
\big[\gamma_D \big(H^N_{0,\Omega} - \lambda I_{\Omega}\big)^{-1}\big]^*}\,\Big)\Big)\Big) 
+ \pi^{-1} \Im(\tr(T_k(\lambda)))   \no \\
& \quad = \pi^{-1} \Im\big(\ln\big({\det}_k \big(M^D_{\Omega}(\lambda) 
M^D_{0,\Omega}(\lambda)^{-1} \big)\big)\big) + \pi^{-1} \Im(\tr(T_k(\lambda)))
\end{align}
with $T_k$ defined in Theorem \ref{t4.1}.
\end{theorem}
\begin{proof} 
This is now an immediate consequence of \eqref{4.30}, \eqref{4.31},  \eqref{4.57}, 
and \eqref{4.70}.
\end{proof}

\appendix
\section{Properties of Dirichlet and Neumann Laplacians}
\lb{sA}
\renewcommand{\theequation}{A.\arabic{equation}}
\renewcommand{\thetheorem}{A.\arabic{theorem}}
\setcounter{theorem}{0} \setcounter{equation}{0}

The purpose of this appendix is to recall some basic operator domain
properties of Dirichlet and Neumann Laplacians on sets
$\Om\subset\bbR^n$, $n\in\bbN$, $n\geq 2$, satisfying Hypothesis \ref{h2.1}. We will show
that the methods developed in \cite{GLMZ05} in the context of
$C^{1,r}$-domains, $1/2<r<1$, in fact, apply to all domains $\Om$
permitted in Hypothesis \ref{h2.1}.

In this manuscript we use the following notation for the standard
Sobolev Hilbert spaces ($s\in\bbR$),
\begin{align}
H^{s}(\bbR^n)&=\left\{U\in \cS(\bbR^n)^* \,|\,
\norm{U}_{H^{s}(\bbR^n)}^2 = \int_{\bbR^n} d^n \xi \, \big|\hatt
U(\xi)\big|^2\big(1+\abs{\xi}^{2s}\big) <\infty \right\},
\\
H^{s}(\Om)&=\left\{u\in C_0^\infty(\Om)^* \,|\, u=U|_\Om \text{
for some } U\in H^{s}(\bbR^n) \right\},
\\
H_0^{s}(\Om)&=\text{the closure of } C_0^\infty(\Om) \text{ in the
norm of } H^{s}(\Om).
\end{align}
Here $C_0^\infty(\Om)^*$ denotes the usual set of distributions on
$\Omega\subseteq \bbR^n$, $\Omega$ open and nonempty,
$\cS(\bbR^n)^*$ is the space of tempered distributions on
$\bbR^n$, and $\hatt U$ denotes the Fourier transform of $U\in
\cS(\bbR^n)^*$. It is then immediate that
\begin{equation}\label{incl-xxx}
H^{s_1}(\Omega)\hookrightarrow H^{s_0}(\Omega) \, \text{ for } \,
-\infty<s_0\leq s_1<+\infty,
\end{equation}
continuously and densely.

Next, we recall the
definition of a $C^{1,r}$-domain $\Omega\subseteq\bbR^n$, $\Om$
open and nonempty, for convenience of the reader: Let ${\mathcal
N}$ be a space of real-valued functions in $\bbR^{n-1}$.  One
calls a bounded domain $\Omega\subset\bbR^n$ of class ${\mathcal
N}$ if there exists a finite open covering $\{{\mathcal
O}_j\}_{1\leq j\leq N}$ of the boundary $\partial\Omega$ of $\Om$
with the property that, for every $j\in\{1,...,N\}$, ${\mathcal
O}_j\cap\Omega$ coincides with the portion of ${\mathcal O}_j$
lying in the over-graph of a function $\varphi_j\in{\mathcal N}$
(considered in a new system of coordinates obtained from the
original one via a rigid motion). Two special cases are going to
play a particularly important role in the sequel. First, if
${\mathcal  N}$ is ${\rm Lip}\,(\bbR^{n-1})$, the space of
real-valued functions satisfying a (global) Lipschitz condition in
$\bbR^{n-1}$, we shall refer to $\Omega$ as being a Lipschitz
domain; cf.\ \cite[p.\ 189]{St70}, where such domains are called
``minimally smooth''. Second, corresponding to  the case when
${\mathcal N}$ is the subspace of ${\rm Lip}\,(\bbR^{n-1})$
consisting of functions whose first-order derivatives satisfy a
(global) H\"older condition of order $r\in(0,1)$, we shall say
that $\Omega$ is of class $C^{1,r}$. The classical theorem of
Rademacher of almost everywhere differentiability of Lipschitz
functions ensures that, for any  Lipschitz domain $\Omega$, the
surface measure $d^{n-1}\sigma$ is well-defined on  $\partial\Omega$ and
that there exists an outward  pointing normal vector $\nu$ at
almost every point of $\partial\Omega$. For a Lipschitz domain
$\Omega\subset\bbR^n$ it is known that
\begin{equation}\lb{dual-xxx}
\bigl(H^{s}(\Omega)\bigr)^*=H^{-s}(\Omega), \quad - 1/2 <s< 1/2.
\end{equation}
See \cite{Tr02} for this and other related properties.

Next, assume that $\Omega\subset\bbR^n$ is the domain lying above
the graph of a function $\varphi:\bbR^{n-1}\to\bbR$ of class
$C^{1,r}$. Then for $0\leq s<1+r$, the Sobolev space
$H^s(\partial\Omega)$ consists of functions $f\in
L^2(\partial\Omega;d^{n-1}\sigma)$ such that $f(x',\varphi(x'))$,
as a function of $x'\in\bbR^{n-1}$, belongs to $H^s(\bbR^{n-1})$.
This definition is easily adapted to the case when $\Omega$ is a 
domain of class $C^{1,r}$ whose boundary is compact, by using a
smooth partition of unity. Finally, for $-1-r<s<0$, we set
$H^s(\partial\Omega)=\big(H^{-s}(\partial\Omega)\big)^*$. For additional 
background information in this context we refer, for instance, to
\cite{Au04}, \cite{Au06}, \cite[Chs.\ V, VI]{EE89},
\cite[Ch.\ 1]{Gr85}, \cite[Ch.\ 3]{Mc00}, \cite[Sect.\ I.4.2]{Wl87}.

To see that $H^1(\partial\Omega)$ embeds compactly into 
$L^2(\partial\Omega;d^{n-1}\sigma)$ one can argue as follows: Given a Lipschitz 
domain $\Omega$ in ${\mathbb{R}}^n$, we recall that 
the Sobolev space $H^1(\partial\Omega)$ is defined as the collection of 
functions in $L^2(\partial\Omega;d^{n-1}\sigma)$ with the property that the norm of 
their tangential gradient belongs to $L^2(\partial\Omega;d^{n-1}\sigma)$. 
It is essentially well-known that an equivalent characterization
is that $f\in H^1(\partial\Omega)$ if and only if the assignment 
${\mathbb{R}}^{n-1}\ni x'\mapsto (\psi f)(x',\varphi(x'))$ is in 
$H^1({\mathbb{R}}^{n-1})$ whenever $\psi\in C^\infty_0({\mathbb{R}}^n)$
and $\varphi:{\mathbb{R}}^{n-1}\to{\mathbb{R}}$ is a Lipschitz function
with the propery that if $\Sigma$ is an appropriate rotation and
translation of $\{(x',\varphi(x'))\in\bbR^n \,|\,x'\in{\mathbb{R}}^{n-1}\}$, then 
${\rm supp\,}(\psi) \cap\partial\Omega\subset\Sigma$. This appears to be folklore, 
but a proof will appear in \cite[Proposition~2.4]{MM07}. 

From the latter characterization of $H^1(\partial\Omega)$ it follows that 
any property of Sobolev spaces (of order $1$) defined in Euclidean domains,
which are invariant under multiplication by smooth, compactly supported 
functions as well as composition by bi-Lipschitz diffeomorphisms, readily 
extends to the setting of $H^1(\partial\Omega)$ (via localization and
pull-back). As a concrete example, for each Lipschitz domain $\Omega$ 
with compact boundary, one has  
\begin{equation} \label{EQ1}
H^1(\partial\Omega)\hookrightarrow L^2(\partial\Omega;d^{n-1}\sigma)
\quad\mbox{compactly}. 
\end{equation}

Going a bit further, we say that a domain $\Omega\subset{\mathbb{R}}^n$
satisfies a {\it uniform exterior ball condition}
(abbreviated by UEBC), if there exists $R>0$ with the following
property: For each $x\in\partial\Omega$, there exists
$y=y(x)\in{\mathbb{R}}^n$ such that
\begin{equation} \label{UEBC}
\overline{B(y,R)} \, \big\backslash\{x\}\subseteq
{\mathbb{R}}^n \backslash\Omega \, \text{ and } \, x\in\partial B(y,R).
\end{equation}
We recall that any $C^{1,1}$-domain (i.e., the first-order partial derivatives of the functions defining the boundary are Lipschitz) satisfies a UEBC.

Assuming Hypothesis \ref{h2.1}, we introduce the Dirichlet and Neumann
Laplacians $\wti H_{0,\Om}^D$ and $\wti H_{0,\Om}^N$ associated
with the domain $\Om$ as the unique self-adjoint operators on
$\LOm$ whose quadratic form equals $q(f,g)=\int_\Om
d^nx\,\ol{\nabla f}\cdot \nabla g$ with (form) domains given by 
$H_0^{1}(\Om)$ and $H^{1}(\Om)$, respectively. Then,
\begin{align}
\dom\big(\wti H_{0,\Om}^D\big) &= \big\{u\in H_0^{1}(\Om) \,\big|\, \text{there
exists} \,
f\in \LOm \text{ such that } \no \\
&\hspace*{1.85cm} q(u,v)=(f,v)_{\LOm} \text{ for all } v\in
H_0^{1}(\Om)\big\},  \\
\dom\big(\wti H_{0,\Om}^N\big) &= \big\{u\in H^{1}(\Om) \,\big|\, \text{there
exists}\,
f\in \LOm \text{ such that } \no \\
&\hspace*{1.85cm} q(u,v)=(f,v)_{\LOm} \text{ for all } v\in
H^{1}(\Om)\big\},
\end{align}
with $(\cdot,\cdot)_{\LOm}$ denoting the scalar product in $\LOm$.
Equivalently, we introduce the densely defined closed linear
operators
\begin{equation}
D=\nabla, \; \dom(D)=H_0^{1}(\Om) \, \text{ and } \, N=\nabla, \;
\dom(N)=H^{1}(\Om)
\end{equation}
from $\LOm$ to $\LOm^n$ and note that
\begin{equation}
\wti H_{0,\Om}^D = D^*D \, \text{ and } \, \wti H_{0,\Om}^N =
N^*N.
\end{equation}
For details we refer to \cite[Sects.\ XIII.14, XIII.15]{RS78}.
Moreover, with ${\rm div}$ denoting the divergence
operator,
\begin{equation}
\dom(D^*)=\big\{w\in L^2(\Omega;d^nx)^n\,\big|\,{\rm div} (w)\in
L^2(\Omega;d^nx)\big\},
\end{equation}
and hence,
\begin{align}
\dom\big(\wti H_{0,\Om}^D\big) &= \{u\in\dom(D) \,|\,Du\in\dom(D^*)\} \no
\\
&= \big\{u\in H_0^{1}(\Om) \,\big|\, \Delta u\in L^{2}(\Om;d^nx)\big\}.
\lb{domHD}
\end{align}
One can also define the following bounded linear map
\begin{equation}
\begin{cases} \big\{w\in
L^2(\Omega;d^nx)^n\,\big|\,{\rm div}(w)\in \big(H^1(\Omega)\big)^*\big\} \to
H^{-1/2}(\partial\Omega)
=\big(H^{1/2}(\partial\Omega)\big)^* \\
\hspace*{5.85cm} w\mapsto \nu\cdot w  \end{cases}  \lb{A.11}
\end{equation}
by setting
\begin{equation}
\langle\nu\cdot
w,\phi\rangle=\int_{\Omega}d^nx\,w(x)\cdot\nabla\Phi(x) + \langle
{\rm div}(w)\,,\,\Phi\rangle \lb{A.11a}
\end{equation}
whenever $\phi\in H^{1/2}(\partial\Omega)$ and $\Phi\in
H^{1}(\Omega)$ is such that $\ga_D\Phi=\phi$. Here the pairing 
$\langle {\rm div}(w)\,,\,\Phi\rangle$ in (\ref{A.11a}) is the natural one 
between functionals in $\big(H^1(\Omega)\big)^*$ and elements in 
$H^1(\Omega)$ (which, in turn, is compatible with the (bilinear) distributional 
pairing). It should be remarked that the above definition is independent of the
particular extension $\Phi\in H^{1}(\Omega)$ of $\phi$. Indeed, by
linearity this comes down to proving that
\begin{equation}\label{ibp}
\langle {\rm div}(w)\,,\,\Phi\rangle
=-\int_{\Omega}d^nx\,w(x)\cdot\nabla\Phi(x)
\end{equation}
if $w\in L^2(\Omega;d^nx)^n$ has ${\rm div}(w)\in \big(H^1(\Omega)\big)^*$
and $\Phi\in H^{1}(\Omega)$ has $\ga_D\Phi=0$. To see this we
rely on the existence of a sequence $\Phi_j\in C^\infty_0(\Omega)$
such that $\Phi_j\underset{j\uparrow\infty}{\rightarrow} \Phi$ in
$H^{1}(\Omega)$. When $\Omega$ is a bounded Lipschitz domain, this
is well-known (see, e.g., \cite[Remark 2.7]{JK95} for a rather
general result of this nature), and this result is easily extended
to the case when $\Omega$ is an unbounded Lipschitz domain with a
compact boundary. Indeed, if $\xi\in C^\infty_0(B(0;2))$ is such that
$\xi= 1$ on $B(0;1)$ and $\xi_j(x)=\xi(x/j)$, $j\in\bbN$ (here
$B(x_0;r_0)$ denotes the ball in $\bbR^n$ centered at
$x_0\in\bbR^n$ of radius $r_0>0$), then
$\xi_j\Phi\underset{j\uparrow\infty}{\rightarrow}\Phi$ in
$H^{1}(\Omega)$ and matters are reduced to approximating
$\xi_j\Phi$ in $H^{1}(B(0;2j)\cap\Omega)$ with test functions
supported in $B(0;2j)\cap\Omega$, for each fixed $j\in\bbN$. Since
$\ga_D(\xi_j\Phi)=0$, the result for bounded Lipschitz domains
applies.

Returning to the task of proving (\ref{ibp}), it suffices to prove
a similar identity with $\Phi_j$ in place of $\Phi$. This, in
turn, follows from the definition of ${\rm div}(\cdot)$ in the
sense of distributions and the fact that the duality between
$\big(H^1(\Omega)\big)^*$ and $H^1(\Omega)$ is compatible with the duality
between distributions and test functions.

Going further, one can introduce a (weak) Neumann trace operator
$\wti\gamma_N$ as follows:
\begin{equation}\lb{A.16}
\wti\gamma_N\colon \big\{u\in H^{1}(\Om) \,\big|\, \Delta u \in
\big(H^1(\Om)\big)^*\big\}\to H^{-1/2}(\dOm),\quad \wti\gamma_N u=\nu\cdot
\nabla u,
\end{equation}
with the dot product  understood in the sense of \eqref{A.11}. We
emphasize that the weak Neumann trace operator $\wti\gamma_N$ in
\eqref{A.16} is a bounded extension of the operator $\gamma_N$ introduced
in \eqref{2.3}. Indeed, to see that
$\dom(\gamma_N)\subset\dom(\wti\gamma_N)$, we note that if $u\in
H^{s+1}(\Omega)$ for some $1/2<s<3/2$, then $\Delta u\in
H^{-1+s}(\Omega)=\bigl(H^{1-s}(\Omega)\bigr)^*\hookrightarrow
\bigl(H^{1}(\Omega)\bigr)^*$, by (\ref{dual-xxx}) and
(\ref{incl-xxx}). With this in hand, it is then easy to show that
$\wti\gamma_N$ in \eqref{domHN} and $\gamma_N$ in \eqref{2.3}
agree (on the smaller domain), as claimed.

We now return to the mainstream discussion. From the above
preamble it follows that
\begin{equation}
\dom(N^*)=\big\{w\in L^2(\Omega;d^nx)^n\,\big|\,{\rm div}(w)\in
L^2(\Omega;d^nx) \mbox{ and }\nu\cdot w=0\big\},
\end{equation}
where the dot product operation is understood in the sense of
\eqref{A.11}. Consequently, with $\wti H_{0,\Om}^N=N^*N$, we have
\begin{align}
\dom\big(\wti H_{0,\Om}^N\big) & = \{u\in\dom(N) \,|\,Nu\in\dom(N^*)\} \no
\\
& =\big\{u\in H^{1}(\Om) \,\big|\, \Delta u\in L^{2}(\Om;d^nx)\mbox{ and }
\wti\gamma_N u=0\big\}. \lb{domHN}
\end{align}

Next, we intend to recall that $H_{0,\Om}^D = \wti H_{0,\Om}^D$ and
$H_{0,\Om}^N = \wti H_{0,\Om}^N$, where $H_{0,\Om}^D$ and
$H_{0,\Om}^N$ denote the operators introduced in \eqref{2.4} and
\eqref{2.5}, respectively. For this purpose one can argue as follows:
Since it follows from the first Green's
formula (cf., e.g., \cite[Theorem 4.4]{Mc00}) that $H_{0,\Om}^D
\subseteq \wti H_{0,\Om}^D$ and $H_{0,\Om}^N \subseteq \wti
H_{0,\Om}^N$, it remains to show that $H_{0,\Om}^D \supseteq \wti
H_{0,\Om}^D$ and $H_{0,\Om}^N \supseteq \wti H_{0,\Om}^N$.
Moreover, it follows from comparing \eqref{2.4} with \eqref{domHD}
and \eqref{2.5} with \eqref{domHN}, that one needs only to show
that $\dom\big(\wti H_{0,\Om}^D\big)$,
$\dom\big(\wti H_{0,\Om}^N\big)\subseteq H^{2}(\Omega)$.
This is the content of the next lemma.

\begin{lemma} \lb{lA.1}
Assume Hypothesis \ref{h2.1}. Then,
\begin{equation}
\dom\big(\wti H_{0,\Om}^D\big)\subset H^{2}(\Omega), \quad \dom\big(\wti
H_{0,\Om}^N\big)\subset H^{2}(\Omega).   \lb{ADNH}
\end{equation}
Moreover,
\begin{equation}
H_{0,\Om}^D = \wti H_{0,\Om}^D, \quad H_{0,\Om}^N = \wti
H_{0,\Om}^N.   \lb{DOM}
\end{equation}
\end{lemma}

For $C^{1,r}$-domains $\Om$, $1/2 <r<1$, Lemma \ref{lA.1} was proved in
\cite[Appendix A]{GLMZ05}. For bounded convex domains $\Om$,
$\dom\big(\wti H_{0,\Om}^D\big)\subset H^2(\Om)$ was shown by Kadlec \cite{Ka64}
and Talenti \cite{Ta65} and $\dom\big(\wti H_{0,\Om}^N\big)\subset H^2(\Om)$ was
proved by Grisvard and Ioss \cite{GI75}. A unified approach to Dirichlet and Neumann problems in bounded convex domains, which also applies to bounded Lipschitz domains satisfying UEBC, has been presented by Mitrea \cite{Mi01}.
The extension to domains $\Omega$ with a compact boundary satisfying UEBC then follows as described in the paragraph following
\eqref{dual-xxx}. This establishes \eqref{ADNH} and hence \eqref{DOM} as discussed after \eqref{domHN}.

We note that Lemma \ref{lA.1} also follows from \cite[Theorem
8.2]{DHP03} in the case of $C^2$-domains $\Om$ with compact
boundary. This is proved in \cite{DHP03} by rather different
methods and can be viewed as a generalization of the classical
result for bounded $C^2$-domains.

As shown in \cite[Lemma A.2]{GLMZ05}, \eqref{ADNH} and methods of real interpolation
spaces yield the following key result \eqref{new6.45} needed in the main body of this paper:

\begin{lemma}\label{lA.2}
Assume Hypothesis \ref{h2.1} and let $q\in[0,1]$. Then for each
$z\in\bbC\backslash[0,\infty)$, one has
\begin{equation}\label{new6.45}
\big(H_{0,\Om}^D-zI_{\Om}\big)^{-q},\,
\big(H_{0,\Om}^N-zI_{\Om}\big)^{-q}\in\cB\big(\LOm,H^{2q}(\Om)\big).
\end{equation}
\end{lemma}

Finally, we recall an extension of a result of Nakamura
\cite[Lemma 6]{Na01} from a cube in $\bbR^n$ to a Lipschitz domain
$\Om$. This requires some preparation. First, we note that
(\ref{A.16}) and (\ref{A.11a}) yield the following Green formula
\begin{equation}
\langle\wti\gamma_N u,\ga_D\Phi\rangle = \big(\ol{\nabla u},
\nabla \Phi\big)_{\LOm^n} + \langle\Delta u,\Phi\rangle,
\lb{wGreen}
\end{equation}
valid for any $u\in H^{1}(\Om)$ with $\Delta
u\in\big(H^{1}(\Om)\big)^*$, and any $\Phi\in H^{1}(\Om)$. The
pairing on the left-hand side of (\ref{wGreen}) is between
functionals in $\big(H^{1/2}(\dOm)\big)^*$ and elements in
$H^{1/2}(\dOm)$, whereas the last pairing on the right-hand side
is between functionals in $\big(H^{1}(\Om)\big)^*$ and elements in
$H^{1}(\Om)$. For further use, we also note that the adjoint of
\eqref{2.2} maps boundedly as follows
\begin{equation}
\ga_D^* : \big(H^{s-1/2}(\dOm)\big)^* \to (H^{s}(\Om)\big)^*,
\quad 1/2<s<3/2. \lb{ga*}
\end{equation}
Next, one observes that the operator $\big(\wti
H^N_{0,\Om}-zI_\Om\big)^{-1}$, $z\in\bbC\big\backslash\si\big(\wti
H^N_{0,\Om}\big)$, originally defined as
\begin{equation}\label{fukcH}
\big(\wti H^N_{0,\Om}-zI_\Om\big)^{-1}:\LOm\to \LOm,
\end{equation}
can be extended to a bounded operator, mapping
$\big(H^{1}(\Om)\big)^*$ into $L^2(\Omega;d^nx)$. Specifically,
since $\big(\wti H^N_{0,\Om}-\ol{z}I_\Om\big)^{-1}: \LOm \to \dom\big(\wti
H^N_{0,\Om}\big)$ is bounded and since the inclusion $\dom\big(\wti
H^N_{0,\Om}\big)\hookrightarrow H^{1}(\Om)$ is bounded, we can
naturally view $\big(\wti H^N_{0,\Om}-\ol{z}I_\Om\big)^{-1}$ as an
operator
\begin{equation}
\big(\hatt H^N_{0,\Om}-\ol{z}I_\Om\big)^{-1} \colon \LOm \to H^1(\Om)
\end{equation}
mapping in a linear, bounded fashion. Consequently, for its
adjoint, we have
\begin{equation}\label{fukcH-bis}
\big(\big(\hatt H^N_{0,\Om}-\ol{z}I_\Om\big)^{-1}\big)^* :
\big(H^{1}(\Om)\big)^* \to \LOm,
\end{equation}
and it is easy to see that this latter operator extends the one in
\eqref{fukcH}. Hence, there is no ambiguity in retaining the same
symbol, that is, $\big(\wti H^N_{0,\Om}-\ol{z}I_\Om\big)^{-1}$, both for
the operator in (\ref{fukcH-bis}) as well as for the operator in
(\ref{fukcH}). Similar considerations and conventions apply to
$\big(\wti H^D_{0,\Om}-zI_\Om\big)^{-1}$.

Given these preparations, we now state without proof (and for the convenience of the reader) the following result proven in \cite[Lemma A.3]{GLMZ05} (an extension of a result 
proven in \cite{Na01}).

\begin{lemma} \lb{lA.3}
Let $\Om\subset\bbR^n$, $n\geq2$, be an open Lipschitz domain and let
$z\in\bbC\big\backslash\big(\si\big(\wti H^D_{0,\Om}\big)\cup\si\big(\wti
H^N_{0,\Om}\big)\big)$. Then, on $L^2(\Omega;d^nx)$,
\begin{equation} \lb{Na1}
\big(\wti H^D_{0,\Om}-zI_\Om\big)^{-1} - \big(\wti H^N_{0,\Om}-zI_\Om\big)^{-1} =
\big(\wti H^N_{0,\Om}-zI_\Om\big)^{-1}\ga_D^*\wti\gamma_N \big(\wti
H^D_{0,\Om}-zI_\Om\big)^{-1},
\end{equation}
where $\ga_D^*$ is an adjoint operator to $\ga_D$ in the sense of
\eqref{ga*}
\end{lemma}

\begin{remark} \lb{rA.4}
While it is tempting to view $\ga_D$ as an unbounded but densely
defined operator on $\LOm$ whose domain contains the space
$C_0^\infty(\Om)$, one should note that in this case its adjoint
$\ga_D^*$ is not densely defined: Indeed (cf.\ \cite[Remark A.4]{GLMZ05}),
 $\dom(\gamma_D^*)=\{0\}$ and hence $\gamma_D$
is not a closable linear operator in $\LOm$.
\end{remark}

\begin{remark} \lb{rA.5}
In the case of domains $\Om$ satisfying Hypothesis \ref{h2.1}, Lemma
\ref{lA.1} implies that the
operators $\wti H^D_{0,\Om}$ and $\wti H^N_{0,\Om}$ coincide with
the operators $H^D_{0,\Om}$ and $H^N_{0,\Om}$, respectively, and
hence one can use the operators $H^D_{0,\Om}$ and $H^N_{0,\Om}$ in
Lemma \ref{lA.3}. Moreover, since $\dom\big(H^D_{0,\Om}\big) \subset
H^2(\Om)$, one can also replace $\wti\gamma_N$ by $\ga_N$ (cf.\
\eqref{2.3}) in Lemma \ref{lA.3}. In particular,
\begin{align}
\begin{split}
&\big(H^D_{0,\Om}-zI_\Om\big)^{-1} - \big(H^N_{0,\Om}-zI_\Om\big)^{-1}
=\big[\ga_D \big(H^N_{0,\Om}-\ol{z}I_\Om\big)^{-1}\big]^*\gamma_N
\big(H^D_{0,\Om}-zI_\Om\big)^{-1},   \lb{Na1-bis}  \\
& \hspace*{7.6cm} z\in\bbC\big\backslash\big(\si\big(H^D_{0,\Om}\big)
\cup\si\big(H^N_{0,\Om}\big)\big),
\end{split}
\end{align}
a result exploited in the proof of Theorem \ref{t4.1} (cf.\ \eqref{4.13}).
\end{remark}

Finally, we prove the following result used in the proof of Lemma \ref{l3.4}.

\begin{lemma} \lb{lA.6}
Suppose $\Om\subset\bbR^n$, $n\geq 2$, is an open 
Lipschitz domain with a compact, nonempty boundary $\dOm$. Then
the Dirichlet trace operator $\ga_D$ satisfies the following  
property $($see also \eqref{2.2}$)$,
\begin{equation}
\ga_D\in \cB\big(H^{(3/2)+\eps}(\Om), H^{1}(\dOm)\big), \quad \eps>0.
\lb{A.62}
\end{equation}
\end{lemma}
\begin{proof}
First, we recall one of the equivalent definitions of $H^{1}(\dOm)$, specifically, 
\begin{equation}
H^{1}(\dOm) = \big\{f\in\LdOm \,\big|\, \partial
f/\partial\tau_{j,k}\in\LdOm, \; j,k = 1,\dots,n\big\},
\end{equation}
where
$\partial/\partial\tau_{k,j}=\nu_k\partial_j-\nu_j\partial_k$,
$j,k=1,\dots,n$, is a tangential derivative operator (cf.\
\eqref{A.31}), or equivalently,
\begin{align}
H^{1}(\dOm) &= \bigg\{f\in\LdOm \,\bigg|\, \, \text{there exists a constant $c>0$
such that for every $v\in C_0^\infty(\bbR^n)$,}  \no \\
&\hspace*{28mm} \bigg|\int_\dOm d^{n-1}\si f\,\partial
v/\partial\tau_{j,k}\bigg| \leq c \norm{v}_{\LdOm},\;
j,k=1,\dots,n\bigg\}. \lb{A.64}
\end{align}

Next, let $u\in H^{(3/2)+\eps}(\Om)$, $v\in C_0^\infty(\bbR^n)$,
and $u_i\in C^\infty(\ol{\Om})\hookrightarrow
H^{(3/2)+\eps}(\Om)$, $i\in\bbN$, be a sequence of functions
approximating $u$ in $H^{(3/2)+\eps}(\Om)$. It follows from
\eqref{2.2} and \eqref{incl-xxx} that $\ga_D u, \ga_D(\nabla
u)\in\LdOm$. Introducing the tangential derivative operator
$\partial/\partial\tau_{k,j}=\nu_k\partial_j-\nu_j\partial_k$,
$j,k=1,\dots,n$, one has
\begin{equation} \lb{A.31}
\int_{\partial\Omega} d^{n-1}\sigma \, \frac{\partial
h_1}{\partial\tau_{j,k}}h_2 =-\int_{\partial\Omega}  d^{n-1}\sigma
\, h_1\frac{\partial h_2}{\partial\tau_{j,k}}, \quad h_1, h_2\in
H^{1/2}(\partial\Omega).
\end{equation}
Utilizing \eqref{A.31}, one computes for all $j,k=1,\dots,n$,
\begin{align}
\bigg|\int_\dOm d^{n-1}\si\, \ga_D u \frac{\partial
v}{\partial\tau_{j,k}} \bigg| &= \bigg|\lim_{i\to\infty} \int_\dOm
d^{n-1}\si\, u_i \frac{\partial v}{\partial\tau_{j,k}}\bigg| =
\bigg|\lim_{i\to\infty} \int_\dOm d^{n-1}\si\, v \frac{\partial
u_i}{\partial\tau_{j,k}}\bigg| \lb{A.65}
\\ &\leq
c \, \bigg|\lim_{i\to\infty} \int_\dOm d^{n-1}\si\,v\, \ga_D(\nabla
u_i)\bigg| \leq c\norm{\ga_D(\nabla u)}_{\LdOm} \norm{v}_{\LdOm}.
\no
\end{align}
Thus, it follows from \eqref{A.64} and \eqref{A.65} that $\ga_D
u \in H^1(\dOm)$.
\end{proof}

\section{Abstract Perturbation Theory}
\lb{sB}
\renewcommand{\theequation}{B.\arabic{equation}}
\renewcommand{\thetheorem}{B.\arabic{theorem}}
\setcounter{theorem}{0} \setcounter{equation}{0}

The purpose of this appendix is to summarize some of the abstract perturbation results  in \cite{GLMZ05} which where motivated by Kato's pioneering work \cite{Ka66} (see also \cite{Ho70}, \cite{KK66}) as they are needed in this paper.

We introduce the following set of assumptions.

\begin{hypothesis} \lb{hB.1}
Let $\cH$ and $\cK$ be separable, complex Hilbert spaces. \\
$(i)$ Suppose that $H_0\colon\dom(H_0)\to\cH$,
$\dom(H_0)\subseteq\cH$ is a densely defined, closed, linear
operator in $\cH$ with nonempty resolvent set,
\begin{equation}
\rho(H_0)\neq\emptyset,
\end{equation}
$A\colon\dom(A)\to\cK$, $\dom(A)\subseteq\cH$ a densely defined,
closed, linear operator from $\cH$ to $\cK$, and
$B\colon\dom(B)\to\cK$, $\dom(B)\subseteq\cH$ a densely defined,
closed, linear operator from $\cH$ to $\cK$ such that
\begin{equation}
\dom(A)\supseteq\dom(H_0), \quad \dom(B)\supseteq\dom(H_0^*).
\end{equation}
In the following we denote
\begin{equation}
R_0(z)=(H_0-zI_{\cH})^{-1}, \quad z\in \rho(H_0).
\end{equation}
$(ii)$ Assume that for some $($and hence for all\,$)$ $z\in\rho(H_0)$, the
operator $-AR_0(z)B^*$, defined on $\dom(B^*)$, has a bounded
extension in $\cK$, denoted by $K(z)$,
\begin{equation}
K(z)=-\ol{AR_0(z)B^*} \in\cB(\cK). \lb{B.4}
\end{equation}
$(iii)$ Suppose that $1\in\rho(K(z_0))$ for some $z_0\in \rho(H_0)$. \\
$(iv)$ Assume that $K(z)\in\cB_\infty(\cK)$ for all $z\in\rho(H_0)$.
\end{hypothesis}

Next, following Kato \cite{Ka66}, one introduces
\begin{equation}
R(z)=R_0(z)-\ol{R_0(z)B^*}[I_{\cK}-K(z)]^{-1}AR_0(z),
\quad  z\in\{\zeta\in\rho(H_0) \,|\, 1\in\rho(K(\zeta))\}.  \lb{B.5}
\end{equation}

\begin{theorem} [\cite{GLMZ05}] \lb{tB.2}
Assume Hypothesis \ref{hB.1}\,$(i)$--$(iii)$ and suppose
$z\in\{\zeta\in\rho(H_0)\,|\, 1\in\rho(K(\zeta))\}$. Then,
$R(z)$ introduced in \eqref{B.5} defines a densely defined,
closed, linear operator $H$ in $\cH$ by
\begin{equation}
R(z)=(H-zI_{\cH})^{-1}.
\end{equation}
In addition,
\begin{equation}
AR(z), BR(z)^* \in \cB(\cH,\cK) \lb{B.7}
\end{equation}
and
\begin{align}
R(z)&=R_0(z)-\ol{R(z)B^*}AR_0(z) \lb{B.8} \\
&=R_0(z)-\ol{R_0(z)B^*}AR(z).    \lb{B.9}
\end{align}
Moreover, $H$ is an extension of
$(H_0+B^*A)|_{\dom(H_0)\cap\dom(B^*A)}$ $($the latter
intersection domain may consist of $\{0\}$ only$)$,
\begin{equation}
H\supseteq (H_0+B^*A)|_{\dom(H_0)\cap\dom(B^*A)}.
\end{equation}
Finally, assume that $H_0$ is self-adjoint in $\cH$. Then $H$ is also
self-adjoint if
\begin{equation}
(Af,Bg)_{\cK}=(Bf,Ag)_{\cK} \, \text{ for all } \,
f,g\in\dom(A)\cap\dom(B). \lb{B.11}
\end{equation}
\end{theorem}

In the case where $H_0$ is self-adjoint, Theorem \ref{tB.2} is due to Kato
\cite{Ka66} in this abstract setting.

\medskip

The next result is an abstract version of the celebrated Birman--Schwinger
principle relating eigenvalues $\lambda_0$ of $H$ and the eigenvalue $1$
of $K(\lambda_0)$:

\begin{theorem}  [\cite{GLMZ05}] \lb{tB.3}
Assume Hypothesis \ref{hB.1} and let $\lambda_0\in\rho(H_0)$.
Then,
\begin{equation}
Hf=\lambda_0 f, \quad 0\neq f\in\dom(H) \, \text{ implies } \,
K(\lambda_0)g=g
\end{equation}
where, for fixed $z_0\in\{\zeta\in\rho(H_0)\,|\,
1\in\rho(K(\zeta))\}$, $z_0\neq \lambda_0$,
\begin{equation}
0 \neq g = [I_{\cK}-K(z_0)]^{-1}AR_0(z_0)f =
(\lambda_0-z_0)^{-1}Af.
\end{equation}
Conversely,
\begin{equation}
K(\lambda_0)g=g, \quad 0\neq g\in\cK \, \text{ implies } \,
Hf=\lambda_0 f,
\end{equation}
where
\begin{equation}
0\neq f=-\ol{R_0(\lambda_0)B^*}g\in\dom(H).
\end{equation}
Moreover,
\begin{equation}
\dim(\ker(H-\lambda_0I_{\cH}))
=\dim(\ker(I_{\cK}-K(\lambda_0)))<\infty.
\end{equation}
In particular, let $z\in\rho(H_0)$, then
\begin{equation}
\text{$z\in\rho(H)$ if and only if\; $1\in\rho(K(z))$.}
\end{equation}
\end{theorem}

In the case where $H_0$ and $H$ are self-adjoint, Theorem \ref{tB.3}
is due to Konno and Kuroda \cite{KK66}.

\medskip

\noindent {\bf Acknowledgments.}
We are indebted to Yuri Latushkin, Konstantin A.\ Makarov, and Anna Skripka
for helpful discussions on this topic. We also thank Tonci Crmaric for pointing out
some references to us and we are grateful to Dorina Mitrea for a critical reading of the  
manuscript. 


\end{document}